\documentclass[10pt]{amsart}

\usepackage[T1]{fontenc} 
\usepackage[utf8]{inputenc} 
\usepackage{longtable}
\usepackage{amssymb}
\usepackage{amsthm}
\usepackage{amsmath}
\usepackage{wasysym}
\usepackage{dcpic, pictex}
\usepackage{bbm}
\usepackage{enumerate}
\usepackage{amsfonts}
\usepackage{amsthm}
\usepackage{amsmath}
\usepackage{times}
\usepackage{amscd}
\usepackage[mathscr]{eucal}
\usepackage{indentfirst}
\usepackage{pict2e}
\usepackage{epic}
\usepackage{epstopdf} 
\usepackage{verbatim}
\usepackage{cancel}
\usepackage[foot]{amsaddr}
\usepackage{lipsum}
\usepackage{mathrsfs}
\usepackage{bm}

\usepackage{graphicx}

\usepackage{
	amsfonts, 
	latexsym, 
	enumerate,
}
\usepackage[english]{babel}
\allowdisplaybreaks
\makeindex
\newtheorem{theorem}{Theorem}[section]
\newtheorem{lemma}[theorem]{Lemma}
\newtheorem{proposition}[theorem]{Proposition}
\newtheorem{definition}[theorem]{Definition}
\newtheorem{example}[theorem]{Example}

\newtheorem{remark}[theorem]{Remark}

\newtheorem{assumption}[theorem]{Assumption}
\newtheorem{setting}[theorem]{Setting}

\usepackage[a4paper,top=3.5cm,bottom=3.5cm,left=3cm,right=3cm]{geometry}
\numberwithin{equation}{section}
\allowdisplaybreaks

\usepackage{xcolor}

\newcommand{\nc}{\normalcolor}
\newcommand{\dif}{\mathrm{d}}
\newcommand{\E}{\mathbf{E}}
\newcommand{\R}{\mathbf{R}}
\newcommand{\C}{\mathbf{C}}
\newcommand{\HH}{\mathbf{H}}

\newcommand{\N}{\mathbf{N}}

\newcommand{\LT}{\mathrm{LT}}

\newcommand{\ii}{\mathrm{i}}
\newcommand{\ee}{\mathrm{e}}

\newcommand{\pair}{\nu}

\usepackage{hyperref}

\title[Eigenvector decorrelation for random matrices]{Eigenvector decorrelation for random matrices}

\date{\today}

\begin{document}
	
	\maketitle
	
	\vspace{0.25cm}
	
	\renewcommand{\thefootnote}{\fnsymbol{footnote}}

	\noindent
	\mbox{}%
	\hfill% 
	\begin{minipage}{0.21\textwidth}
		\centering
		{Giorgio Cipolloni}\footnotemark[1]\\
		\scriptsize{\textit{cipolloni@axp.mat.uniroma2.it}}
	\end{minipage}
	\hfill%
	\begin{minipage}{0.21\textwidth}
		\centering
		{L\'aszl\'o Erd\H{o}s}\footnotemark[2]\\
		\scriptsize{\textit{lerdos@ist.ac.at}}
	\end{minipage}
	\hfill%
	\begin{minipage}{0.21\textwidth}
		\centering
		{Joscha Henheik}\footnotemark[3]\\
		\scriptsize{\textit{joscha.henheik@unige.ch}}
	\end{minipage}
	\hfill%
	\begin{minipage}{0.21\textwidth}
		\centering
		{Oleksii Kolupaiev}\footnotemark[2]\\
		\scriptsize{\textit{okolupaiev@ist.ac.at}}
	\end{minipage}
	\hfill%
	\mbox{}%
	\footnotetext[1]{Department of Mathematics, University of Rome Tor Vergata, Via della Ricerca Scientifica, 1, 00133 Roma RM, Italy.}
	\footnotetext[2]{Institute of Science and Technology Austria, Am Campus 1, 3400 Klosterneuburg, Austria.}
\footnotetext[3]{Section of Mathematics, University of Geneva, Rue du Conseil Général 7-9, 1205 Geneva, Switzerland.}
	\renewcommand*{\thefootnote}{\arabic{footnote}}
	\vspace{0.25cm}
	
	\begin{abstract} 
		
		We study the sensitivity of the eigenvectors of random matrices, showing that even small perturbations make the eigenvectors almost orthogonal. More precisely, we consider two deformed Wigner matrices $W+D_1$, $W+D_2$ and show that their bulk eigenvectors become asymptotically orthogonal as soon as $\mathrm{Tr}(D_1-D_2)^2\gg 1$, or their respective energies are separated on a scale much bigger than the local eigenvalue spacing. Furthermore, we show that quadratic forms of eigenvectors of $W+D_1$, $W+D_2$ with any deterministic matrix $A\in\C^{N\times N}$ in a specific subspace of codimension one are of size $N^{-1/2}$. This proves a generalization of the Eigenstate Thermalization Hypothesis to eigenvectors belonging to two different spectral families.
	\end{abstract}
	\vspace{0.15cm}
	
	
	\vspace{0.25cm}
	\normalsize

\section{Introduction} \label{sec:intro}

\subsection{Main result}  The behavior of the eigenvectors of a Hermitian matrix under perturbations is known to be quite subtle: even a small change in the matrix may lead to a significant rotation of the eigenvectors  due to resonances.  This
phenomenon is ubiquitous in a broad range of numerical and statistical applications, see, e.g., \cite{Chen, Davis, DPillis, Fan, Ng, Yoon}.  For example, in 
the classical paper \cite{Davis} Davis and Kahan  give a deterministic upper bound for the deviation of the eigenvectors from the unperturbed ones in terms of the spectral gap, while in \cite[Theorem 2]{Ng} the authors show that there always exists a perturbation causing a big change in the eigenvectors. 
When the magnitude of the perturbation exceeds the local eigenvalue spacing of the initial matrix, standard perturbation 
theory does not control the eigenbasis of the perturbed matrix any more and
the behavior of the eigenvectors is highly sensitive to the properties of the original matrix. While in some  rare cases 
even such  larger  perturbations still cause only a small change, 
typically the perturbed eigenbasis is completely decoupled from the initial one.
 In this paper, we show that indeed this typical scenario occurs for random matrices with very high probability.

More precisely, we consider two \emph{deformed Wigner matrices} of the form $H_1=W+D_1$, $H_2=W+D_2$, where $W$ is a \emph{Wigner matrix}\footnote{A Wigner matrix is a Hermitian $N \times N$ matrix $W=W^*$ with independent, identically distributed centered entries (up to the Hermitian symmetry) with $\E|W_{ab}|^2=1/N$; see also Assumption~\ref{ass:momass}.} and $D_1, D_2$ are Hermitian deterministic \emph{deformations}, which we assume to be traceless without loss of generality.
 Denote the eigenvalues (\emph{energies}) of $H_l$ in increasing order\footnote{The upper index $l=1,2$ of the eigenvalues $\lambda$ and other related quantities should not be confused with a power.} by $\lambda_1^l\le \lambda_2^l\le\ldots\le \lambda_N^l$, $l=1,2$, and let ${\boldsymbol u}_1^l, {\boldsymbol u}_2^l,\ldots, {\boldsymbol u}_N^l$ be the corresponding orthonormal eigenvectors. We measure the distance between the families $\lbrace {\boldsymbol u}^1_i\rbrace_{i=1}^N$ and $\lbrace {\boldsymbol u}^2_j\rbrace_{j=1}^N$ by looking at the \emph{eigenvector overlaps} $\langle \boldsymbol{u}_i^1, A\boldsymbol{u}_j^2\rangle$ for a deterministic observable matrix $A$. 
 
 Our first main result (Theorem \ref{theo:maintheo2}) is the decomposition
 \begin{equation}
 \label{eq:decorrintro}
 \langle \boldsymbol{u}_i^1, A\boldsymbol{u}_j^2\rangle=\langle VA\rangle \langle \boldsymbol{u}_i^1,\boldsymbol{u}_j^2\rangle+\mathcal{O}\left(\frac{\lVert A\rVert}{\sqrt{N}}\right),
 \end{equation}
for bulk indices\footnote{We say that an index $i$ is in the bulk of the spectrum if the density of states around $\lambda_i^l$ is strictly positive; see \eqref{eq:kappabulk} for the precise definition.} $i,j$, where
$\langle X \rangle: = \frac{1}{N}{\mathrm Tr}\, X$ denotes the averaged trace of $X\in\C^{N\times N}$. Here
 $V$ is an appropriately chosen deterministic matrix depending on the deformations $D_1,D_2$ and the (typical locations of the) energies $\lambda_i^1, \lambda_j^2$ with $\lVert V\rVert \lesssim 1$ (see \eqref{eq:defV} for its definition). The $N^{-1/2}$ error term in \eqref{eq:decorrintro} is optimal.

As our second main result (Theorem \ref{theo:maintheo}), we give an upper bound on the overlap $\langle \boldsymbol{u}_i^1,\boldsymbol{u}_j^2\rangle$ in \eqref{eq:decorrintro}. In the special case $D_1=D_2$ we trivially have $\langle \boldsymbol{u}_i^1,\boldsymbol{u}_j^2\rangle=\delta_{ij}$, hence \eqref{eq:decorrintro} is just the Eigenstate Thermalization Hypothesis (ETH) for deformed Wigner matrices proven in \cite{equipart}. However, in general, when we consider two different deformations, the overlap $\langle \boldsymbol{u}_i^1,\boldsymbol{u}_j^2\rangle$ is non-trivial. In fact, it subtly depends on two effects; the difference in deformations, $D_1 - D_2$, and the difference in energy $\lambda_i^1 - \lambda_j^2$.   
In order to study the decorrelation properties of $\langle \boldsymbol{u}_i^1, A\boldsymbol{u}_j^2\rangle$ we thus need to give an estimate on this eigenvector overlap in terms of these two differences. More precisely, in Theorem \ref{theo:maintheo}, we prove the optimal bound
\begin{equation}
\label{eq:decestort}
\big|\langle \boldsymbol{u}_i^1,\boldsymbol{u}_j^2\rangle\big|^2\lesssim \frac{1}{N}\cdot \frac{1}{\langle (D_1-D_2)^2\rangle+\mathrm{LT}+|\lambda_i^1 - \lambda_j^{2}|^2},
\end{equation}
where the so-called \emph{linear term} $\mathrm{LT}$ is (the absolute value of) a specific linear combination of $D_1 - D_2$ and $\lambda_i^1 - \lambda_j^2$ and its precise definition 
will be given in \eqref{eq:linterm}. The estimate \eqref{eq:decestort} manifests the interplay of the two 
decay effects in three different terms, which can make the eigenvectors $\boldsymbol{u}_i^1, \boldsymbol{u}_j^2$ almost orthogonal. 
The identification of the decay in  $D_1-D_2$ is the main new result in this paper. It captures the effect that the spectral
resolutions of $W+D_1, W+D_2$ become more and more independent as $\langle (D_1-D_2)^2\rangle$ grows. 
We describe the relation between the three terms in \eqref{eq:decestort} in more details below Theorem~\ref{theo:maintheo}. Here we only comment on  the optimality of our proven 
decay in terms of  $\langle (D_1-D_2)^2\rangle$. Standard second order perturbation theory (outlined in Remark \ref{rem:perturb}) indicates
that  $\langle \boldsymbol{u}_i^1,\boldsymbol{u}_i^2\rangle \approx 1$ 
in the regime $\langle (D_1-D_2)^2\rangle\ll 1/N$. Our bound \eqref{eq:decestort} shows that $
\langle \boldsymbol{u}_i^1,\boldsymbol{u}_i^2\rangle \approx 0$ in the opposite regime $\langle (D_1-D_2)^2\rangle\gg 1/N$.

Putting together our two main results, \eqref{eq:decorrintro} and \eqref{eq:decestort}, we see that the overlap $ \langle \boldsymbol{u}_i^1, A\boldsymbol{u}_j^2\rangle$ can be small on two different grounds: Either the observable matrix $A$ is (nearly) orthogonal to $V$, i.e.~$\langle V A \rangle \approx  0$, or the overlap $ \langle \boldsymbol{u}_i^1, \boldsymbol{u}_j^2\rangle$ is small as estimated in \eqref{eq:decestort}. We coin the first the \emph{regularity effect} and the second the \emph{overlap decay effect}. All results hold with very large probability.

\subsection{Previous  related results} To put our results \eqref{eq:decorrintro}--\eqref{eq:decestort} into context we now describe several related results, which partially explored only one of the two smallness effects at a time. We stress that 
our results \eqref{eq:decorrintro}--\eqref{eq:decestort} manage to catch both these effects in a unified and optimal manner. In fact, prior to this work, the regularity effect \eqref{eq:decorrintro} was only studied in the context of the same matrix $H$, i.e. $D_1=D_2$ (and possibly both equal to zero), to prove the ETH in the setting of random matrices. The ETH, posed by Deutsch in \cite{Deutsch} as a signature of chaos in quantum systems, states that quadratic forms of eigenfunctions of chaotic Hamiltonians can be described purely by macroscopic quantitites and that the (pseudo-random) fluctuations are entropically suppressed. In the context of a \emph{single} random matrix ensemble the ETH reads as
\begin{equation}
\label{eq:ETH}
\langle \boldsymbol{u}_i, A\boldsymbol{u}_j\rangle=\langle VA\rangle \delta_{ij}+\mathcal{O}\left(\frac{\lVert A\rVert}{\sqrt{N}}\right),
\end{equation}
where $\boldsymbol{u}_i$ are the orthonormal eigenvectors of an $N\times N$ random matrix $H$. The ETH in the form \eqref{eq:ETH} was first proven for Wigner matrices (i.e.~$D_1=D_2=0$, in which case $V = I$) in \cite{ETH} (see also \cite{momentflow, RBM} for previous partial results). We point out that even the Gaussianity of the fluctuations in the $N^{-1/2}$-term is
 known for Wigner matrices for special observables \cite{momentflow, RBM}, for general observables \cite{fluctETH, Berryconj, fluctETHedge}, and for deformed Wigner matrices \cite{equipart}. The result \eqref{eq:ETH} was extended in several directions: to more general random matrix ensembles \cite{equipart, ETHgen, quantumchaos, ETHW-type}, where $V$ becomes energy dependent, to $d$--regular graphs \cite{reg_graph_sc, reg_graph} and to improvement of the error term in \eqref{eq:ETH} from $\lVert A\rVert$ to $\langle A^2\rangle^{1/2}$ \cite{fluctETHgen, rank_loc_law, edgeETH}. Related to \eqref{eq:ETH}, we also mention that in the past few years there has been great interest in studying eigenvector overlaps of different nature in several other contexts, including tensor principal component analysis (PCA) \cite{PCA}, shrinkage estimators \cite{shrink_est, nonlin_shrink}, noise detection \cite{noise_gauss, noise}, minors \cite{minors}, the equipartition principle \cite{equipart_init}, and many body physics \cite{manybody}.

On the other hand, estimates of the form \eqref{eq:decestort}, focusing on the $D_1 - D_2$ behavior, were 
previously studied only for Hermitian matrices in the very special case when $D_i=x_iD$, for a scalar $x_i$, in \cite{quench}, and in the context of decorrelation estimates for the Hermitization of non--Hermitian matrices in \cite{CLTnon-herm, CLTnon-hermreal, mesoCLT, gumbel}, where the deformation has a very special  $2\times 2$ block structure
 with zero diagonal blocks and off--diagonal blocks being constant multiples of the identity.
 As a related problem, sensitivity of the top eigenvector for a Wigner matrix to resampling of a small portion of 
 the matrix elements was studied in \cite{Bordenave} and extended to sparse matrices 
 in \cite{BordenaveLee}. \nc

\subsection{Multi--resolvent local laws} Local laws in general are concentration estimates for a single resolvent $G$ of a random matrix, or alternating chains of resolvents and deterministic matrices $A$, i.e. $GAGAGA...$. 
The main technical tool that we use to prove the decorrelation estimates for eigenvectors in \eqref{eq:decorrintro}--\eqref{eq:decestort} is a \emph{two-resolvent local law}, which is stated in Theorem \ref{theo:multigllaw} below. 

We now first describe our new multi-resolvent local law and then relate it to previous results. Let us denote the resolvent of $W+D_j$ at $z_j\in\C\setminus\R$ by $G_j:=(W+D_j-z_j)^{-1}$ and let $A$ be a deterministic $N\times N$ matrix. Then our new multi--resolvent local law asserts that, as $N$ tends to infinity, the matrix product $G_1AG_2$ concentrates around its deterministic approximation, denoted by $M_{12}^{A}$, which is explicitly given by\footnote{Note that $G_1AG_2$ is \emph{not} close to $M_1 AM_2$, indicating that multi-resolvent local laws are not simple consequences of the single resolvent local law. }
\[
M_{12}^A = M_1 A M_2 + \frac{\langle M_1 A M_2 \rangle }{1 - \langle M_1 M_2 \rangle} M_1 M_2.
\]
Here $M_i$ denotes the deterministic approximation of the single resolvent $G_i$ obtained as the unique solution $M_i = M^{D_i}(z_i)$ to the \emph{Matrix Dyson Equation} (MDE)
\begin{equation} \label{eq:MDE}
-M_i^{-1} = z_i - D_i + \langle M_i \rangle
\end{equation}
under the constraint $\Im M_i \, \Im z_i > 0$. 
 We optimally control the fluctuation of $G_1AG_2$ around $M_{12}^{A}$ in terms of $D_1-D_2$ and $z_1-z_2$, showing that typically the size of the fluctuation around $M_{12}^{A}$ is smaller than the size of $M_{12}^{A}$ itself. For this reason, in Proposition \ref{prop:stab} below, we give the following bound on $M_{12}^{A}$:
\begin{equation}
\lVert M_{12}^{A}\rVert\lesssim \frac{1}{\gamma}, \qquad\quad \gamma : = \langle (D_1-D_2)^2\rangle+\vert \Re z_1-\Re z_2\vert^2+\LT +|\Im z_1| + |\Im z_2| ,
\label{eq:stab_intro}
\end{equation}
where the $\LT\ge 0$ behaves as (the absolute value of) a linear combination of $D_1-D_2$ and $z_1-z_2$ (a precise definition will be given in \eqref{eq:linterm} later). The interesting regime is when $\gamma \ll 1$. However, when $A\in \C^{N\times N}$ lies in a specific subspace of codimension one, the bound in \eqref{eq:stab_intro} improves to $\lVert M_{12}^{A}\rVert\lesssim 1$. We call such matrices \emph{regular} and establish an improved local law for $G_1AG_2$ in this case. When one deals with Wigner matrices, i.e. $D_1=D_2=0$, then $A$ is regular if and only if ${\mathrm Tr}\, A=0$. However, when the deformations $D_1, D_2$ are non--zero, the notion of regularity depends on $D_1, D_2$, as well as on the spectral parameters $z_1, z_2$, in a nontrivial way; see Definition~\ref{def:regulardef} for the precise definition.

We now informally discuss the structure of the bounds in the multi--resolvent local laws in Theorem~\ref{theo:multigllaw} and Proposition~\ref{prop:isoLL} with a concrete example. Let $\boldsymbol{x}, \boldsymbol{y}\in\C^N$ be deterministic unit vectors. When $D_1=D_2$, it was shown in \cite[Proposition 4.4]{equipart} that for $\lVert A\rVert\lesssim1$ we have\footnote{By $\langle \cdot,\cdot\rangle$ we denote the inner product in $\C^N$.}
\begin{equation}
\vert\langle \boldsymbol{x}, (G_1AG_2-M_{12}^A)\boldsymbol{y}\rangle\vert \lesssim
\begin{cases}
\frac{1}{\sqrt{N\eta}}\cdot\frac{1}{\eta}=(N\eta^3)^{-1/2},& A \text{ is general},\\
\frac{1}{\sqrt{N\eta}}\cdot\frac{1}{\eta}\cdot\sqrt{\eta}=(N\eta^2)^{-1/2},& A \text{ is regular}
\end{cases}
\label{eq:iso_intro}
\end{equation}
for $N\eta\gg 1$, where $\eta: =\vert \Im z_1\vert\wedge\vert\Im z_2\vert$ is small in the interesting \emph{local} regime. Note that the bound in the case of regular $A$ is $\sqrt{\eta}$ times better than in the general case. This improvement is known as a $\sqrt{\eta}$\emph{--rule} and was initially observed in \cite{multiG} in the context of Wigner matrices. This rule correctly predicts the size of an arbitrarily long resolvent chain $G_1A_1G_2\cdots A_{k-1}G_k$: each regular $A_i$ accounts for an additional $\sqrt{\eta}$ improvement compared with the bound uniform in $\Re z_1, \Re z_2$ and all bounded observables.

In this paper we make a step further and show how \eqref{eq:iso_intro} improves once we start taking into account the distance between spectral parameters and between deformations. We also show how this decay effect can be combined with the effect that the matrix $A$ is regular. Namely, we prove that (see Proposition~\ref{prop:isoLL} below)
\begin{equation}
\vert\langle \boldsymbol{x}, (G_1AG_2-M_{12}^A)\boldsymbol{y}\rangle\vert \lesssim
\begin{cases}
\frac{1}{\sqrt{N\eta}}\cdot\frac{1}{\eta}\cdot\sqrt{\frac{\eta}{\gamma}}=(N\eta^2\gamma)^{-1/2},& A \text{ is general},\\
\frac{1}{\sqrt{N\eta}}\cdot\frac{1}{\eta}\cdot\sqrt{\frac{\eta}{\gamma}}\cdot\sqrt{\gamma}=(N\eta^2)^{-1/2},&  A \text{ is regular}.
\end{cases}
\label{eq:iso_intro_gamma}
\end{equation}
Note that for the control parameter $\gamma$ from \eqref{eq:stab_intro} we have $\sqrt{\eta/\gamma}\lesssim 1$, showing that in fact this additional factor in \eqref{eq:iso_intro_gamma}, compared to \eqref{eq:iso_intro}, gives additional smallness. 

From \eqref{eq:iso_intro_gamma}, we can thus draw the following two rules of thumb, refining the previous $\sqrt{\eta}$-rule. 
\vspace{2mm}

\noindent $\boldsymbol{\sqrt{\eta/\gamma}}$\textbf{--rule} (\emph{Decay effect}): For each pair of neighboring resolvents with different indices, $G_1$, $G_2$, we gain an additional (small) factor $\sqrt{\eta/\gamma}$. 

\vspace{2mm}

\noindent $\boldsymbol{\sqrt{\gamma}}$\textbf{--rule} (\emph{Regularity effect}): For each regular matrix we gain an additional (small) factor $\sqrt{\gamma}$. 
\vspace{2mm}

\noindent Note that when both effects are present, we gain back the $\sqrt{\eta/\gamma} \, \sqrt{\gamma} = \sqrt{\eta}$-rule. Thus with the proper definition of regularity no additional gain can be obtained from the decay effect; this is natural since the $\sqrt{\eta/\gamma}$-rule comes from the unique unstable direction of the two-body stability operator \eqref{eq:stabop}, while the concept of regularity exactly removes this worst direction.

 In \eqref{eq:iso_intro}--\eqref{eq:iso_intro_gamma} we presented the example of the two-resolvent isotropic law for clarity of presentation, but in Theorem~\ref{theo:multigllaw} and Proposition~\ref{prop:isoLL} we prove analogous results also in the averaged case and for isotropic chains containing three resolvents, respectively. Longer chains can also be handled by our method and our two new rules correctly predict their size, but we refrain from doing so, since they are not needed for the eigenvector overlap. In fact, on a heuristic level one could deduce the results in Theorem~\ref{theo:multigllaw}, Proposition~\ref{prop:isoLL} by using the $\sqrt{\gamma}$-- and $\sqrt{\eta/\gamma}$--rules for each unit $G_1AG_2$ and multiplying the gains from them. In particular, in the averaged case $\langle G_1AG_2A\rangle$ one can extract the gain from both units $G_1AG_2$ and $G_2AG_1$ because of the cyclicity of the trace.

 Our paper is the first instance when both the decay and the regularity effects are considered together, previously only at most one of them was identified at a time. In fact, the study of multi-resolvent local laws started in the context of Wigner matrices where none of these two effects were exploited \cite{therm}; see also \cite{mesofluc1, mesofluc2} for concrete cases when some decay in $|\Re z_1-\Re z_2|$ was identified in the context of central limit theorems for linear eigenvalue statistics. After \cite{therm}, there has been great progress in proving multi-resolvent local laws either for regular observables \cite{ETH, multiG, rank_loc_law, OptLowerBound, ETHgen, edgeETH, quantumchaos,ETHW-type,non-Herm_overlaps} or for different deformations of a specific form for Hermitian matrices \cite{quench} and for the Hermitization of non--Hermitian matrices \cite{CLTnon-herm, mesoCLT, gumbel}. 

We conclude this section by pointing out that the multi-resolvent local laws mentioned above have also been used in several other important problems in random matrix theory; we now name some of them. They played a key role in the recent solution of the bulk universality conjecture for non--Hermitian random matrices \cite{bulk_univ, bulk_real, bulk_real_Dubova}, as well as in proving universality of the distribution of diagonal overlaps of left/right non--Hermitian eigenvectors \cite{non-Herm_overlaps} and of their entries \cite{lr_eigenvectors, sing_val}. Two--resolvents local laws have also been used to prove decorrelation estimates for the resolvent of the Hermitization of non--Hermitian matrices in the context of space--time correlation of linear statistics of non--Hermitian eigenvalues \cite{non-herm_fluct}, and to compute the leading order asymptotic of the log-determinant of non--Hermitian matrices \cite{max_char}. Lastly we point out that similar decorrelation estimates, proven in \cite{quench}, have been used in \cite{Narayanan} to study random hives associated to the eigenvalues of GUE matrices.

\subsection{The method of characteristics} We prove multi-resolvent local laws in Theorem \ref{theo:multigllaw} using the so--called \emph{zigzag} strategy \cite{edgeETH, mesoCLT}, which involves three key steps. First, we prove a concentration bound on the global scale (\emph{global law}), i.e. when the spectral parameters are at a distance of order one from the spectrum. Then we propagate this bound down to the real line by evolving the matrix $W$ along the Ornstein--Uhlenbeck flow, while the spectral parameters $z_1, z_2$ and the deformations $D_1, D_2$ evolve according to a certain deterministic evolution, called \emph{characteristic equations} (see \eqref{eq:def_flow} below for the definition). Along this flow the imaginary part of the spectral parameters is reduced (\emph{zig step}). This second step establishes local laws for spectral parameters with small imaginary parts, though only for matrices with a Gaussian component, added by the Ornstein-Uhlenbeck flow. Finally, the last step of the zigzag strategy eliminates this Gaussian component, again dynamically, via a \emph{Green function comparison} argument (\emph{zag step}). We point out that zig and zag steps are used many times in tandem to decrease the distance of the spectral parameters to the spectrum step by step. 

While the zigzag strategy is a well--established method which has been worked out in many instances, there are several important novelties in our current approach.  
The first novelty is that we perform the proof for an abstract control parameter satisfying certain general conditions which we precisely describe in Definition~\ref{def:gamma}. We do this since the structure of the upper bounds in Theorem~\ref{theo:maintheo} is fairly complicated and we thus need to keep track of different effects at the same time. The second novelty is the self--improving estimates in the zag step stated in Lemmas \ref{lem:condGron2iso} and \ref{lem:condGron3iso}. In fact, we need to perform several zigzag steps to prove the optimal $1/\sqrt{\gamma}$ decay, instead of the $1/\sqrt{\eta}$ in \eqref{eq:iso_intro_gamma}. We do this gradually: We first prove \eqref{eq:iso_intro_gamma} with $1/\sqrt{\gamma}$ replaced by $1/\sqrt{\eta^{1-b}\gamma^b}$ for some $b\in (0,1)$ and then, using this bound as an input, we improve it to $1/\sqrt{\eta^{1-b'}\gamma^{b'}}$ for some $b'>b$. Iterating this procedure finitely many times we finally obtain the desired $1/\sqrt{\gamma}$ in \eqref{eq:iso_intro_gamma}. As an additional third novelty, we extend the delicate analysis of the two-body stability from \cite{echo} to include the new linear term $\mathrm{LT}$.

We conclude this section with a brief historical discussion of the use of the \emph{method of characteristics} (zig step) in random matrix theory\footnote{We point out that, even if we do not mention it, some of the following references also use a comparison step similar to the zag step to remove the additional Gaussian component added via the zig step.}. The idea to study the evolution of the resolvent along the characteristic flow was first introduced in \cite{Pastur, DBM_fluct, Warzel, DBM_beta, gaps} to prove local laws for single resolvents in the bulk and the edge of the spectrum, though only for matrices which have a Gaussian component. In the edge regime a similar version of the characteristics was used before to prove Tracy--Widom universality for the largest eigenvalue of deformed Wigner matrices \cite{D_edge}. In the context of single resolvent local laws, this method was later extended to cover also the cusp regime \cite{UBM, spec_edge, CuspUniv}. All the results mentioned above concern single resolvent local laws. Only more recently the method of characteristics was used to prove local laws for products of two or more resolvents. The first instances of multi--resolvent local laws proven with this method are for the unitary Brownian motion \cite{LQG} and for the product of resolvents of the Hermitization of non--Hermitian matrices at different spectral parameters \cite{mesoCLT}. Since then this method has been very successful in proving a multitude of multi--resolvent local laws for regular matrices or for matrices with specific different deformations \cite{edgeETH, quantumchaos, gumbel, Schatten, ETHW-type, non-Herm_overlaps, nonHermdecay}. In the current work we show that this method is also effective to optimally catch both the decay and the regularity effect at the same time. Finally, we mention that the method of characteristics was also useful to prove central limit theorems for linear eigenvalues statistics 
\cite{DBM_fluct, DBM_beta, CLT_applied, eig_fluct, optCLT}, to study their time correlations \cite{non-herm_fluct}, as well as to study certain extremal statistics \cite{max_char}.

\nc

\subsection*{Notations and conventions}
We set $[k] := \{1, ... , k\}$ for $k \in \N$ and $\langle A \rangle := N^{-1} \mathrm{Tr}(A)$, $N \in \N$,
for the normalized trace of an $N \times N$-matrix $A$.
For positive quantities $f, g$ we write $f \lesssim g$, $f \gtrsim g$, to denote that $f \le C g$ and $f \ge c g$, respectively, for some $N$-independent constants $c, C > 0$ that depend only on the basic control parameters of the model
in Assumption~\ref{ass:momass} below. We denote the complex upper--half plane by $\HH:=\{z\in \C : \Im z>0\}$

 We denote vectors by bold-faced lower case Roman letters $\boldsymbol{x}, \boldsymbol{y} \in \C^{N}$, for some $N \in \N$. Moreover, for vectors $\boldsymbol{x}, \boldsymbol{y} \in \C^{N}$ and a matrix $A\in\C^{N\times N}$ we define
 \begin{equation*}
	\langle \boldsymbol{x}, \boldsymbol{y} \rangle := \sum_i \bar{x}_i y_i\,, 
	\qquad\quad A_{\boldsymbol{x} \boldsymbol{y}} := \langle \boldsymbol{x}, A \boldsymbol{y} \rangle\,. 
\end{equation*}
Matrix entries are indexed by lower case Roman letters $a, b, c , ... ,i,j,k,... $ from the beginning or the middle of the alphabet and unrestricted sums over those are always understood to be over $\{ 1 , ... , N\}$. 

Finally, we will use the concept  \emph{with very high probability},  meaning that for any fixed $D > 0$, the probability of an $N$-dependent event is bigger than $1 - N^{-D}$ for all $N \ge N_0(D)$. We will use the convention that $\xi > 0$ denotes an arbitrarily small positive exponent, independent of $N$.
 Moreover, we introduce the common notion of \emph{stochastic domination} (see, e.g., \cite{semicirclegeneral}): For two families
\begin{equation*}
	X = \left(X^{(N)}(u) \mid N \in \N, u \in U^{(N)}\right) \quad \text{and} \quad Y = \left(Y^{(N)}(u) \mid N \in \N, u \in U^{(N)}\right)
\end{equation*}
of non-negative random variables indexed by $N$, and possibly a parameter $u$, we say that $X$ is stochastically dominated by $Y$, if for all $\epsilon, D >0$ we have 
\begin{equation*}
	\sup_{u \in U^{(N)}} \mathbf{P} \left[X^{(N)}(u) > N^\epsilon Y^{(N)}(u)\right] \le N^{-D}
\end{equation*}
for large enough $N \ge N_0(\epsilon, D)$. In this case we write $X \prec Y$. If for some complex family of random variables we have $\vert X \vert \prec Y$, we also write $X = O_\prec(Y)$. 

\subsection*{Acknowledgments} L. Erd{\H o}s, J. Henheik and O. Kolupaiev were supported by the ERC Advanced Grant ``RMTBeyond'' No.~101020331. G. Cipolloni is partially supported by the MUR Excellence Department Project MatMod@TOV awarded to the Department of Mathematics, University of Rome Tor Vergata, CUP E83C18000100006.

\section{Main results}
\label{sec:mainres}
We consider an $N\times N$ deformed Wigner matrix of the form $H = D + W$, where $D= D^*$ is a deterministic deformation and $W$ is a Wigner matrix, i.e. real symmetric or complex Hermitian matrix $W=W^*$ with independent entries (up to the symmetry constraint) having distribution
\begin{equation}
W_{aa}\stackrel{\dif}{=} \frac{1}{\sqrt{N}}\chi_\dif, \qquad\qquad\quad 
W_{ab}\stackrel{\dif}{=}\frac{1}{\sqrt{N}}\chi_{\mathrm{od}},\quad a>b.
\end{equation}
On the ($N$-independent)
random variables $\chi_\dif \in \R$, $\chi_{\mathrm{od}}\in \C$ we formulate the following assumptions:
\begin{assumption}
\label{ass:momass}
Both $\chi_\dif$, $\chi_{\mathrm{od}}$ are centered $\E\chi_\dif=\E \chi_{\mathrm{od}}=0$ and have unit variance 
$\E \chi_\dif^2=  \E|\chi_{\mathrm{od}}|^2=1$. In the complex case we also assume\footnote{We make this further assumption just to keep the presentation cleaner and shorter. In fact, inspecting the proof of Sections~\ref{sec:zig} and \ref{sec:zag} it is clear that this assumption can easily be removed. This was explained in detail in \cite[Sec. 4.4]{edgeETH}.} that $\E\chi_{\mathrm{od}}^2=0$. Furthermore, we assume the existence of high moments, i.e. for any $p\in\N$ there exists a constant $C_p>0$ such that
\begin{equation}
\label{eq:momass}
\E\big[|\chi_\dif|^p+|\chi_{\mathrm{od}}|^p\big]\le C_p.
\end{equation}
\end{assumption}

Our main goal is to study the decorrelation of the eigenvectors of $W+D_1$, $W+D_2$ for two different 
 Hermitian deformations $D_1,D_2\in\C^{N\times N}$. For simplicity, we
 will always assume that the deformations $D_1,\,D_2$ are traceless, i.e. that $\langle D_1\rangle=\langle D_2\rangle=0$. This is not restrictive, since the spectrum of $W+D_l$, for $l=1,2$, differs from the spectrum of $W+(D_l-\langle D_l\rangle)$ only by a shift of size $\langle D_l\rangle$ to the right. In particular, all the results presented below also hold without the restriction to traceless deformations, one just needs to shift the spectral parameters properly. 

Before stating our main result we introduce some useful notations and definitions. Let $D=D^*\in \C^{N\times N}$ with $\lVert D\rVert\lesssim 1$,  denote its empirical eigenvalue density by
\begin{equation}
\mu(D):=\frac{1}{N}\sum_{i=1}^N \delta_{d_i},
\end{equation}
with $d_1, \dots, d_N$ denoting the eigenvalues of $D$. Let $\mu_{\mathrm{sc}}$ be the semicircular distribution with density $\rho_{\mathrm{sc}}(x):=(2\pi)^{-1}\sqrt{(4-x^2)_+}$; we recall that $\rho_{\mathrm{sc}}$ is the limiting density of the eigenvalues of a Wigner matrix $W$. Then the limiting eigenvalue density of $W+D$ is given by the free convolution (see \cite{Biane} for a detailed discussion)
\begin{equation}
\mu_D=\mu_{\mathrm{sc}}\boxplus\mu(D),
\end{equation}
which is a probability distribution on $\R$. 
Let $m_D$ be the Stieltjes transform of $\mu_D$, i.e. for $z\in \C\setminus \R$ we have
\begin{equation}
\label{eq:Stieltjes}
m_D(z):=\int_\R \frac{\mu_D(\dif x)}{x-z},
\end{equation}
and define the corresponding density by
\begin{equation}
\rho_D(x):=\lim_{\eta\to 0^+}\rho_{D}(x+\ii \eta), \qquad\quad \rho_D(z):=\frac{1}{\pi}\big|\Im m_D(z)\big|.
\end{equation}

Next, fix a small $\kappa>0$, and define the \emph{$\kappa$-bulk} of the density $\rho_D$ by
\begin{equation} \label{eq:kappabulk}
\mathbf{B}_\kappa(D):=\{x\in\R: \rho_D(x)\ge \kappa\}.
\end{equation}
Furthermore, we define the \emph{quantiles} $\gamma_i^D$ of $\rho_D$ implicitly via
\begin{equation}
\label{eq:quant}
\int_{-\infty}^{\gamma_i^D} \rho_D(x)\,\dif x=\frac{i}{N}, \qquad\quad i\in [N].
\end{equation}
From the \emph{eigenvalue rigidity} it is known \cite{StabCor, SlowCor} that $\gamma_i^D$ very well approximates the $i$th eigenvalue $\lambda_i$ of $W+D$.

We are now ready to state our two main results.

\subsection{First main result: Regular observables and eigenstate thermalization (Theorem \ref{theo:maintheo2})}

In order to prove the decomposition in \eqref{eq:decorrintro} with such a precise estimate of the error term, we need to find
the appropriate one-codimensional set of observables, $A=A(D_1,D_2,\gamma_i^{D_1},\gamma_j^{D_2})$, depending both on $D_1,D_2$ as well as on the approximate eigenvalues so that $\langle {\boldsymbol u}_i^1,A {\boldsymbol u}_j^2\rangle$ can be bounded by $N^{-1/2}$. In Definition~\ref{def:regulardef} we characterize the family of such matrices. This result can be thought as a generalization of the ETH  for eigenvectors belonging to two different spectral families.

We start by introducing the notion of \emph{regular observables}, a concept, which in this generality was first introduced in \cite[Def. 3.1]{OptLowerBound} and later in \cite[Def. 4.2]{equipart}.

\begin{definition}[Regular observables]
\label{def:regulardef}
Let $A\in\C^{N\times N}$ be a deterministic matrix, let $z_1, z_2\in \C\setminus \R$ be spectral parameters, and let $D_1,D_2\in \C^{N\times N}$ be deterministic deformations. Fix a small constant\footnote{The precise dependence of $\delta$ on $\kappa$ and $\lVert D_1\rVert, \lVert D_2\rVert$  is discussed in the last paragraph of the proof of Theorem~\ref{theo:maintheo}.} $\delta>0$ depending on $\kappa$ from \eqref{eq:kappabulk} and $\lVert D_1\rVert, \lVert D_2\rVert$. Introduce the short--hand notation $\nu_l:=(z_l,D_l)$, $l=1,2$, we will call $\nu_l$ a \emph{spectral pair}.   Set
\begin{equation}
\phi(\nu_1,\nu_2)=\phi_\delta(\nu_1,\nu_2):=\chi_\delta(\Re z_1-\Re z_2)\chi_\delta(\langle (D_1-D_2)^2\rangle)\chi_\delta(\Im z_1)\chi_\delta(\Im z_2),
\label{eq:def_phi}
\end{equation}
where $0\le \chi_\delta(x)\le 1$ is a symmetric bump function such that it is equal to one for $|x|\le\delta/2$ and equal to zero for $|x|\ge \delta$.

We define the $(\nu_1,\nu_2)$--\emph{regular component} of $A$ by
\begin{equation}
\mathring{A}^{\nu_1,\nu_2}:=A-\phi(\nu_1,\nu_2)\langle VA\rangle I,
\label{eq:regulardef}
\end{equation}
where we used the short--hand notation
\begin{equation}
\label{eq:defV}
V=V(\nu_1,\nu_2):=\frac{M^{D_2}(\Re z_2+\ii \mathfrak{s}\Im z_2)M^{D_1}(\Re z_1+\ii \Im z_1)}{\langle M^{D_1}(\Re z_1+\ii \Im z_1)M^{D_2}(\Re z_2+\ii \mathfrak{s}\Im z_2)\rangle}.
\end{equation}
In \eqref{eq:defV} the relative sign of the imaginary parts is defined as
\[
\mathfrak{s}=\mathfrak{s}(z_1,z_2):=-\mathrm{sgn}(\Im z_1\Im z_2).
\]
We say that $A$ is a \emph{regular observable with respect to $(\nu_1,\nu_2)$} if $A=\mathring{A}^{\nu_1,\nu_2}$.
\end{definition}

Note that our definition of regularity is \emph{asymmetric} in the two spectral pairs. In particular, while $\mathring{A}^{\nu_1,\nu_2} = \mathring{A}^{\nu_1,\bar{\nu}_2}$, it does \emph{not} necessarily hold that $\mathring{A}^{\nu_1,\nu_2}$ equals $\mathring{A}^{\bar{\nu}_1,\nu_2}$. The way of regularization presented in Definition \ref{def:regulardef} is not the only possible one. Alternatively, one could exchange the indices $1$ and $2$, or put $\mathfrak{s}$ on the other argument. It is also possible to define a regularization which is symmetric in $\nu_1, \nu_2$, hence may look more canonical, however we do not proceed in this direction since the definition \eqref{eq:regulardef} which we use is technically more manageable.

\begin{remark}[On the choice of $V$]
The convenience of our choice of $V$ and thus  the definition of regular observables in \eqref{eq:regulardef} lies in the fact that $V$ is the right eigenvector $R_{12}=R(\nu_1,\nu_2)$ corresponding to the largest (in absolute value) eigenvalue of the operator $\mathcal{X}_{12}$, which is defined by
\begin{equation}
\label{eq:stabop}
\mathcal{X}_{12}[\cdot]:=\big[([\mathcal{B}_{12}]^{-1})^*[\cdot^*]\big]^*, \qquad\quad \mathcal{B}_{12}[\cdot]:=1-M^{D_1}(z_1)\langle\cdot\rangle M^{D_2}(z_2)=1-M_1\langle\cdot\rangle M_2,
\end{equation} 
with $M_l$ from \eqref{eq:MDE}.  Here $\mathcal{B}_{12}$ denotes the two--body stability operator that naturally appears when solving the analog of the Dyson equation for the deterministic approximation $M_{12}^A$ of the two-resolvent chain $G_1 A G_2$. With the above choice $\mathring{A}^{\nu_1,\nu_2}$ is defined so that $\langle \mathring{A}^{\nu_1,\nu_2} R_{12}\rangle=0$, i.e. $V=R_{12}$.

The operator $\mathcal{X}_{12}$ has a single very large eigenvalue if and only if $D_1\approx D_2$, $z_1\approx \bar{z}_2$ and $\vert\Im z_1\vert, \vert\Im z_2\vert$ are small. Regular observables are defined precisely such that the action of $\mathcal{X}_{12}$ (and also $\mathcal{X}_{1\bar{2}}$) remain bounded on them. This also explains the role of the cutoff function $\phi$ in \eqref{eq:def_phi}: regularity is a nontrivial concept only when $\phi\neq 0$; in the complementary regime $\phi=0$ every matrix $A$ is regular.
\end{remark}

We are now ready to state our first main result. 
\begin{theorem}[Generalized Eigenstate Thermalization]
\label{theo:maintheo2}
Fix any $\kappa>0$ and fix $D_1,D_2\in\C^{N\times N}$ with $\lVert D_l\rVert\lesssim 1$. Let $W$ be a Wigner matrix satisfying Assumption~\ref{ass:momass}, and, for $l=1,2$, denote by ${\boldsymbol u}_1^l,\dots, {\boldsymbol u}_N^l$ the orthonormal eigenvectors of $W+D_l$. 
Fix indices $i,j$ such that the quantiles $\gamma_i^{D_1}\in\mathbf{B}_\kappa(D_1)$ and $\gamma_j^{D_2}\in\mathbf{B}_\kappa(D_2)$ are in the $\kappa$--bulk of the corresponding densities.  Let $A\in\C^{N\times N}$ be a deterministic matrix which is regular with respect to $ (\nu_1,\nu_2):=((\gamma_i^{D_1}+\ii  0^+,D_1),(\gamma_j^{D_2}+\ii  0^+,D_2))$. Then,
\begin{equation}
\big|\langle {\boldsymbol u}_i^1, A {\boldsymbol u}_j^2\rangle\big|\prec \frac{\lVert A\rVert}{\sqrt{N}}.
\label{eq:A_overlap}
\end{equation}
More generally, for arbitrary observables $A\in\C^{N\times N}$, we have
\begin{equation}
\label{eq:optb2}
\big|\langle {\boldsymbol u}_i^1, A {\boldsymbol u}_j^2\rangle- \langle VA\rangle \phi_{ij} \langle {\boldsymbol u}_i^1, {\boldsymbol u}_j^2\rangle\big|\prec \frac{\lVert A\rVert}{\sqrt{N}},
\end{equation}
where $V=V(\nu_1,\nu_2)$ is defined in \eqref{eq:defV}. Here we defined
\begin{equation}
\phi_{ij}=\phi_{ij}(\delta):=\boldsymbol1(|\gamma_i^{D_1}-\gamma_j^{D_2}|\le \delta)\boldsymbol1(\langle (D_1-D_2)^2\rangle\le \delta)
\end{equation}
for a fixed $\delta=\delta(\kappa)>0$ which is chosen sufficiently small, so that $\|V\|\lesssim 1$ when $\phi_{ij}\neq 0$. The bounds \eqref{eq:A_overlap} and \eqref{eq:optb2} are uniform in the indices $i,j$ such that $\gamma_i^{D_1}\in\mathbf{B}_\kappa(D_1)$ and $\gamma_j^{D_2}\in\mathbf{B}_\kappa(D_2)$.
\end{theorem}

\begin{example}[Eigenstate Thermalization]\label{ex:ETH}
Some special cases of \eqref{eq:optb2} recover previously known results:
\begin{itemize}
\item[(i)] For $D_1=D_2=0$ \eqref{eq:optb2} is the ETH bound for Wigner matrices \cite[Theorem 2.2]{ETH}, as in this case $V=I$, yielding
\begin{equation}
\big|\langle {\boldsymbol u}_i,A{\boldsymbol u}_j\rangle-\langle A\rangle\delta_{ij}\big|\prec \frac{\lVert A\rVert}{\sqrt{N}}.
\label{eq:ETH_Wign}
\end{equation}
Here $\lbrace{\boldsymbol u}_i\rbrace_{i=1}^N$ denote the orthonormal eigenvectors of $W$. Though \eqref{eq:optb2} implies \eqref{eq:ETH_Wign} only for bulk indices, in \cite{ETH},  \eqref{eq:ETH_Wign} was proven for all $i,j\in[N]$.

\item[(ii)] More generally, when $D_1=D_2=D\in\C^{N\times N}$, we have
\[
V=\frac{M(\gamma_i)M^*(\gamma_j)}{\langle M(\gamma_i)M^*(\gamma_j)\rangle},
\]
with $\gamma_i:=\gamma_i^D$. In this case, \eqref{eq:optb2} is the ETH bound for deformed Wigner matrices as given in \cite[Theorem~2.7]{equipart}: 
\[
\left|\langle {\boldsymbol u}_i,A{\boldsymbol u}_j\rangle-\frac{\langle \Im M(\gamma_i)A\rangle}{\langle \Im M(\gamma_i)\rangle}\delta_{ij}\right|\prec \frac{\lVert A\rVert}{\sqrt{N}}
\]
for bulk indices, where we used that $V=\Im M(\gamma_i)/\langle \Im M(\gamma_i)\rangle$. Here $\lbrace{\boldsymbol u}_i\rbrace_{i=1}^N$ denote the orthonormal eigenvectors of $W+D$.
 \end{itemize}

\end{example}

In the next section, we will estimate the overlap $\langle {\boldsymbol u}_i^1, {\boldsymbol u}_j^2\rangle$ appearing in \eqref{eq:optb2}.

\subsection{Second main result: Optimal eigenvector decorrelation (Theorem \ref{theo:maintheo})}
In \eqref{eq:optb2} we showed that for general observables (matrices) $A$ the overlap $\langle {\boldsymbol u}_i^1, A {\boldsymbol u}_j^2\rangle$ can be decomposed as $\langle VA\rangle \langle {\boldsymbol u}_i^1, {\boldsymbol u}_j^2\rangle$ plus a very small error. However, while $\lVert V\rVert\lesssim 1$ is deterministic, the overlap $\langle {\boldsymbol u}_i^1, {\boldsymbol u}_j^2\rangle$ is in general still random. This naturally raises the question if we can give a non-trivial bound on the overlap $\langle {\boldsymbol u}_i^1, {\boldsymbol u}_j^2\rangle$. We positively answer this question in Theorem~\ref{theo:maintheo} below. \nc In particular,
 we show that the size of the overlaps $\langle {\boldsymbol u}_i^1, {\boldsymbol u}_j^2\rangle$
is typically smaller when $D_1,D_2$ are more separated. Another effect is that the overlap becomes smaller when we consider eigenvectors corresponding to well separated eigenvalues. To quantify these types of decay we introduce the \emph{linear term}, defined as
\begin{equation}
\label{eq:linterm}
\LT(z_1,z_2):=\begin{cases}
\left\vert z_1-z_2 -\frac{\langle M_1(D_1-D_2)M_2\rangle}{\langle M_1M_2\rangle}\right\vert\wedge 1,& \text{if } \Im z_1\Im z_2<0,\\[2 mm]
\left\vert z_1-\bar{z}_2 -\frac{\langle M_1(D_1-D_2)M^*_2\rangle}{\langle M_1M^*_2\rangle}\right\vert\wedge 1 ,& \text{if } \Im z_1\Im z_2>0.
\end{cases}
\end{equation}
Here, $M_l=M^{D_l}(z_l)$, for $l=1,2$, is the unique solution \cite[Theorem 2.1]{Helton} of the MDE \eqref{eq:MDE} under the constraint $\Im M_l \Im z_l>0$. We also mention that from \eqref{eq:MDE} one can recover \eqref{eq:Stieltjes} by $m_{D_l}(z_l)=\langle M_l(z_l)\rangle$. From the definition \eqref{eq:linterm} and the fact that $M_l$ and $D_l$ commute it follows that $\LT(z_1,z_2)=\LT(z_1,\bar{z}_2)$ and $\LT(z_1,z_2)=\LT(\bar{z}_1,z_2)$ for any $z_1,z_2\in\C\setminus\R$. Therefore, \eqref{eq:linterm} extends continuously to the real line, i.e. $\LT(z_1,z_2)$ is well-defined for $z_1,z_2\in\R$.

We are now ready to state our second main result. 
\begin{theorem}[Optimal eigenvector decorrelation]
\label{theo:maintheo}
Fix any $\kappa>0$ and fix $D_1,D_2\in\C^{N\times N}$ Hermitian with $\lVert D_l\rVert\lesssim 1$. Let $W$ be a Wigner matrix satisfying Assumption~\ref{ass:momass}, and, for $l=1,2$, denote by ${\boldsymbol u}_1^l,\dots, {\boldsymbol u}_N^l$ the orthonormal eigenvectors of $W+D_l$. Then,
\begin{equation}
\label{eq:optb}
\big|\langle {\boldsymbol u}_i^1, {\boldsymbol u}_j^2\rangle\big|^2\prec \frac{1}{N}\cdot \frac{1}{\langle (D_1-D_2)^2\rangle+\LT(\gamma_i^{D_1},\gamma_j^{D_2})+|\gamma_i^{D_1}-\gamma_j^{D_2}|^2}\wedge 1,
\end{equation}
uniformly over indices $i,j$ such the quantiles $\gamma_i^{D_l}\in\mathbf{B}_\kappa(D_l)$, for $l=1,2$, are in the $\kappa$--bulk of the density $\rho_{D_l}$.
\end{theorem}

We now briefly comment on \eqref{eq:optb}. 
There are several effects that make the eigenvectors almost orthogonal; these are manifested by 
the various terms in the denominator on the rhs. of~\eqref{eq:optb}. 
The main novel effect is expressed by the  term $\langle (D_1-D_2)^2\rangle$ 
that measures the decay due to the fact that the spectra of $W+D_1$, $W+D_2$ become more and more independent as $\langle (D_1-D_2)^2\rangle$ increases. Focusing on this effect only, 
\eqref{eq:optb} simplifies to
\begin{equation}
\big|\langle {\boldsymbol u}_i^1, {\boldsymbol u}_j^2\rangle\big|^2\prec \frac{1}{N\langle (D_1-D_2)^2\rangle},
\end{equation}
uniformly for bulk indices. 
The second effect appears when the corresponding eigenvalues (\emph{energies}), which are well approximated by the quantiles $\gamma^D$, 
are far away. This effect is trivially   present even for a single deformation, $D_1=D_2=D$, in which 
case $\langle {\boldsymbol u}_i^D, {\boldsymbol u}_j^D\rangle =\delta_{ij}$.
Finally, the combination of these two effects is more delicate.  The last term in~\eqref{eq:optb}
shows that the {\it square} of the energy difference, $|\gamma_i^{D_1}-\gamma_j^{D_2}|$, is 
always present in the estimate. This is improved to {\it linear} decay, contained in the term $\LT$,  but 
for  the difference of the {\it renormalized energies} that are the 
energies $\gamma_i^{D_l}$ shifted with  
$\langle M_1 D_l M_2^{(*)}\rangle/\langle  M_1 M_2^{(*)}\rangle$.

\begin{remark}[Eigenvector correlation in perturbative regime]\label{rem:perturb}
As discussed above, we showed that the overlaps $\langle {\boldsymbol u}_i^1, {\boldsymbol u}_j^2\rangle$ are much smaller than $\lVert {\boldsymbol u}_i^1\rVert\cdot\lVert{\boldsymbol u}_j^2\rVert=1$ when $\langle (D_1-D_2)^2\rangle\gg 1/N$. Here, for simplicity, we only consider diagonal overlaps, i.e. $i=j$. We point out that the smallness of \eqref{eq:optb} may be due also to the other two terms in the denominator of the right--hand side of \eqref{eq:optb}, however we do not consider these effects in this remark to keep the presentation simpler. We now
show that this condition is  necessary, in fact  we claim that for $\langle (D_1-D_2)^2\rangle\ll 1/N$ we have
\begin{equation}
\label{eq:almostd}
\langle {\boldsymbol u}_i^1, {\boldsymbol u}_i^2\rangle=1+o(1).
\end{equation}

We now describe how to obtain \eqref{eq:almostd}. By second order perturbation theory we have
\begin{equation}
\label{eq:secordper}
\langle {\boldsymbol u}_i^1, {\boldsymbol u}_i^2\rangle=\langle {\boldsymbol u}_i^1, {\boldsymbol u}_i^1\rangle+\sum_{j\ne i}\frac{|\langle {\boldsymbol u}_i^1,(D_1-D_2) {\boldsymbol u}_j^1\rangle|^2}{(\lambda_i^1-\lambda_j^1)^2}+\dots.
\end{equation}
Since $\langle {\boldsymbol u}_i^1, {\boldsymbol u}_i^1\rangle=1$, we only need 
 to estimate the second term in the right--hand side of \eqref{eq:secordper}. 
 Higher order terms in the perturbation series \eqref{eq:secordper} can be estimated similarly but we
 omit them for simplicity.  In order to deduce \eqref{eq:almostd}, we need to give a lower bound on the denominator  and an upper bound on the numerator in the rhs. of \eqref{eq:secordper}.
 
  For the lower bound
 we have 
  \begin{equation}
\label{eq:levrep}
(\lambda_i^1-\lambda_j^1)^2\gtrsim \frac{|i-j|^2}{N^2}
\end{equation}
with high probability. To see this,  in case of $|i-j|\ge N^\xi$ with an arbitrary small $\xi>0$,
  we employ the rigidity estimate \cite{StabCor, SlowCor}.
  For nearby indices, say  $i<j\le i+ N^\xi$, we use
  \[
\mathbf{P}\big(|\lambda_i^1-\lambda_j^1|\le N^{-1-\omega}\big) \le 
\mathbf{P}\big(|\lambda_i^1-\lambda_{i+1}^1|\le N^{-1-\omega}\big) \le N^{-c\omega},
\]
for some small fixed $c,\omega>0$.
 In the last step we used  the universality of the eigenvalue gaps for deformed
  Wigner matrices\footnote{The first bulk universality result in terms of correlation functions
   for deformed Wigner matrices with diagonal deformations was given in \cite{LeeSchnelliStetlerYau}. The 
gap universality in full generality was given, e.g.,  in Corollary 2.6 of \cite{SlowCor}.}
   and the explicit level repulsion bound for GOE/GUE matrix:
\[
\mathbf{P}^{\mathrm GOE/GUE}\big(|\lambda_i^1-\lambda_{i+1}^1|\le N^{-1-\omega}\big)\le N^{-c\omega}.
\]

 For the upper bound, we employ ETH for deformed Wigner matrix $W+D_1$ in the Hilbert-Schmidt norm form:
 \begin{equation}
\label{eq:ETHhs}
\left|\langle {\boldsymbol u}_i^1,(D_1-D_2) {\boldsymbol u}_j^1\rangle-\frac{\langle (D_1-D_2)\Im M^{D_1}(\gamma_i^1)\rangle}{\langle \Im M^{D_1}(\gamma_i^1)\rangle}\delta_{ij}\right|\lesssim  \frac{\langle (D_1-D_2)^2 \rangle^{1/2}}{\sqrt{N}},
\end{equation}
with $M^{D_1}$ from \eqref{eq:MDE} and $\gamma_i^1:=\gamma_i^{D_1}$ being the quantiles from \eqref{eq:quant}.
% First, we notice that
%  the typical distance between neighboring bulk eigenvalues is $1/N$, 
% and so we have
%\begin{equation}
%\label{eq:levrep}
%(\lambda_i^1-\lambda_j^1)^2\sim \frac{|i-j|^2}{N^2}
%\end{equation}
%with probability tending to one as $N\to\infty$. The upper bound in \eqref{eq:levrep} follows from rigidity, the 
%lower bound from level repulsion [Add proper reference]. 
In \cite{equipart} we proved ETH for deformed Wigner matrices in the form
\begin{equation}
\label{eq:ETHop}
\left|\langle {\boldsymbol u}_i^1,(D_1-D_2) {\boldsymbol u}_j^1\rangle-\frac{\langle (D_1-D_2)\Im M^{D_1}(\gamma_i^1)\rangle}{\langle \Im M^{D_1}(\gamma_i^1)\rangle}\delta_{ij}\right|\lesssim  \frac{\lVert D_1-D_2 \rVert}{\sqrt{N}},
\end{equation}
i.e with the operator norm $\lVert D_1-D_2 \rVert$ instead of the Hilbert-Schmidt norm of $D_1-D_2$.
%with $M^{D_1}$ from \eqref{eq:MDE} and $\gamma_i^1:=\gamma_i^{D_1}$ being the quantiles from \eqref{eq:quant}. However, this bound is not enough to show the smallness of the $j$--summation in \eqref{eq:secordper} 
%just for $\langle (D_1-D_2)^2\rangle\ll 1/N$. Fortunately,  \eqref{eq:ETHop} can be improved to
%\begin{equation}
%%\label{eq:ETHhs}
%\left|\langle {\boldsymbol u}_i^1,(D_1-D_2) {\boldsymbol u}_j^1\rangle-\frac{\langle (D_1-D_2)\Im M^{D_1}(\gamma_i^1)\rangle}{\langle \Im M^{D_1}(\gamma_i^1)\rangle}\delta_{ij}\right|\lesssim  \frac{\langle (D_1-D_2)^2 \rangle^{1/2}}{\sqrt{N}},
%\end{equation}
%i.e. the operator norm $\lVert D_1-D_2 \rVert$ may be replaced with the Hilbert-Schmidt norm of $D_1-D_2$.
Strictly speaking, the improved bound \eqref{eq:ETHhs}  is nowhere proven for the eigenvectors 
of  deformed Wigner matrix $W+D_1$, $D_1\ne 0$, 
however this can be easily obtained using a similar  (in fact much simpler) zigzag approach as the
 one presented in Sections~\ref{sec:zig}--\ref{sec:zag} of this paper. 
 We also point out that a bound similar to \eqref{eq:ETHhs} has already been obtained, using similar arguments, for Wigner matrices ($D_1=0$) in \cite{edgeETH} and for Wigner--type matrices with a diagonal deformation in \cite{ETHW-type}.

Finally, combining \eqref{eq:levrep} with \eqref{eq:ETHhs}, from \eqref{eq:secordper} we obtain
\[
\langle {\boldsymbol u}_i^1, {\boldsymbol u}_i^2\rangle=1+\mathcal{O}\left(N\langle (D_1-D_2)^2\rangle\right)
\]
which directly implies the desired claim \eqref{eq:almostd}. 
Every step of this argument can easily be made rigorous but we omit details for brevity.

\end{remark}

\begin{remark}[Independence of eigenvalue gaps] We point out that using the eigenvector overlap bound \eqref{eq:optb} we can prove that the eigenvalue gaps in the bulk of the spectrum of $W+D_1$, $W+D_2$ are independent as long as $\langle (D_1-D_2)^2\rangle\gg 1/N$. In fact, following verbatim \cite[Section 7]{CLTnon-herm} and its adaptation to the Hermitian case in \cite{quench}, we can prove the desired independence via the study of weakly correlated Dyson Brownian motions. The only input required for this proof is the overlap bound $|\langle {\boldsymbol u}_i^1, {\boldsymbol u}_j^2\rangle|\ll 1$. 
\end{remark}

We point out that
the bounds \eqref{eq:optb2} and \eqref{eq:optb} are optimal except for the $N^\epsilon$-factor (for any $\epsilon > 0$) coming from the $\prec$ bound. This can be seen by the fact that a local $N^\delta$-average of eigenvectors
\[
\frac{1}{N^{2\delta}}\sum_{|i-i_0|\le N^\delta\atop |j-j_0|\le N^\delta} N\big|\langle {\boldsymbol u}_i^1, A {\boldsymbol u}_j^2\rangle\big|^2
\]
for some small $\delta >0$ is proportional to products of resolvents, as shown in the rhs. of \eqref{eq:fromevecttores} below, for which we precisely compute the deterministic approximation in Theorem~\ref{theo:multigllaw}.  
 
We stated our main results Theorems~\ref{theo:maintheo2} and \ref{theo:maintheo} only for indices in the bulk of the spectra of $W+D_1$, $W+D_2$ and estimated the error in Theorem \ref{theo:maintheo2} in terms of the operator norm $\Vert A \Vert$. %{\cor [J: This is not yet quite elegant]}
In Section~\ref{sec:extension} below, we comment on possible extensions and improvements.

\section{Proofs of the Main Results: Multi-resolvent local laws}

In this section we present several technical tools and preliminary results that will be often used in this paper. More precisely, in Section~\ref{sec:stabop} we study lower bounds on the stability operator, which are one of the fundamental inputs to obtain the decay in the rhs. of \eqref{eq:optb}. Then, in Section~\ref{sec:multigllaw}, we state our main technical result (Theorem~\ref{theo:multigllaw} below), which is a \emph{multi--resolvent local law} for the product of the resolvents of $W+D_1$ and $W+D_2$, with $D_1,D_2\in\C^{N\times N}$. Lastly, in Section~\ref{sec:extension} we comment on the optimality and discuss some possible extension of Theorem~\ref{theo:multigllaw}.

\subsection{Preliminaries on the stability operator}
\label{sec:stabop}

Recall the definition of the stability operator from \eqref{eq:stabop}. One can easily see that its smallest (in absolute value) eigenvalue is $1-\langle M_1M_2\rangle$ with associated eigenvector $M_1 M_2$; {the only other eigenvalue, trivially equal to one, is highly degenerate}. %\cor [Add refernece] \nc, 
Here, $M_1=M^{D_1}(z_1)$, $M_2=M^{D_2}(z_2)$ are the solutions of the MDE \eqref{eq:MDE}.  In this section we give a lower bound on its absolute value
\begin{equation}
\label{eq:smallestev}
\beta(z_1,z_2):=\big|1-\langle M_1M_2\rangle\big|.
\end{equation}
The main control parameters in the following statements are $\langle (D_1-D_2)^2\rangle$ and the linear term $\LT(z_1,z_2)$ which is defined as in \eqref{eq:linterm}, for $z_1,z_2\in\C\setminus\R$. The proof of the following proposition and comments about its optimality are postponed to Section~\ref{sec:addstabop} of the Supplementary Material \cite{supplement}.

\begin{proposition}[Stability bound]\label{prop:stab} 
Fix a (large) constant $L>0$. Let $D_1, D_2\in\C^{N\times N}$ be Hermitian matrices with $\langle D_l\rangle=0$ and $\lVert D_l\rVert\le L$ for $l=1,2$. For $z_l=E_l+\ii\eta_l\in\HH$, $l=1,2$, recall the notation $\rho_l:=\pi^{-1}\langle \Im M^{D_l}(z_l)\rangle$ and denote 
\begin{equation}
\beta_*:=\beta_*(z_1,z_2)=\beta(z_1,z_2)\wedge\beta(z_1,\bar{z}_2),
\label{eq:def_beta_*}
\end{equation}
\begin{equation}
\widehat{\gamma}:=\widehat{\gamma}(z_1,z_2)= \langle (D_1-D_2)^2\rangle+\LT + \vert E_1-E_2\vert^2\wedge 1 + \frac{\eta_1}{\rho_1}\wedge 1+\frac{\eta_2}{\rho_2}\wedge 1. \label{eq:def_gamma_0}
\end{equation}
Then uniformly in $z_1, z_2\in\HH$ it holds that
\begin{equation}
(\rho_1+\rho_2)^2\lesssim\beta(z_1,z_2).
\label{eq:rho_stab}
\end{equation}
Moreover, fix a (large) constant $C_0>0$ and assume that for some  intervals $\mathbf{I}_1,\mathbf{I}_2\subset \R$ we have
\begin{equation}
\sup_{\Re z_l\in\mathbf{I}_l}\lVert M^{D_l}(z_l)\rVert\le C_0,\quad l=1,2.
\label{eq:M_bound_I}
\end{equation}
Then uniformly in $z_l=E_l+\ii \eta_l\in\HH$ with $E_l\in\mathbf{I}_l$\nc, $l=1,2$, it holds that
\begin{equation}
\widehat{\gamma} \lesssim \beta_*\lesssim \widehat{\gamma}^{1/4},		
\label{eq:stab_bound}
\end{equation}
where the implicit constants depend only on $L$ and $C_0$.
\end{proposition}

Note that \eqref{eq:M_bound_I} is automatically satisfied for $\mathbf{I}_l=\mathbf{B}_\kappa(D_l)$ with the constant $C_0$ depending only on $\kappa$. This follows from the bound
\begin{equation*}
\lVert M^D_l(z_l)\rVert\le \left(\vert \Im z_l\vert + \vert\langle\Im M^{D_l}(z_l)\rangle\vert\right)^{-1} \le C \kappa^{-1}\,. 
\end{equation*}

\nc

We point out that, even if not highlighted in the notation, the quantities $\beta,\beta_*$ and $\widehat{\gamma}$ also depend on the deformations $D_1,D_2$. We will often omit this dependence in notations when it is clear what the arguments are.

The most relevant part of Proposition~\ref{prop:stab} is the lower bound $\beta_*\gtrsim\widehat{\gamma}$. This bound in a weaker form (more precisely, without $\LT$ included in $\widehat{\gamma}$) has already appeared in \cite[Proposition 4.2]{echo}. It should be viewed as a lower bound on the two--body stability operator in terms of simpler control parameters collected in $\widehat{\gamma}$.  In the inequality $\beta_*\lesssim \widehat{\gamma}^{1/4}$ we do not pursue getting the optimal power for $\widehat{\gamma}$. 
In fact, any positive exponent would work for our purpose.

\subsection{Multi--resolvent local law: Proofs of Theorems \ref{theo:maintheo} and \ref{theo:maintheo2}}
\label{sec:multigllaw}

The main idea to give a bound on single eigenvector overlaps as in Theorems~\ref{theo:maintheo}--\ref{theo:maintheo2} is to upper bound the overlaps by traces of products of two resolvents, and then prove a bound for these quantities (see e.g. \eqref{eq:fromevecttores} below). For this reason in this section we first recall the traditional single resolvent local law, and then state our new \emph{multi--resolvent local laws}, which are our main technical result.

Let $D\in\C^{N\times N}$, with $\lVert D\rVert\lesssim 1$, and let $W$ be a Wigner matrix satisfying Assumption~\ref{ass:momass}. Then, for $z\in\C\setminus\R$ we define the resolvent of $W+D$ by $G(z)=G^D(z):=(W+D-z)^{-1}$. It is well known that in the limit $N\to\infty$ the resolvent becomes approximately deterministic $G(z)\approx M(z)$, with  $M(z)=M^D(z)$ being the solution of \eqref{eq:MDE}. This is expressed by the following single resolvent local law \cite[Theorem 2.1]{SlowCor}
\begin{equation}
\label{eq:singllaw0}
\big|\langle (G(z)-M(z))A\rangle\big|\prec \frac{1}{N|\Im z|}, \qquad\quad \big|\langle{\boldsymbol x}, (G(z)-M(z)){\boldsymbol y}\rangle\big|\prec {\frac{1}{\sqrt{N|\Im z|}}},
\end{equation}
uniformly in deterministic matrices $A\in\C^{N\times N}$ with $\lVert A\rVert\le 1$, unit vectors ${\boldsymbol x},{\boldsymbol y}\in\C^N$, and spectral parameters $z$ in the bulk regime, i.e.~$\Re z \in \mathbf{B}_\kappa(D)$ for some fixed $\kappa > 0$. 

The main topic of this section, however, is to compute the deterministic approximation of the products of two resolvents $G_1AG_2$, with $G_l:=G^{D_l}(z_l)$ for $l=1,2$ and a deterministic observable in between. While $G^{D_l}(z_l)\approx M^{D_l}(z_l)$, the deterministic approximation of $G_1AG_2$ is not given by the product of the deterministic approximations $M_1AM_2$, but, as we will see from our result, rather by
\begin{equation}
\label{eq:defM12}
M_{\nu_1,\nu_2}^A:=\mathcal{B}_{12}^{-1}\big[M_1AM_2\big],
\end{equation}
with $\nu_l=(z_l,D_l)$, $M_l=M^{D_l}(z_l)$ and with $\mathcal{B}_{12}$ being the stability operator defined in \eqref{eq:stabop}. We will stick to the following notational convention. In most cases we will simplify the notation $M_{\nu_1,\nu_2}^A$ to $M_{12}^A$ when it is clear from the context what the arguments are. Moreover, if $\nu_1, \nu_2$ depend on an additional parameter $t$, i.e. $\nu_1=\nu_1(t)$, $\nu_2=\nu_2(t)$, we will denote the dependence of \eqref{eq:defM12} on $t$ in two equivalent ways:
\begin{equation}
M^A_{\nu_1(t),\nu_2(t)}=M^A_{12,t}.
\label{eq:M2_t-dep}
\end{equation}

On the deterministic approximation defined in \eqref{eq:defM12} we have the bound (see Proposition~\ref{prop:norm_M_bounds} below)
\begin{equation}
\label{eq:boundM12}
\lVert M_{\nu_1,\nu_2}^A\rVert\lesssim \frac{\lVert A\rVert}{\beta_*},
\end{equation}
with $\beta_*$ from \eqref{eq:def_beta_*}. In the case when $A$ is $(\nu_1,\nu_2)$-regular, i.e. $A=\mathring{A}^{\nu_1,\nu_2}$, \eqref{eq:boundM12} improves to $\lVert M_{\nu_1,\nu_2}^A\rVert\lesssim \lVert A\rVert$. For precise statement see Proposition \ref{prop:norm_M_bounds}. We are now ready to state our main technical result.

\begin{theorem}[Average two-resolvent local laws in the bulk]
\label{theo:multigllaw}
Fix $L, \epsilon, \kappa > 0$. Let $W$ be a Wigner matrix satisfying Assumption~\ref{ass:momass}, and let $D_1,D_2\in\C^{N\times N}$ be Hermitian matrices
such that $\langle D_l\rangle=0$ and $\lVert D_l\rVert\le L$ for $l=1,2$. For spectral parameters $z_1, z_2 \in \C \setminus \R$, denote  $\eta_l:=|\Im z_l|$ and  $\eta_*:=\eta_1\wedge\eta_2\wedge 1$. 
Finally, let $\widehat{\gamma} = \widehat{\gamma}(z_1, z_2)$ be defined as in \eqref{eq:def_gamma_0}. Then, the following holds:

\begin{itemize}

\item[Part 1.] \textnormal{[General case]} For deterministic \nc $B_1,B_2\in\C^{N\times N}$ we have
\begin{equation}
\label{eq:g1g2b}
\big|\big\langle \big(G^{D_1}(z_1)B_1G^{D_2}(z_2)-M_{\nu_1,\nu_2}^{B_1}\big)B_2\big\rangle\big|\prec \left(\frac{1}{N\eta_1\eta_2}\wedge\frac{1}{\sqrt{N\eta_*}\widehat{\gamma}}\right)\lVert B_1\rVert \lVert B_2\rVert
\end{equation}
uniformly in $B_1, B_2$\nc, spectral parameters satisfying $\Re z_l\in\mathbf{B}_{\kappa}(D_l)$, $|z_l| \le N^{100}$, for $l=1,2$, and~$\eta_* \ge N^{-1+\epsilon}$. 

\item[Part 2.]  \textnormal{[Regular case]} Consider deterministic \nc $A_1, A_2, B\in\C^{N\times N}$. Moreover, recalling \eqref{eq:regulardef}, let $A_1$ be $(\pair_1, \pair_2)$-regular and $A_2$ be $(\pair_2, \pair_1)$-regular. Then,
\begin{align}
\label{eq:reg1}
\big|\big\langle \big(G^{D_1}(z_1)A_1G^{D_2}(z_2)-M_{\nu_1,\nu_2}^{A_1}\big)B\big\rangle\big|&\prec \left(\frac{1}{N\eta_1\eta_2}\wedge\frac{1}{\sqrt{N\eta_* \widehat{\gamma}}}\right)\lVert A_1\rVert\lVert B\rVert, \\
\label{eq:reg2}
\big|\big\langle \big(G^{D_1}(z_1)A_1G^{D_2}(z_2)-M_{\nu_1,\nu_2}^{A_1}\big)A_2\big\rangle\big|&\prec \left(\frac{1}{N\eta_1\eta_2}\wedge \frac{1}{\sqrt{N\eta_*}}\right)\lVert A_1\rVert\lVert A_2\rVert,
\end{align}
uniformly in $A_1, A_2, B$, spectral parameters satisfying $\Re z_l\in\mathbf{B}_{\kappa}(D_l)$, $|z_l| \le N^{100}$, for $l=1,2$, and $\eta_* \ge N^{-1+\epsilon}$.

\end{itemize}
\end{theorem}

One important technical tool needed for the proof of Part 2 Theorem \ref{theo:maintheo2} is the content of the following lemma, which compares regularizations of a deterministic matrix with respect to different pairs of spectral pairs. We point out that this is not a type of continuity statement about the dependence of $\mathring{A}^{\nu_1,\nu_2}$ on $(\nu_1,\nu_2)$ like in \cite[Lemma 3.3]{OptLowerBound}.
The proof of Lemma \ref{lem:reg_reg} is presented in Section \ref{sec:reg_reg} of the Supplementary Material \cite{supplement}.
\begin{lemma}[Comparison of different regularizations]\label{lem:reg_reg} Fix (large)
$L>0$ and (small) $\kappa>0$. Let $D_1,D_2\in\C^{N\times N}$ be Hermitian deformations.
Moreover, assume that $\langle D_1\rangle=\langle D_2\rangle=0$ and $\lVert D_1\rVert\le L$, $\lVert D_2\rVert\le L$. Take spectral parameters $z_1, z_2\in\HH$ such that $\min\lbrace\rho_1(z_1),\rho_2(z_2)\rbrace\ge\kappa$, where $\rho_l(z_l):=\langle\Im M^{D_l}(z_l)\rangle\pi^{-1}$, $l=1,2$. For $y_1,y_2\ge 0$ denote $z_l':=z_l+\ii y_l\in\HH$, $l=1,2$. Additionally we use notations $\nu_l:=(z_l,D_l)$, $\nu_l':=(z_l',D_l)$ and $\bar{\nu}_l':=(\bar{z}_l',D_l)$ for $l=1,2$. Then for any observable $A\in\C^{N\times N}$ we have
\begin{equation}
\begin{split}
& \lVert \mathring{A}^{\nu_1',\nu_2'}-\mathring{A}^{\nu_1,\nu_2}\rVert\lesssim \lVert A\rVert\sqrt{\widehat{\gamma}(z_1',z_2')},\qquad  \lVert \mathring{A}^{\bar{\nu}_1',\nu_2'}-\mathring{A}^{\nu_1,\nu_2}\rVert\lesssim \lVert A\rVert\sqrt{\widehat{\gamma}(z_1',z_2')},\\
& \lVert \mathring{A}^{\nu_2',\nu_1'}-\mathring{A}^{\nu_1,\nu_2}\rVert\lesssim \lVert A\rVert\sqrt{\widehat{\gamma}(z_1',z_2')},\qquad  \lVert \mathring{A}^{\bar{\nu}_2',\nu_1'}-\mathring{A}^{\nu_1,\nu_2}\rVert\lesssim \lVert A\rVert\sqrt{\widehat{\gamma}(z_1',z_2')}.
\end{split}
\label{eq:reg_reg}
\end{equation}
The implicit constants in \eqref{eq:reg_reg} depend only on $L$ and $\kappa$. Recall that when the complex conjugation falls on the second spectral pair, we have $\mathring{A}^{\nu_1',\bar{\nu}_2'}=\mathring{A}^{\nu_1',\nu_2'}$ by definition.
\end{lemma}
In the rhs.~of \eqref{eq:reg_reg} we do not aim to get the optimal power of $\widehat{\gamma}$, but rather formulate Lemma \ref{lem:reg_reg} minimalistically and collect only those bounds which will be used later. 

Given Theorem~\ref{theo:multigllaw} and Lemma \ref{lem:reg_reg} we immediately conclude the proof of Theorems~\ref{theo:maintheo} and \ref{theo:maintheo2}.

\begin{proof}[Proof of Theorems~\ref{theo:maintheo2} and  \ref{theo:maintheo}]

We first prove Theorem~\ref{theo:maintheo}. Consider $i,j\in [N]$ such that $\gamma_i^{D_1}\in\mathbf{B}_\kappa(D_1)$ and  $\gamma_j^{D_2}\in\mathbf{B}_\kappa(D_2)$, and let $\eta=N^{-1+\epsilon}$ for a small fixed $\epsilon>0$. Then, by spectral decomposition we readily obtain
\begin{equation}
\label{eq:fromevecttores}
N\big|\langle {\boldsymbol u}_i^1,{\boldsymbol u}_j^2\rangle\big|^2\prec (N\eta)^2\big|\big\langle \Im G^{D_1}(\gamma_i^{D_1}+\ii\eta)\Im G^{D_2}(\gamma_j^{D_2}+\ii\eta) \big\rangle\big|.
\end{equation}
We point out that to prove \eqref{eq:fromevecttores} we also use the standard rigidity bound from \cite{StabCor, SlowCor}:
\begin{equation}
\big|\lambda_i^D-\gamma_i^D\big|\prec \frac{1}{N}, \qquad\quad \gamma_i^D\in\mathbf{B}_\kappa(D).
\end{equation}
Finally, combining \eqref{eq:fromevecttores} with \eqref{eq:g1g2b}, using \eqref{eq:boundM12} and
\begin{equation}
\label{eq:boundmnorm}
\lVert M^{D_l}(z_l) \rVert=\left\lVert \frac{1}{D_l-z_l-\langle M^{D_l}(z_l) \rangle} \right\rVert\le \vert\langle\Im M^{D_l}(z_l)\rangle\vert^{-1}\le C_\kappa
\end{equation}
in the $\kappa$--bulk of the spectrum, and the definition of $\widehat{\gamma}$ from \eqref{eq:def_gamma_0}, we immediately conclude \eqref{eq:optb}.

Now we discuss how to adjust the argument above to prove Theorem \ref{theo:maintheo2}. The fact that $\lVert V\rVert\lesssim 1$ in the regime when $\phi_{ij}\ne 0$ follows by simple perturbation theory if $\delta$ is chosen sufficiently small in terms of $\kappa$ from \eqref{eq:kappabulk} and in terms of $\lVert D_1\rVert, \lVert D_2\rVert$. In the complementary regime this bound is trivial. One more input which is needed for the proof of Theorem~\ref{theo:maintheo2} is Lemma~\ref{lem:reg_reg} specialized to the case $y_1=y_2=0$. In fact, Lemma~\ref{lem:reg_reg} implies that if $A$ is $(\nu_1,\nu_2)$--regular, then it is close in operator norm to $\mathring{A}^{\nu_2,\nu_1}$. The rest of the details are omitted for the sake of brevity (see \cite[Theorem 2.2]{OptLowerBound} for the details of a very similar proof).
\end{proof}

We conclude this section commenting on the optimality of Theorem~\ref{theo:multigllaw}. 

\subsection{Optimality and possible extensions of Theorem~\ref{theo:multigllaw}}
\label{sec:extension}

In this section  explain in what sense Theorem~\ref{theo:multigllaw}  is optimal and that it can be  extended 
to energies where the limiting eigenvalue density is small. We also comment on the possibility of replacing the operator norm in the rhs. of the estimates in Theorem~\ref{theo:multigllaw} with the typically smaller Hilbert--Schmidt norm. All these improvements and extensions can be achieved following our {\it zigzag} strategy of first proving the desired result for matrices with a fairly large Gaussian component, as in Section~\ref{sec:zig} (zig step), and then prove it for general matrices using a dynamical comparison argument, as in Section~\ref{sec:zag} (zag step). We omit the details of these proofs to keep the presentation simple and short, in fact, the main focus of this paper is to develop techniques to handle different deformations $D_1, D_2$ and we do it in the simpler cases
  of the bulk of the spectrum proving estimates in terms of the operator norm. In the following we give more precise references to papers where  similar analyses were already performed in detail. 

We first consider the bound \eqref{eq:g1g2b}. In the bulk of the spectrum, this bound is optimal except for the fact that $1/(\sqrt{N\eta_*}\widehat{\gamma})$ should be replaced by $1/(N\eta_*\widehat{\gamma})$. Notice, that once the bound $1/(N\eta_*\widehat{\gamma})$ is achieved, the term $1/(N\eta_1\eta_2)$ in \eqref{eq:g1g2b} is obsolete, as it is always bigger. This improvement can be achieved by proving (weaker) local laws also for products of longer resolvents; see, e.g., \cite{multiG, edgeETH, Schatten} for similar arguments. In fact, this overestimate is due to the fact that the four resolvents chains, appearing e.g. in the quadratic variation of the stochastic term in \eqref{eq:GG_evol} below are currently estimated in terms of products of traces of two resolvents using certain crude reduction inequalities (see e.g. \eqref{eq:mart_reduct}). Following the evolution of these longer chains more carefully would give the improvement $1/(N\eta_*\widehat{\gamma})$.

We also believe  that assuming that $M^{D_l}(z_l)$ are bounded throughout the spectrum (see e.g. condition \eqref{eq:M_bound_I} with $\mathbf{I}=\R$ and Remark~\ref{rmk:Mbdd} below) one can extend the local laws \eqref{eq:g1g2b}--\eqref{eq:reg2} to hold uniformly in the spectrum with the similar zigzag strategy\nc. In fact, in this case we expect an additional gain $\sqrt{\rho_1+\rho_2}$ in their rhs.; see \cite[Theorem 2.4]{edgeETH} for a similar argument. We postpone the details to future work.  

Finally, the operator norm in \eqref{eq:g1g2b}--\eqref{eq:reg2} can be replaced by the typically smaller Hilbert--Schmidt norm. Again, this can be achieved following our proof in Sections~\ref{sec:zig}--\ref{sec:zag}, but we omit a detailed proof of brevity. A similar proof was carried out in full detail in \cite{edgeETH} in the simpler setting of Wigner matrices using the Lindeberg swapping technique, but it can be readily adapted to the current case. Additionally, we expect that the Lindeberg technique can be replaced by a dynamical argument similar to Section~\ref{sec:zag}.

\section{Zigzag strategy: Proof of Theorem \ref{theo:multigllaw}} \label{sec:zigzag}

To prove the multi--resolvent local law in Theorem~\ref{theo:multigllaw} we follow the \emph{zigzag strategy}, similarly to \cite{mesoCLT, edgeETH}. That is, we prove Theorem~\ref{theo:multigllaw} by running in tandem the \emph{characteristic flow} associated to a matrix valued Ornstein--Uhlenbeck process, and a \emph{Green's function comparison theorem (GFT)}.

More precisely, the zigzag strategy consists of the following three steps: 
\begin{itemize}
\item[\textbf{1.}] \textbf{Global law:} Prove a \emph{global law} for spectral parameters $z_j$ that are "far away" from the self-consistent spectrum, $\min_j \mathrm{dist}(z_j, \mathrm{supp}(\rho^{D_j})) \ge \delta$ (see Section \ref{subsec:global}). 
\item[\textbf{2.}] \textbf{Characteristic flow:} Propagate the bound from large distances to a smaller one by considering the evolution of the Wigner matrix $W$ along an Ornstein-Uhlenbeck flow, thereby introducing a Gaussian component (see Section \ref{subsec:charflowzig}). The spectral parameters evolve according to the \emph{characteristic flow} defined in \eqref{eq:def_flow}. 
The simultaneous effect of these two evolutions is a key cancellation of two large terms.
\item[\textbf{3.}] \textbf{Green function comparison:}  Remove the Gaussian component by a Green function comparison (GFT) argument (see Section \ref{subsec:GFTzag}).
\end{itemize}
In order to reduce the distance of the spectral parameters down to the optimal scale for the local law, Steps~2 and 3 will be applied many times in tandem. This inductive argument is carried out in Proposition \ref{lem:induction} in Section~\ref{subsec:zigzagconcl}. 

While Theorem~\ref{theo:multigllaw} states local laws only for average quantities, within the GFT, isotropic resolvent chains of the form 
\begin{equation}
(GBG)_{\boldsymbol x \boldsymbol y} \quad \text{or} \quad (GBGBG)_{\boldsymbol x \boldsymbol y}
\label{eq:need_3G}
\end{equation}
naturally arise, which requires to analyze them as well. That is, we necessarily need to perform the zigzag strategy for such quantities in an analogous way.

Throughout the entire argument, all process will run for times $t$ in a fixed interval $[0,T]$ for some terminal time $T > 0$ of order one, which we will choose below in \eqref{eq:def_T}. 
\subsection{Input global laws} \label{subsec:global}

Here we state the necessary global laws that will be used as an input to prove Theorem~\ref{theo:multigllaw}.  Note that in the global regime no restriction to the bulk is necessary.\nc
\begin{proposition}[Global law] \label{prop:global}
Fix $L, \epsilon, \delta > 0$. Let $W$ be a Wigner matrix satisfying Assumption~\ref{ass:momass}, and let $D_1,D_2\in\C^{N\times N}$ be bounded Hermitian matrices, %satisfying Assumption \ref{ass:M_bound} with constant $C_0$, 
i.e.~$\lVert D_l\rVert\le L$ for $l=1,2$. For spectral parameters $z_1, z_2 \in \C \setminus \R$ with $\min_j \mathrm{dist}(z_j,\mathrm{supp}(\rho^{D_j}))\ge \delta$, deterministic unit vectors ${\boldsymbol x}, {\boldsymbol y}\in\C^N$ and matrices $B_1,B_2\in\C^{N\times N}$, we have (recall $G_j:=(W+D_j-z_j)^{-1}$)
\begin{align}
\left\vert\left\langle \left(G_1B_1G_2 - M_{\nu_1,\nu_2}^{B_1}\right) B_2\right\rangle\right\vert&\prec \frac{\lVert B_1\rVert\lVert B_2\rVert }{N}\,, 
\label{eq:avincond} \\
\left\vert \left\langle {\boldsymbol x},\left(G_1B_1G_2 -M_{\nu_1,\nu_2}^{B_1}\right){\boldsymbol y}\right\rangle\right\vert &\prec \frac{\lVert B_1\rVert}{\sqrt{N}} \,,
\label{eq:iso2incond} \\
\left\vert \left\langle {\boldsymbol x},G_1B_1G_2B_2G^{(*)}_{1,s} {\boldsymbol y}\right\rangle\right\vert &\prec \lVert B_1\rVert\lVert B_2\rVert \,.
\label{eq:iso3incond}
\end{align}
\end{proposition}
\begin{proof}
The proof of these global laws is very similar to the one presented in \cite[Appendix B]{multiG}, we thus omit several details and just present the main steps. To keep the presentation short and simple we only present the proof in the averaged case.

In the following we will often use the fact that
\begin{equation}
\label{eq:normbound}
\lVert G_j\rVert\lesssim \frac{1}{\min_j  \mathrm{dist}(z_j,\mathrm{supp}(\rho^{D_j}))}\le \frac{1}{\delta}\lesssim 1.
\end{equation}
By explicit computations it is easy to see that
\begin{equation}
\begin{split}
\label{eq:2geq}
(1-M_1\langle\cdot\rangle M_2)\big(G_1B_1G_2-M_{12}^{B_1}\big)&=M_1B_1(G_2-M_2)-M_1\underline{WG_1B_1G_2} \\
&\quad+M_1\langle G_1B_1G_2\rangle(G_2-M_2)+M_1\langle G_1-M_1\rangle G_1B_1G_2,
\end{split}
\end{equation}
where
\[
\underline{WG_1B_1G_2}:=WG_1G_2+\langle G_1\rangle G_1B_1G_2+\langle G_1B_1G_2\rangle G_2.
\]
Taking the trace in \eqref{eq:2geq} against $B_2$, by the single resolvent local law \eqref{eq:singllaw0}, the norm bound \eqref{eq:normbound}, and the fact that $(1-M_1\langle\cdot\rangle M_2)^{-1}$ is bounded in the regime $\min_j \mathrm{dist}(z_j,\mathrm{supp}(\rho^{D_j}))\ge \delta$, we have
\begin{equation}
\langle (G_1B_1G_2-M_{12}^{B_1})B_2\rangle=-\langle M_1\underline{WG_1B_1G_2B_2}\rangle+\mathcal{O}_\prec\left(\frac{\lVert B_1\rVert\lVert B_2\rVert}{N}\right).
\end{equation}
Finally, using a minimalistic cumulant expansion as in \cite[(B.4)--(B.8)]{multiG}, we conclude $|\langle M_1\underline{WG_1B_1G_2B_2}\rangle|\prec N^{-1}\lVert B_1\rVert\lVert B_2\rVert$ and so \eqref{eq:avincond}.
\end{proof}

\subsection{Zig step: Propagating bounds via the characteristic flow} \label{subsec:charflowzig}
For Hermitian $D_j \in \C^{N \times N}$ with $\langle D_j\rangle=0$, spectral parameters $z_j\in \C\setminus\R$, $j=1,2$, and fixed $T>0$ the characteristic flow is defined by the following ODEs (see also \cite[(5.3)]{mesoCLT}): 
\begin{equation}
	\partial_t D_{j,t}:=-\frac{1}{2}D_{j,t},\qquad \partial_t z_{j,t} = -\langle M^{D_{j,t}}(z_{j,t})\rangle -\frac{z_{j,t}}{2},\qquad j=1,2\,,
	\label{eq:def_flow}
\end{equation}
with terminal conditions $D_{j,T} = D_j$ and $z_{j,T} = z_j$. 

We will often use the following short--hand notations:
\begin{equation*}
	M_{j,t}:=M^{D_{j,t}}(z_{j,t}), \quad \rho_{j,t}=\rho_{j,t}(z_{j,t}):=\frac{1}{\pi}\big|\langle \Im M^{D_{j,t}}(z_{j,t})\rangle\big|, \quad \eta_{j,t} := |\Im z_{j,t}|, \qquad j=1,2.
\end{equation*}

Even if our main results in Sections~\ref{sec:mainres} and \ref{sec:multigllaw} are presented only in the bulk of the spectrum, in the case of general observables we study the zig step uniformly in the spectrum, since this does not present significant additional difficulties (see Part 1 of Proposition~\ref{prop:Zig} below). For this reason, several definition in the remainder of this section will be presented uniformly in the spectrum.
In the zig step for regular observables and in the zag step for both types of observables we still restrict ourselves to the bulk.
\nc
\subsubsection{Preliminaries on the characteristic flow and admissible control parameters} \label{subsubsec:prelimchar}
Before formulating the fundamental building blocks for Step 2 of the zigzag strategy in Section \ref{subsubsec:propagating} below, we collect a few preliminaries concerning the characteristic flow and introduce \emph{admissible control parameters} $\gamma$, generalizing the concretely chosen $\widehat{\gamma}$ in \eqref{eq:def_gamma_0}. 

First, we  define a time--dependent version of the spectral domains on which we prove the local law from Theorem~\ref{theo:multigllaw} along the flow.

\begin{definition}[Spectral domains]\label{def:specdom} We define the time dependent spectral domains as follows: 
	\begin{itemize}
\item[(i)] \textnormal{[Unrestricted domains]} Fix a (small) $\epsilon>0$. For $j \in [2]$ define 
\begin{equation} \label{eq:specdom}
	\Omega_T^j :=\left\lbrace z\in\C\setminus\R: \vert\Im z\cdot \rho_{j,T}(z)\vert\ge N^{-1+\epsilon},\, \vert\Im z\vert\le N^{100}, \vert\Re z\vert\le N^{200}\right\rbrace.
\end{equation}
For $s,t\in [0,T]$ denote by $\mathfrak{F}_{t,s}^j$ the evolution operator along the flow \eqref{eq:def_flow}, i.e. $\mathfrak{F}_{t,s}^j(z_{j,s})=z_{j,t}$. Then we construct the family of \emph{unrestricted spectral domains} $\{\Omega_t^j\}_{t\in[0,T]}$, $j \in [2]$, by $\Omega_t^j:=\mathfrak{F}_{t,T}^j(\Omega_T^j)$.
\item[(ii)] \textnormal{[Bulk-restricted domains]} Fix additionally a (small) $\kappa > 0$ and recall \eqref{eq:kappabulk} for the definition of the $\kappa$-bulk $\mathbf{B}_\kappa(D)=\cup_{r=1}^{m}I_r$. Here $I_r=[a_r,b_r]\subset \R$ are closed non-intersecting intervals and $b_r<a_{r+1}$ for $r\in[1,m-1]$. We also denote $b_0:=-\infty$ and $a_{m+1}:=+\infty$. Then we define the family of \emph{bulk-restricted spectral domains} as
\begin{equation}
	\Omega_{\kappa,T}^j:= \Omega_T^j\setminus \left( \bigcup\limits_{r=0}^{m} \lbrace z\in\C\setminus \R: \Re z\in [b_r,a_{r+1}], \vert \Im z\vert \le \vert \Re z-a_{r+1}\vert\wedge\vert \Re z-b_r\vert \rbrace\right)
	\label{eq:specdom_bulk}
\end{equation}
and $\Omega_{\kappa,t}^j:=\mathfrak{F}_{t,T}^j(\Omega_{\kappa,T}^j)$ for $t \in [0,T]$. 
	\end{itemize}
\end{definition}

The bulk-restricted spectral domains $\Omega_{\kappa,t}^j$ are depicted in Figure \ref{fig:domains}.

\begin{figure}[h]
\begin{center}
\includegraphics[height=3cm]{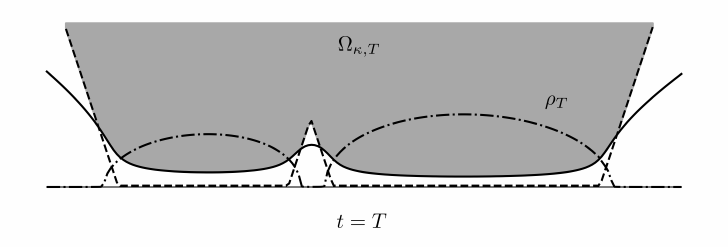}

\includegraphics[height=3cm]{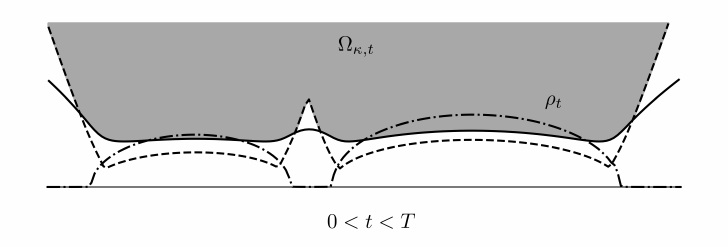}

\includegraphics[height=3cm]{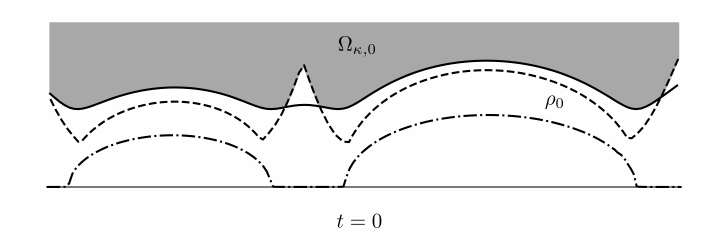}
\end{center}
\caption{In gray, we illustrated the $\Im z > 0$ part of the bulk-restricted spectral domains $\Omega_{\kappa,t}$ for three times, $t= 0$, $ t \in (0,T)$, and $t = T$ (the $\Im z < 0$ part is obtained by reflection). 
 On each of the panels, the graph of the density $\rho_t$ is superimposed in a dash-dotted style. The solid curve in the $t=T$ panel represents the implicitly defined curve $\vert\Im z\vert \rho(z)=N^{-1+\epsilon}$, above which one has the unrestricted domain $\Omega_T \supset \Omega_{\kappa, T}$. On the same $t = T$ panel, the region below the dashed curve is removed in the rhs.~of \eqref{eq:specdom_bulk}. For $t<T$ the solid and the dashed curves are the images of the corresponding curves at $t=T$ under the flow $\mathfrak{F}_{t,T}$.}
\label{fig:domains}
\end{figure}

Next, we state some trivially checkable properties of the characteristics flow \eqref{eq:def_flow}. Since Lemma \ref{lem:flow_properties}~(i) holds for $j=1$ and $j=2$, we drop the index $j$ in $z_j$, $D_j$, $\Omega_t^j$ and related quantities. In particular, we use the notation $z_t$ for $z_{j,t}$.
\begin{lemma}[Elementary properties of the characteristic flow]\label{lem:flow_properties} 
We have the following. 
\begin{itemize}
\item[(i)] Let $z_0\in\Omega_0$ be given. Then we have
\begin{itemize}
	\item[(1)] $M_t(z_t)=\ee^{t/2}M_0(z_0)$.
	\item[(2)] The map $t\mapsto\eta_t$ is monotone decreasing.
	\item[(3)] The solution to the second equation in \eqref{eq:def_flow} is explicitly given by
	\begin{equation}
		z_t = \ee^{-t/2}z_0-2\langle M_0(z_0)\rangle\sinh \frac{t}{2}.
		\label{eq:z_t}
	\end{equation}
	\item[(4)] For any $m>1$ and $t\in [0,T]$ we have
	\begin{equation}
		\int_0^t \frac{\rho_s}{\eta_s^m}\dif s \le \frac{1}{(m-1)\eta_t^{m-1}},\qquad\quad \int_0^t \frac{\rho_s}{\eta_s}\dif s\le \log \left(\frac{\eta_0}{\eta_t}\right).
		\label{eq:int_bounds}
	\end{equation} 
	\item[(5)] We have
	\begin{equation}
	\label{eq:linetarho}
	\frac{\eta_t}{\rho_t}=e^{s-t}\frac{\eta_s}{\rho_s}-\pi(1-e^{s-t}).
	\end{equation}
\end{itemize}
\item[(ii)] Let $z_j\in\Omega_0^j$ and $D_j=D_j^*\in\C^{N\times N}$ be given for $j \in [2]$. Denote $\nu_{j,t}:=(z_{j,t},D_{j,t})$ for $t\in [0,T]$ and assume that $\phi(\nu_{1,T},\nu_{2,T})=1$ (recall \eqref{eq:def_phi} for its definition). Let $A\in \C^{N\times N}$ be a regular observable with respect to $(\nu_{1,T},\nu_{2,T})$. Then $A$ is regular with respect to $(\nu_{1,t},\nu_{2,t})$ for any $t\in[0,T]$.
\end{itemize}
\end{lemma}

Notice that the error terms in Theorem~\ref{theo:multigllaw} are expressed in terms of the control parameter $\widehat{\gamma}$.  In Theorem~\ref{theo:multigllaw}, $\widehat{\gamma}$ is explicitly given, however, in order to make the argument more transparent, we collect in Definition \ref{def:gamma} all properties of $\widehat{\gamma}$ which are needed for the proof of Theorem \ref{theo:multigllaw}, arriving to the definition of an \emph{admissible control parameter}. In Proposition \ref{prop:gamma} we show that $\widehat{\gamma}$ on its own is an admissible control parameter. Further in Section~\ref{sec:zig} we work in this more general framework using a general admissible parameter $\gamma$ instead of $\widehat{\gamma}$.\nc

\begin{definition}[Admissible control parameter]
	\label{def:gamma}
Let $\gamma:\big( \C\setminus\R\big)^2\times \big(\C^{N \times N}\big)^2 \to (0,+\infty)$ be uniformly bounded in $N$ and assume that $\gamma(z_1,z_2,D_1,D_2)=\gamma(\bar{z}_1,z_2,D_1,D_2)$ and the same for $z_2 \to \bar{z}_2$. Moreover, for $t \in [0,T]$, let $\gamma_t : \Omega_0^1 \times \Omega_0^2 \times \big(\C^{N \times N}\big)^2 \to (0, \infty)$  with 
\begin{equation} \label{eq:gammat}
 \gamma_t(z_{1},z_{2},D_{1},D_{2}):=\gamma(z_{1,t},z_{2,t},D_{1,t},D_{2,t})
\end{equation}
 be the time-dependent version of $\gamma$, and $\beta_{*,t} : \Omega_0^1 \times \Omega_0^2 \times \big(\C^{N \times N}\big)^2 \to (0, \infty)$  with 
\begin{equation} \label{eq:beta*t}
\beta_{*,t}(z_{1},z_{2},D_{1},D_{2}):=\beta_*(z_{1,t},z_{2,t},D_{1,t},D_{2,t}).
\end{equation}	
the time-dependent version of $\beta_*$ (recall \eqref{eq:def_beta_*} and \eqref{eq:smallestev}). In \eqref{eq:gammat} and \eqref{eq:beta*t}, $z_{j,t}$ and $D_{j,t}$ are the solutions to \eqref{eq:def_flow} with $z_{j,0} = z_j$ and $D_{j,0} = D_j$ for $j \in [2]$. 

	Let $\mathfrak{D}_1, \mathfrak{D}_2 \subset \C^{N \times N}$ be $N$-dependent families of $N\times N$ Hermitian matrices. 	   We say that $\gamma$ is a \emph{$(\mathfrak{D}_1,\mathfrak{D}_2)$-admissible control parameter} if the following conditions hold uniformly in $D_{j}\in\mathfrak{D}_j$, $z_{j}\in\Omega^j_0$, $j\in[2]$, $t\in[0,T]$ and $N$:
	\begin{itemize}
		\item[(1)] \textnormal{[$\gamma$ is a lower bound on the stability operator]} It holds that
		\begin{equation}
			\left(\vert \Im z_{1,t}\vert/\rho_{1,t}(z_{1,t})+\vert \Im z_{2,t}\vert/\rho_{2,t}(z_{2,t})\right)\wedge 1\lesssim \gamma_t\lesssim \beta_{*,t}\,, 
			\label{eq:g_def}
		\end{equation}
	where both $\gamma_t$ and $\beta_{*,t}$ are evaluated at $(z_{1},z_{2},D_{1},D_{2})$.
		\item[(2)] \textnormal{[Monotonicity in time]} Uniformly in $0 \le s\le t \le T$, we have
		\begin{equation}
			\gamma_s(z_{1},z_{2},D_{1},D_{2})\sim \gamma_t(z_{1},z_{2},D_{1},D_{2})+t-s.
			\label{eq:time_monot}
		\end{equation}
		\item[(3)] \textnormal{[Vague monotonicity in imaginary part]} Uniformly in $z_1,z_2\in\HH$ and $x\in[0,\infty)$ it holds that
		\begin{equation}
		     \gamma(z_1,z_2,D_{1,t},D_{2,t})\lesssim \gamma(z_1,z_2+\ii x,D_{1,t},D_{2,t})\wedge \gamma(z_1+\ii x,z_2,D_{1,t},D_{2,t}).
			\label{eq:g_monot}
		\end{equation}
	\end{itemize}
\end{definition}
We now verify that $\widehat{\gamma}$ is an admissible control parameter in the sense of Definition \ref{def:gamma}. 
\begin{proposition}[Admissibility of $\widehat{\gamma}$]
	\label{prop:gamma}
	Fix $L, C_0>0$. Let $\mathfrak{D}$ be a set of all  traceless \nc $N\times N$ Hermitian matrices such that any $D\in\mathfrak{D}$ satisfies  \eqref{eq:M_bound_I} for $\mathbf{I}=\R$ \nc with constant $C_0$ and $\lVert D\rVert\le L$. Then $\widehat{\gamma}$ defined in \eqref{eq:def_gamma_0} is a $(\mathfrak{D},\mathfrak{D})$-admissible control parameter. 
\end{proposition}
The proof of Proposition \ref{prop:gamma} and a sufficient condition for $D$ to satisfy \eqref{eq:M_bound_I} for $\mathbf{I}=\R$ \nc are given in Section \ref{sec:cone} of the Supplementary Material \cite{supplement}.

As discussed around \eqref{eq:need_3G}, during the proof of Theorem \ref{theo:multigllaw} we need to handle resolvent products of the form $GBGBG$. More precisely, let $D_l\in\C^{N\times N}$ be Hermitian deformations and $z_l\in\C\setminus\R$ for $l\in [3]$. Denote $G_l:=(W+D_l-z_l)^{-1}$, $\nu_l:=(z_l,D_l)$ and $M_l:=M^{D_l}(z_l)$. We define the deterministic approximation of $G_1B_1G_2B_2G_3$ by (see \cite[Definition 4.1]{equipart})
\begin{equation}
M_{\nu_1,\nu_2,\nu_3}^{B_1,B_2}:=\mathcal{B}_{13}^{-1}\left[M_1B_1M_{\nu_2,\nu_3}^{B_2}+\langle M_{\nu_1,\nu_2}^{B_1}\rangle M_1M_{\nu_2,\nu_3}^{B_2}\right],
\end{equation}
where $\mathcal{B}_{13}$ is defined in \eqref{eq:stabop}, i.e.~$\mathcal{B}_{13}[\cdot]=1-M_1\langle\cdot\rangle M_3$. 
In the case when $\nu_l$ depend on additional parameter $t$, i.e. $\nu_l=\nu_l(t)$, $l\in[3]$, we adhere the analogue of the convention \eqref{eq:M2_t-dep} for $M_{\nu_1,\nu_2,\nu_3}^{B_1,B_2}$.  Namely, we use the shorthand notation
\begin{equation}
M_{123,t}^{B_1,B_2}:=M_{\nu_1(t),\nu_2(t),\nu_3(t)}^{B_1,B_2}.
\label{eq:M3t-dep}
\end{equation}

We are now ready to state bounds on the deterministic approximation of products of two and three resolvents:

\begin{proposition}[Bounds on $M$]\label{prop:norm_M_bounds} 
Fix $L>0$. Let $D_1,D_2\in \C^{N\times N}$ be Hermitian deformations with $\langle D_j\rangle=0$ and $\lVert D_j\rVert\le L$ for $j=1,2$. For spectral parameters $z_1,z_2\in\C\setminus\R$ denote the corresponding spectral pairs by $\nu_j=(z_j,D_j)$. Additionally we denote $\bar{\nu}_j:=(\bar{z}_j,D_j)$ and
\begin{equation*}
\ell(z_1,z_2):=\eta_1\rho_1(z_1)\wedge\eta_2\rho_2(z_2),\, \text{where}\,\, \eta_j=\vert\Im z_j\vert,\,\,\, \rho_j(z_j)=\pi^{-1}\vert\langle \Im M^{D_j}(z_j)\rangle\vert,\,\, j=1,2.\nc
\end{equation*}
\begin{itemize}
\item[Part 1:] \textnormal{[General case]}  Fix additionally $C_0>0$ and assume that $D_1, D_2$ satisfy \eqref{eq:M_bound_I} for $\mathbf{I}=\R$ with constant $C_0$. \nc Then uniformly in $z_1,z_2\in\C\setminus\R$ and deterministic matrices $B_1,B_2\in\C^{N\times N}$ it holds that
\begin{subequations}
\begin{align}
\left\lVert M_{\nu_1,\nu_2}^{B_1}\right\rVert\lesssim \frac{\lVert B_1\rVert}{\beta_*(z_1,z_2)},\label{eq:M2_gen}\\
\lVert M_{\nu_1,\nu_2,\nu_1}^{B_1,B_2}\rVert+\lVert M_{\nu_1,\nu_2,\bar{\nu}_1}^{B_1,B_2}\rVert \lesssim \frac{\lVert B_1\rVert\lVert B_2\rVert}{\ell(z_1,z_2)\beta_*(z_1,z_2)}.\label{eq:M3_gen}
\end{align}
\end{subequations}
Here the implicit constants depend only on $C_0$ and $L$.
\item[Part 2:] \textnormal{[Regular case]} Fix $\kappa>0$. Uniformly in $z_1,z_2\in \C\setminus\R$ with $\rho_j(z_j)\ge\kappa$ for $j=1,2$, in $(\nu_1,\nu_2)$-regular $A_1\in\C^{N\times N}$ and general $B_2\in\C^{N\times N}$ we have
\begin{subequations}
\begin{align}
\left\lVert M_{\nu_1,\nu_2}^{A_1}\right\rVert\lesssim \lVert A_1\rVert, \label{eq:M2_reg}\\
\lVert M_{\nu_1,\nu_2,\nu_1}^{A_1,B_2}\rVert+\lVert M_{\nu_1,\nu_2,\bar{\nu}_1}^{A_1,B_2}\rVert\lesssim \frac{\lVert A_1\rVert\lVert B_2\rVert}{\ell(z_1,z_2)\sqrt{\beta_*(z_1,z_2)}},\label{eq:M3_reg_gen}\\
\lVert M_{\nu_1,\nu_2,\nu_1}^{A_1,A_2}\rVert+\lVert M_{\nu_1,\nu_2,\bar{\nu}_1}^{A_1,A_2}\rVert\lesssim \frac{\lVert A_1\rVert\lVert A_2\rVert}{\ell(z_1,z_2)}.\label{eq:M3_reg_reg}
\end{align}
\end{subequations}
 We point out that $\ell\sim\eta_*$, since $z_1,z_2$ satisfy $\rho_i(z_i) \ge \kappa$. \nc The implicit constants in \eqref{eq:M2_reg}-\eqref{eq:M3_reg_reg} depend only on $L$ and $\kappa$, \eqref{eq:M3_reg_gen} also holds when the second observable is $(\nu_2,\nu_1)$-regular and the first one is general.
\end{itemize}
\end{proposition}  
The proof of Proposition \ref{prop:norm_M_bounds} is given in Section \ref{sec:norm_M_bounds} of the Supplementary Material \cite{supplement}.

\subsubsection{Propagating local law bounds} \label{subsubsec:propagating}

The general setting for propagating local law bounds is the following:
\begin{setting}[Zig step] \label{setting}
	Fix large constant $L>0$ and let $\mathfrak{D}_1,\mathfrak{D}_2$ be sets of $N\times N$ traceless Hermitian matrices such that $\lVert D\rVert\le L$ for any $D\in\mathfrak{D}_j$, $j\in[2]$.
	 Let $\gamma$ be a $(\mathfrak{D}_1,\mathfrak{D}_2)$-admissible control parameter as in Definition~\ref{def:gamma}.
	
	Fix a terminal time $T > 0$, let $z_{j,0} \in \Omega_0^j$, $D_{j,0} \in \mathfrak{D}_{j}$ for $j \in [2]$, and denote their time evolutions \eqref{eq:def_flow} by $z_{j,t} \in \Omega_t^j$ and $D_{j,t} \in \mathfrak{D}_j$, respectively. Moreover, let $s \in [0,T]$ be an initial time for the Ornstein-Uhlenbeck process. That is, for $t \in [s,T]$, let $W_t$ be the solution to \eqref{eq:OU},
		\begin{equation}
		\dif W_t = -\frac{1}{2}W_t\dif t +\frac{\dif B_{t-s}}{\sqrt{N}} \quad \text{with initial condition} \quad W_s = W\,.  %\qquad\quad W_0=W,
		\label{eq:OU}
	\end{equation}
	Here, $B_{t-s}$ is a real symmetric ($\beta =1 $) or complex Hermitian ($\beta=2$) matrix-valued Brownian motion with entries having variance equal to $(t-s)$ times those of a GOE/GUE matrix. Finally, we denote the resolvent of $W_t + D_{j,t}$ at $z_{j,t}$ by
	\begin{equation}\label{Gjt}
		G_{j,t}:=(W_t+D_{j,t}-z_{j,t})^{-1}
	\end{equation}
and introduce the abbreviations $\gamma_t$ from Definition~\ref{def:gamma}, and
	\begin{equation} \label{eq:etaelldef}
\eta_{*,t} := \eta_{1,t} \wedge \eta_{2,t}\wedge 1 \quad \text{and} \quad \ell_{t} := (\eta_{1,t} \rho_{1,t}) \wedge (\eta_{2,t} \rho_{2,t}) \,. 
	\end{equation}
\end{setting}

The following proposition formalizes the propagation of local laws along the evolution from Setting \ref{setting}. 

\begin{proposition}[Average and isotropic zig step for two and three resolvents]
\label{prop:Zig} 
In the Setting \ref{setting}, we have the following. 
\begin{itemize}
\item[Part 1:] \textnormal{[General case]} Consider Setting \ref{setting} with $\mathfrak{D}_1, \mathfrak{D}_2$ such that \eqref{eq:M_bound_I} is satisfied for $\mathbf{I}=\R$ with some constant $C_0$ for any matrix $D\in\mathfrak{D}_1\cup\mathfrak{D}_2$. \nc Assume that, for fixed initial time  $s \in [0,T]$, we have\footnote{The notation $G^{(*)}$ indicates both choices of adjoint $G^*$ and no adjoint $G$. }
\begin{subequations}
		\begin{align}
				\left\vert\left\langle \left(G_{1,s}B_1G_{2,s} - M_{12,s}^{B_1}\right) B_2\right\rangle\right\vert&\prec \left(\frac{1}{N\eta_{1,s}\eta_{2,s}}\wedge \frac{1}{\sqrt{N\ell_s}\gamma_s}\right)\lVert B_1\rVert\lVert B_2\rVert \,, 
			\label{eq:Zig_init} \\
	\left\vert \left\langle {\boldsymbol x},\left(G_{1,s}B_1G_{2,s} -M_{12,s}^{B_1}\right){\boldsymbol y}\right\rangle\right\vert &\prec \frac{1}{\sqrt{N\ell_s}}\cdot\frac{1}{\sqrt{\eta_{*,s} \gamma_s}}\lVert B_1\rVert \,,
	\label{eq:init_iso2} \\
		\left\vert \left\langle {\boldsymbol x},G_{1,s}B_1G_{2,s}B_2G^{(*)}_{1,s} {\boldsymbol y}\right\rangle\right\vert &\prec \frac{1}{\ell_s} \cdot \frac{1}{ \gamma_s}\lVert B_1\rVert\lVert B_2\rVert \,,
	\label{eq:init_iso3}
\end{align}
\end{subequations}
uniformly in $z_{j,s}\in\Omega_s^{j}$, $j \in [2]$, deterministic matrices $B_1,B_2$ and unit vectors $\boldsymbol{x},\boldsymbol{y}\in \C^N$. Then it holds that 
\begin{subequations}
\begin{align}
		\left\vert\left\langle \left(G_{1,t}B_1G_{2,t} - M_{12,t}^{B_1}\right) B_2\right\rangle\right\vert&\prec \left(\frac{1}{N\eta_{1,t}\eta_{2,t}}\wedge \frac{1}{\sqrt{N\ell_t}\gamma_t}\right)\lVert B_1\rVert\lVert B_2\rVert \,, 
	\label{eq:Zig}\\
	\left\vert \left\langle {\boldsymbol x},\left(G_{1,t}B_1G_{2,t} -M_{12,t}^{B_1}\right){\boldsymbol y}\right\rangle\right\vert &\prec \frac{1}{\sqrt{N\ell_t}} \cdot\frac{1}{\sqrt{\eta_{*,t} \gamma_t}}\lVert B_1\rVert \,, 
	\label{eq:2G_iso} \\
		\left\vert \left\langle {\boldsymbol x},G_{1,t}B_1G_{2,t}B_2G^{(*)}_{1,t} {\boldsymbol y}\right\rangle\right\vert &\prec \frac{1}{\ell_t} \cdot\frac{1}{ \gamma_t}\lVert B_1\rVert\lVert B_2\rVert \,, 
	\label{eq:3G_iso}
\end{align}
\end{subequations}
uniformly in $t\in [s,T]$, $z_{j,t}\in\Omega_t^{j}$, $j \in [2]$, matrices $B_1,B_2$ and unit vectors $\boldsymbol x, \boldsymbol y \in \C^N$.

\item[Part 2:] \textnormal{[Regular case]}  Assume that the result of Part 1 holds in $[0,T]$ and consider the slightly modified Setting \ref{setting} with $z_{j,0} \in \Omega_{\kappa, 0}^j$ for some  bulk parameter $\kappa > 0$. 
Assume that for fixed initial time $s \in [0,T]$, we have 
\begin{subequations}
	\begin{align}
	\left\vert\left\langle \left(G_{1,s}A_1 G_{2,s} - M_{12,s}^A\right) B\right\rangle\right\vert& \prec \left(\frac{1}{N\eta_{1,s}\eta_{2,s}}\wedge \frac{1}{\sqrt{N\ell_s\gamma_s}}\right)\lVert A_1\rVert\lVert B\rVert \,, 
			\label{eq:Zig_inithr} \\
					\left\vert\left\langle \left(G_{1,s}A_1G_{2,s} - M_{12,s}^{A_1}\right) A_2\right\rangle\right\vert &\prec \left(\frac{1}{N \eta_{1,s} \eta_{2,s}}\wedge  \frac{1}{\sqrt{N\ell_s}}\right)\lVert A_1\rVert\lVert A_2\rVert \,, 
		\label{eq:Zig_init_reg} \\
		\left\vert \left\langle {\boldsymbol x},\left(G_{1,s}A_1G_{2,s} -M_{12,s}^{A_1}\right){\boldsymbol y}\right\rangle\right\vert &\prec \frac{1}{\sqrt{N\ell_s}} \cdot\frac{1}{\sqrt{\eta_{*,s}}}\lVert A_1\rVert \,, 
		\label{eq:init_iso2_reg} \\
	\left\vert \left\langle {\boldsymbol x},G_{1,s}A_1G_{2,s}BG^{(*)}_{1,s} {\boldsymbol y}\right\rangle\right\vert & \prec \frac{1}{\ell_s} \cdot \frac{1}{\sqrt{\gamma_s}}\lVert A_1\rVert\lVert B\rVert \,,
	\label{eq:init_iso3hr} \\
				\left\vert \left\langle {\boldsymbol x},G_{1,s}A_1G_{2,s}A_2G^{(*)}_{1,s} {\boldsymbol y}\right\rangle\right\vert &\prec \frac{1}{\ell_s}\lVert A_1\rVert\lVert A_2\rVert  \,, 
		\label{eq:init_iso3_reg}
	\end{align}
\end{subequations}
	uniformly in $z_{j,s}\in\Omega_{\kappa,s}^{j}$, $j \in [2]$, deterministic $B\in \C^{N\times N}$, $(\pair_1, \pair_2)$-regular  $A_1$ and $(\pair_2, \pair_1)$-regular $A_2$ and unit vectors $\boldsymbol{x},\boldsymbol{y}\in \C^N$. Then it holds that 
	\begin{subequations}
	\begin{align}
	\left\vert\left\langle \left(G_{1,t}A_1G_{2,t} - M_{12,t}^A\right) B\right\rangle\right\vert&\prec \left(\frac{1}{N\eta_{1,t}\eta_{2,t}}\wedge \frac{1}{\sqrt{N\ell_t\gamma_t}}\right)\lVert A_1\rVert\lVert B\rVert \,, 
	\label{eq:Zighr}\\
					\left\vert\left\langle \left(G_{1,t}A_1G_{2,t} - M_{12,t}^{A_1}\right) A_2\right\rangle\right\vert&\prec \left(\frac{1}{N\eta_{1,t}\eta_{2,t}}\wedge \frac{1}{\sqrt{N\ell_t}}\right)\lVert A_1\rVert\lVert A_2\rVert \,, 
		\label{eq:Zig_reg} \\
		\left\vert \left\langle {\boldsymbol x},\left(G_{1,t}A_1G_{2,t} -M_{12,t}^{A_1}\right){\boldsymbol y}\right\rangle\right\vert &\prec \frac{1}{\sqrt{N\ell_t}} \cdot\frac{1}{\sqrt{\eta_{*,t}}}\lVert A_1\rVert \,, 
		\label{eq:2G_iso_reg} \\
		\left\vert \left\langle {\boldsymbol x},G_{1,t}A_1G_{2,t}BG^{(*)}_{1,t} {\boldsymbol y}\right\rangle\right\vert & \prec \frac{1}{\ell_t} \cdot\frac{1}{\sqrt{\gamma_t}}\lVert A_1\rVert\lVert B\rVert \,, 
	\label{eq:3G_isohr} \\
				\left\vert \left\langle {\boldsymbol x},G_{1,t}A_1G_{2,t}A_2G^{(*)}_{1,t} {\boldsymbol y}\right\rangle\right\vert &\prec \frac{1}{\ell_t}\lVert A_1\rVert\lVert A_2\rVert  \,, 
		\label{eq:3G_iso_reg}
	\end{align}
\end{subequations}
uniformly in $t\in [s,T]$, $z_{j,t}\in\Omega_{\kappa,t}^{j}$, $j \in [2]$, deterministic $B\in \C^{N\times N}$, $(\pair_1, \pair_2)$-regular  $A_1$ and $(\pair_2, \pair_1)$-regular $A_2$ and unit vectors $\boldsymbol x, \boldsymbol y \in \C^N$.
\end{itemize}
\end{proposition}

Note that while the case of general observables is self--contained (i.e. it does not require any information about regular observables), the cases of one or two regular observables have to be done in tandem. In fact, when computing the quadratic variation for the stochastic term in \eqref{eq:dif_g} for traces with only one regular observable one gets a trace with two regular observables. On the other hand, the case of two regular observables is not self--contained either, because of the $\mathrm{Lin}_t$ term in \eqref{eq:dif_g}.

\subsection{Zag step: Removing Gaussian components via a GFT} \label{subsec:GFTzag}
As already explained below \eqref{eq:def_flow}, from now on we constrain the argument to the \emph{bulk}, i.e. we assume the spectral parameters to be in bulk-restricted domains $\Omega_{\kappa,t}^j$, 
where it holds that $\ell \sim \eta_*$. For ease of notation, we shall also write $\eta \equiv \eta_*$. 

The general setting for removing a Gaussian component in Lemmas \ref{lem:uncondGron2iso}--\ref{lem:condGronreg} is the following. 
\begin{setting}[Zag step] \label{settingzag}
		Fix a large constant $L>0$ and let $\mathfrak{D}_1,\mathfrak{D}_2$ be sets of $N\times N$ traceless Hermitian matrices such that any $D \in\mathfrak{D}_j$, $j\in[2]$, satisfies $\lVert D\rVert\le L$. Let $\gamma$ be a $(\mathfrak{D}_1,\mathfrak{D}_2)$-admissible control parameter as in Definition~\ref{def:gamma}.
	
	Fix some $\kappa > 0$ (bulk parameter) and a terminal time $T > 0$, let $z_{j,0} \in \Omega_{\kappa, 0}^j$, $D_{j,0} \in \mathfrak{D}_{j}$ for $j \in [2]$, and denote their time evolutions \eqref{eq:def_flow} by $z_{j,t} \in \Omega_{\kappa, t}^j$ and $D_{j,t} \in \mathfrak{D}_j$, respectively. 
	Now, take two fixed times $s, t \in [0,T]$ with $s \le t$ and consider the Ornstein-Uhlenbeck process
	\begin{equation*}
		\dif W_r = -\frac{1}{2}W_r\dif r +\frac{\dif B_{r-s}}{\sqrt{N}} \quad \text{with initial condition} \quad W_s = W
	\end{equation*}
	for times  $r \in [s,t]$. 
 Finally, we denote the resolvent of $W_r + D_{j,t}$ at $z_{j,t}$ (note that the $t$ index is fixed!) by
	\begin{equation}
		G_{j,r}:=(W_r+D_{j,t}-z_{j,t})^{-1}\,. 
	\end{equation}
	The times $s,t$ and hence, in particular, the spectral parameters $z_{j,t} \in \Omega_{\kappa, t}^j$ remain fixed through the Lemmas~\ref{lem:uncondGron2iso}--\ref{lem:condGronreg} below. Thus, dropping the time arguments, we denote $\eta_j = |\Im z_{j,t}|$, $\eta := \min_j \eta_j$, and $\gamma = \gamma_t(z_1, z_2, D_1, D_2)$. 
\end{setting}

Contrary to the the \emph{zig} step in Section \ref{subsec:charflowzig}, where all three local law bounds (two resolvent average, two and three resolvent isotropic) are propagated together (cf.~Proposition \ref{prop:Zig}), the Gronwall estimates needed to remove the Gaussian component will be done separately in a carefully chosen order. More precisely, we begin with an \emph{unconditional} Gronwall estimate for isotropic two resolvents, see Lemma \ref{lem:uncondGron2iso}. We call it unconditional, because the differential inequality obtained in \eqref{eq:uncondGron2iso} does not require any input. The differential inequalities \eqref{eq:condGron2iso}, \eqref{eq:condGron3iso}, and \eqref{eq:condGron2av} in Lemmas \ref{lem:condGron2iso}--\ref{lem:condGron2av}, however, require certain inputs, which are obtained from integrating the differential inequalities in time. \nc Since Lemmas \ref{lem:condGron2iso}--\ref{lem:condGron2av} require inputs, we call them \emph{conditional} Gronwall estimates. Moreover, we point out that the proof of the two resolvent and three resolvent isotropic bounds contain an internal recursion. In fact, Lemmas \ref{lem:condGron2iso}--\ref{lem:condGron3iso} are used several times to gradually improve the bound (the exponent $b$ is improved to $b'$). Finally, in Lemma~\ref{lem:condGronreg} we explain how the conditional Gronwall estimates change in case of \emph{regular observables}.

All estimates in Lemmas \ref{lem:uncondGron2iso}--\ref{lem:condGronreg} hold \emph{uniformly} in all spectral parameters $z_{j,t} \in \Omega_{\kappa, t}^j$ for the fixed time $t$. 
\subsubsection{Unconditional Gronwall estimate for isotropic two-resolvent chains}
We begin with an unconditional Gronwall estimate for isotropic two resolvents. 

\begin{lemma}[Unconditional Gronwall estimate for isotropic two resolvents] \label{lem:uncondGron2iso} Let $\boldsymbol x, \boldsymbol y \in \C^N$ be bounded and set  
	\begin{equation} \label{eq:unGFT1isoinput}
		R_r := \left| \big( G_{1,r} B_1 G_{2,r} - M_{12}^{B_1}\big)_{\boldsymbol x \boldsymbol y} \right|  \quad \text{and} \quad \mathcal{E}_0 := \frac{1}{\sqrt{N \eta}} \frac{1}{\eta} \,,
	\end{equation}
	Then, for $p \in \N$ and any $\xi > 0$, we have that 
	\begin{equation} \label{eq:uncondGron2iso}
		\frac{\dif }{\dif r} \E |R_r|^{2p} \lesssim \left(1 + \frac{1}{\sqrt{N \eta} \, \eta}\right) \, \left( \E |R_r|^{2p} + N^\xi \mathcal{E}_0^{2p}\right)\,. 
	\end{equation}
\end{lemma}
The proof of Lemma \ref{lem:uncondGron2iso} is given in Section \ref{sec:zag}. 
\subsubsection{Conditional Gronwall estimates: general case} In this section, we collect our conditional Gronwall estimates for general observables. The initial input, i.e.~\eqref{eq:GFT1isoinput} for $b=0$, will be obtained from integrating \eqref{eq:uncondGron2iso} in time. A similar approach was introduced in parallel in \cite{nonHermdecay}. 

\begin{lemma}[Conditional Gronwall estimate for isotropic two resolvents] \label{lem:condGron2iso}Assume that for some fixed  $b \in [0,1]$ it holds that 
	\begin{equation} \label{eq:GFT1isoinput}
		R_r := \left| \big( G_{1,r} B_1 G_{2,r} - M_{12}^{B_1}\big)_{\boldsymbol x \boldsymbol y} \right| \prec \mathcal{E}_0 \quad \text{with} \quad \mathcal{E}_0 := \frac{1}{\sqrt{N \eta}} \frac{1}{\eta^{1-b/2} \gamma^{b/2}} \,,
	\end{equation}
	uniformly in bounded $\boldsymbol x, \boldsymbol y \in \C^N$ and $r \in [s,t]$. 
	Then, for $p \in \N$ and any $\xi > 0$, we have that 
	\begin{equation} \label{eq:condGron2iso}
		\frac{\dif }{\dif r} \E |R_r|^{2p} \lesssim \left(1 + \frac{1}{\sqrt{N \eta} \, \eta}\right) \, \left( \E |R_t|^{2p} + N^\xi \mathcal{E}_1^{2p}\right)\,, 
	\end{equation}
	where we denoted
	\begin{equation}
		\mathcal{E}_1 := \frac{1}{\sqrt{N \eta}} \frac{1}{\eta^{1-b'/2} \gamma^{b'/2}} \quad \text{with} \quad b':= (b+1/3) \wedge 1 \,. 
	\end{equation}
\end{lemma}

The proof of Lemma \ref{lem:condGron2iso} is given in Section \ref{sec:zag}. The conditional Gronwall estimate concerning isotropic three resolvents is given in the following lemma. 

\begin{lemma}[Conditional Gronwall estimate for isotropic three resolvents] \label{lem:condGron3iso}
	Assume that for some fixed $b \in [0,1]$ it holds that 
	\begin{equation} \label{eq:GFT2isoinput}
		R_r := \left| \big( G_{1,r} B_1 G_{2,r} B_2 G_{1,r}\big)_{\boldsymbol x \boldsymbol y} \right| \prec \mathcal{E}_0 \quad \text{with} \quad \mathcal{E}_0 := \frac{1}{\eta} \frac{1}{\eta^{1-b} \gamma^{b}} \,,
	\end{equation}
	uniformly bounded $\boldsymbol x, \boldsymbol y \in \C^N$ and $r \in [s,t]$. Moreover, suppose that 
	\begin{equation} \label{eq:3Gisoass}
		\left| \big( G_{1,r} B_1 G_{2,r} - M_{12}^{B_1}\big)_{\boldsymbol x \boldsymbol y} \right| \prec \frac{1}{\sqrt{N \eta} \, \eta^{1-b'/2} \gamma^{b'/2}} \quad \text{with} \quad b':= (b+1/3) \wedge 1  \,, 
	\end{equation}
	uniformly in $r \in [s,t]$, bounded $\boldsymbol x, \boldsymbol y \in \C^N$, and $B_1 \in \C^{N \times N}$ (and the same for indices $1$ and $2$ interchanged).  
	Then, for $p \in \N$ and any $\xi > 0$, we have that 
	\begin{equation} \label{eq:condGron3iso}
		\frac{\dif }{\dif r} \E |R_t|^{2p} \lesssim \left(1 + \frac{1}{\sqrt{N \eta} \, \eta}\right) \, \left( \E |R_r|^{2p} + N^\xi \mathcal{E}_1^{2p}\right)\,, 
	\end{equation}
	where we denoted
	\begin{equation} \label{eq:E13iso}
		\mathcal{E}_1 := \frac{1}{\eta} \frac{1}{\eta^{1-b'} \gamma^{b'}} \,. 
	\end{equation}
\end{lemma}

	The proof of Lemma \ref{lem:condGron3iso} is completely analogous to that of Lemma \ref{lem:condGron2iso} and so omitted.

We point out that the input bound \eqref{eq:GFT2isoinput} with $b = 0$ is trivially satisfied since (neglecting the time dependence)
\begin{equation} \label{eq:3Gisotrivial}
	\left| \big(G_1 B_1 G_2 B_2 G_1\big)_{\boldsymbol x \boldsymbol y} \right| \le \frac{\Vert B_2 \Vert}{\eta} \sqrt{\big(G_1 B_1 \Im G_2 B_1^* G_1^*\big)_{\boldsymbol x \boldsymbol x} \big(\Im G_1\big)_{\boldsymbol y \boldsymbol y}} \prec \frac{1}{\eta^2}
\end{equation}
by a simple Schwarz inequality together with Ward identities, the trivial bound $\Vert G \Vert \le \eta^{-1}$, and a single resolvent local law giving $|G_{\boldsymbol u \boldsymbol v}| \prec 1$ for $\boldsymbol u, \boldsymbol v$ of bounded norm. The other input \eqref{eq:GFT2isoinput} will be obtained by integrating the differential inequality \eqref{eq:condGron2iso} from Lemma \ref{lem:condGron2iso}. 

The time integrated versions of the differential inequalities from Lemmas \ref{lem:condGron2iso}--\ref{lem:condGron3iso} both serve as inputs for the following lemma concerning average two resolvents. 
\nc
\begin{lemma}[Conditional Gronwall estimate for average two resolvents]
	\label{lem:condGron2av}
	Assume that
	\begin{equation} \label{eq:GFT2avinput}
		\left| \big( G_{1,r} B_1 G_{2,r} - M_{12}^{B_1}\big)_{\boldsymbol x \boldsymbol y} \right| \prec \frac{1}{\sqrt{N \eta} \, \eta^{1/2} \gamma^{1/2}} \quad \text{and} \quad 	\left| \big( G_{1,r} B_1 G_{2,r}B_2 G_{1,r} \big)_{\boldsymbol x \boldsymbol y} \right| \prec \frac{1}{\eta\, \gamma}
	\end{equation}
	uniformly in $r \in [s,t]$, bounded $\boldsymbol x, \boldsymbol y \in \C^N$, and $B_1, B_2 \in \C^{N \times N}$.  
	Then, defining
	\begin{equation*} 
		R_t := \left| \big\langle \big(G_{1,t} B_1 G_{2,t} -M_{12}^{B_1}\big)B_2\big\rangle  \right| \,,
	\end{equation*}
	for $p \in \N$ and any $\xi > 0$, we have that 
	\begin{equation} \label{eq:condGron2av}
		\frac{\dif }{\dif r} \E |R_r|^{2p} \lesssim \left(1 + \frac{1}{\sqrt{N} \, \eta}\right) \, \left( \E |R_r|^{2p} + N^\xi \mathcal{E}_1^{2p}\right)\,, \quad \text{where} \quad \mathcal{E}_1 := \frac{1}{N \eta_1 \eta_2} \wedge \frac{1}{\sqrt{N\eta} \, \gamma} \,. 
	\end{equation}
\end{lemma}

The proof of Lemma \ref{lem:condGron2av} is given in Section \ref{sec:zag}.

\subsubsection{Conditional Gronwall estimates: regular case}

For \emph{regular observables}, the desired local law enjoys a further improvement in accordance with the  $\sqrt{\gamma}$-rule (see the discussion in Section \ref{sec:intro}) 
for such observables. In order to remove the Gaussian component introduced in the characteristic flow step, we again employ conditional Gronwall estimates. 

\begin{lemma}[Conditional Gronwall estimates for regular observables] \label{lem:condGronreg}
	Let $A_1, A_2 \in \C^{N \times N}$ be bounded matrices and assume that $A_1$ is $(\pair_1, \pair_2)$-regular and $A_2$ is $(\pair_2, \pair_1)$-regular. Then we have the following: 
	\begin{itemize}
		\item[(i)] Upon replacing $B_1 \to A_1$ and $\gamma \to 1$, Lemma \ref{lem:condGron2iso} holds verbatim.
		\item[(ii)] Upon replacing $B_i \to A_i$, for $i \in [2]$, and $\gamma \to 1$, Lemma \ref{lem:condGron3iso} holds verbatim. 
		
		Moreover, in case that only \emph{one} of the general observables $B_i$ is replaced by a regular one $A_i$, and the assumption \eqref{eq:3Gisoass} is suitably adjusted (namely replacing $\gamma \to 1$ only for the case with a regular observable), Lemma \ref{lem:condGron3iso} holds with $\gamma \to \sqrt{\gamma}$ in the definition of $\mathcal{E}_0$ and  $\mathcal{E}_1$ in \eqref{eq:GFT2isoinput} and \eqref{eq:E13iso}, respectively. 
		\item[(iii)] Upon replacing $B_i \to A_i$, for $i \in [2]$, and $\gamma \to 1$, Lemma \ref{lem:condGron2av} holds verbatim. 
		
		Moreover, in case that only \emph{one} of the general observables $B_i$ is replaced by a regular one $A_i$, and the assumption \eqref{eq:GFT2avinput} is suitably adjusted (as described in item (iii) above), the conclusion \eqref{eq:condGron2av} holds with $\gamma \to \sqrt{\gamma}$. 
	\end{itemize} 
\end{lemma} 
\begin{proof}
	The proof of Lemma \ref{lem:condGronreg} works in the exact same way as the proofs of Lemmas \ref{lem:condGron2iso}--\ref{lem:condGron2av}, with the only difference that the bound \eqref{eq:Mboundcond} gets complemented by the improved estimates
	\begin{equation} \label{eq:Mboundcondimproved}
		\big\Vert M_{12}^{A_1} \Vert \lesssim \Vert A_1 \Vert  \quad \text{and} \quad \big\Vert M_{21}^{A_2} \Vert \lesssim \Vert A_2 \Vert
	\end{equation}
	from Proposition \ref{prop:norm_M_bounds} (note that there is no $\gamma^{-1}$ on the rhs.~of \eqref{eq:Mboundcondimproved}). The rest of the argument is identical. 
\end{proof}

\subsection{Conclusion of the zigzag strategy: Proof of Theorem \ref{theo:multigllaw}} \label{subsec:zigzagconcl}

We start with the following trivially checkable lemma, which follows by standard ODE theory and \eqref{eq:linetarho}.

\begin{lemma}[Initial conditions] \label{lem:ODEbasic}
Fix $0\le T<1$, and pick a spectral parameter $|z|\lesssim 1$ and a matrix $\lVert D\rVert\lesssim 1$. Then there exist initial conditions $z_0,D_0$ such that the solutions $z_t,D_t$ of \eqref{eq:def_flow}, with initial conditions $z_0,D_0$, satisfies $z_T=z$ and $D_T=D$. Additionally, we have $\mathrm{dist}(z_0,\mathrm{supp}(\rho_{D_0}))\ge c T$, for some universal constant $c>0$.
\end{lemma}

Along the proof of Theorem \ref{theo:multigllaw}, we will also prove the following proposition. 

\begin{proposition}[Isotropic two- and three-resolvent local laws] \label{prop:isoLL}
Fix $L, C_0,\epsilon > 0$. Let $W$ be a Wigner matrix satisfying Assumption~\ref{ass:momass}, and let $D_1,D_2\in\C^{N\times N}$ be bounded Hermitian matrices. 
For spectral parameters $z_1, z_2 \in \C \setminus \R$, denote  $\eta_l:=|\Im z_l|$,  $\rho_l:=\pi^{-1}|\langle \Im M_l\rangle|$, and $\ell := \min_{l \in [2]} \eta_l \rho_l$. Finally, let $\widehat{\gamma} = \widehat{\gamma}(z_1, z_2)$ be defined as in \eqref{eq:def_gamma_0}. Then, the following holds:
\begin{itemize}
\item[\textnormal{Part 1:}] \textnormal{[General case]} For bounded $B_1,B_2\in \C^{N\times N}$ and unit ${\boldsymbol x}, {\boldsymbol y}\in\C^N$, we have
\begin{subequations}
\begin{align}
	\left\vert \left\langle {\boldsymbol x},\left(G_{1}B_1G_{2} -M_{z_{1},z_{2}}^{B_1}\right){\boldsymbol y}\right\rangle\right\vert &\prec \frac{1}{\sqrt{N\ell}} \cdot\frac{1}{\sqrt{\eta_* \gamma}} \,, 
\label{eq:2G_iso_final} \\
\left\vert \left\langle {\boldsymbol x},G_{1}B_1G_{2}B_2G^{(*)}_{1} {\boldsymbol y}\right\rangle\right\vert &\prec \frac{1}{\ell \gamma} \,, 
\label{eq:3G_iso_final}
\end{align}
\end{subequations}
uniformly in spectral parameters satisfying $|z_1|,|z_2|\le N^{100}$ and $N\ell\ge N^\epsilon$.
\item[\textnormal{Part 2:}] \textnormal{[Regular case]} Recall \eqref{eq:regulardef}, let $A_1\in \C^{N\times N}$ be $(\nu_1,\nu_2)$--regular and let $A_2\in \C^{N\times N}$ be $(\nu_2,\nu_1)$--regular. Then, for bounded $A_1, A_2, B\in \C^{N\times N}$ and unit ${\boldsymbol x}, {\boldsymbol y}\in\C^N$, we have
\begin{subequations}
	\begin{align}
		\left\vert \left\langle {\boldsymbol x},\left(G_{1}A_1G_{2} -M_{z_{1},z_{2}}^{A_1}\right){\boldsymbol y}\right\rangle\right\vert &\prec \frac{1}{\sqrt{N\ell}} \cdot\frac{1}{\sqrt{\eta_*}} \,, 
		\label{eq:2G_iso_finalreg} \\
		\left\vert \left\langle {\boldsymbol x},G_{1}A_1G_{2}BG^{(*)}_{1} {\boldsymbol y}\right\rangle\right\vert &\prec \frac{1}{\ell\sqrt{\gamma}}  \,, 
		\label{eq:3G_iso_finalregnew} \\
		\left\vert \left\langle {\boldsymbol x},G_{1}A_1G_{2}A_2G^{(*)}_{1} {\boldsymbol y}\right\rangle\right\vert &\prec \frac{1}{\ell}  \,, 
		\label{eq:3G_iso_finalreg}
	\end{align}
\end{subequations}
uniformly in spectral parameters satisfying $|z_1|,|z_2|\le N^{100}$ and $N\ell\ge N^\epsilon$.
\end{itemize}

\end{proposition}

\subsubsection{General case: Proof of Part 1 of Theorem \ref{theo:multigllaw} and Proposition \ref{prop:isoLL}}
Fix a bulk parameter $\kappa > 0$ and $\epsilon > 0$. For $j \in [2]$, we now define sequences of domains in the following way: 
Consider the monotonically increasing sequence $(a_k)_{k \in \N_0} \subset [0,1]$ defined recursively as
\begin{equation}
	a_{k+1} := \frac{2}{3} a_k + \frac{1}{3} \quad \text{with} \quad a_0 = 0 \,. 
\end{equation}
Moreover, set 
\begin{equation*}
\eta_k := N^{-a_k}
\end{equation*}
and let $K \in \N$ be the smallest integer satisfying $\eta_K < N^{-1+\epsilon}$ (note that $K = O(|\log \epsilon|)$ is independent of $N$). By Lemma \ref{lem:ODEbasic}, choose the terminal time $T> 0$ in such a way that 
\begin{equation}
\Omega_{\kappa, 0}^j \subset \{z \in \C : \vert\Im z\vert \ge c\} \quad \text{for} \quad j \in [2] \,. 
\label{eq:def_T}
\end{equation}
Here, $c > 0$ depends only on $L$ and $\kappa$ via Lemma \ref{lem:ray}~(ii). 
Next, let $(t_k)_{k=0}^K \subset [0,T]$ be monotonically increasing sequence of times with $t_0 = 0$, $t_K = T$ and, for  $k \in [K-1]$, we define $t_k$ as the largest time in $[0,T]$ satisfying 
\begin{equation}
\Omega_k^j := \Omega_{\kappa, t_k} \subset \{z \in \C : \vert\Im z\vert \ge \eta_k\} \quad \text{for} \quad j \in [2] \,. 
\end{equation}

After having set up these sequences of domains, the key for proving the target local laws is the following \emph{induction argument}, which we prove below. 
\begin{proposition}[Induction on scales] \label{lem:induction}
Assume that the local laws \eqref{eq:g1g2b} and \eqref{eq:2G_iso_final}--\eqref{eq:3G_iso_final} hold uniformly on $\Omega_k^j$ for the deformed Wigner matrices $W+ D_{j, t_k}$. Then they also hold uniformly on $\Omega_{k+1}^j$ for the deformed Wigner matrices $W+ D_{j, t_{k+1}}$. 
\end{proposition}
The input for $k=0$ is ensured by the global law in Proposition \ref{prop:global}. Then, applying Proposition~\ref{lem:induction} in total $K$ times, we arrive at Part 1 of Theorem \ref{theo:multigllaw} and Proposition \ref{prop:isoLL}. 

\begin{proof}[Proof of Proposition \ref{lem:induction}]
Given the assumption in Proposition \ref{lem:induction}, we find from Proposition \ref{prop:Zig} with $s = t_k$ and $t = t_{k+1}$, the local laws to hold on $\Omega_{k+1}^j$ at the cost of having introduced a Gaussian component of order $t_{k+1}-t_k$. We now remove this Gaussian component in several steps. Here, we will frequently employ Gronwall's Lemma to integrate the differential inequalities \eqref{eq:uncondGron2iso}, \eqref{eq:condGron2iso}, \eqref{eq:condGron3iso}, and \eqref{eq:condGron2av}, and thereby use that (by construction)
\begin{equation}
t_{k+1} - t_k \lesssim 1 \quad \text{and} \quad \frac{t_{k+1} - t_k}{\sqrt{N} |\Im z|^{3/2}} \lesssim 1 \quad \text{uniformly for} \quad z \in \Omega_{k+1}^j\,, \ j \in [2]\,. 
\end{equation}
The steps are as follows: 
\begin{itemize}
\item[1.] With the aid of Lemma \ref{lem:uncondGron2iso}, integrating \eqref{eq:uncondGron2iso} ending at $t = t_{k+1}$, we infer \eqref{eq:GFT1isoinput} with $b = 0$ and $s = t_k$. 
\item[2.] By Lemma \ref{lem:condGron2iso}, integrating \eqref{eq:condGron2iso} ending at $t = t_{k+1}$, we infer \eqref{eq:3Gisoass} for $b=0$ (i.e.~$b' = 1/3$) and $s = t_k$. 
\item[3.] By Lemma \ref{lem:condGron3iso} (and using \eqref{eq:3Gisotrivial}), integrating \eqref{eq:condGron3iso} ending at $t = t_{k+1}$, we obtain \eqref{eq:GFT2isoinput} with $b= 1/3$. 
\item[4.] In order to improve the exponent $b$, repeat steps 2 and 3 for two more times, giving us \eqref{eq:GFT2isoinput}--\eqref{eq:3Gisoass} for $b=1$ and $s=t_k$, $t= t_{k+1}$. That, is we proved \eqref{eq:2G_iso_final}--\eqref{eq:3G_iso_final} to hold on $\Omega_{k+1}^j$. 
\item[5.] Finally, by application of Lemma \ref{lem:condGron2av} (note that \eqref{eq:GFT2avinput} is obtained in Step 4), we integrate \eqref{eq:condGron2av} ending at $t = t_{k+1}$ to infer \eqref{eq:g1g2b} to hold on $\Omega_{k+1}^j$. 
\end{itemize}
This concludes the proof of Proposition \ref{lem:induction}.
\end{proof}

\subsubsection{Regular case: Proof of Part 2 of Theorem \ref{theo:multigllaw} and Proposition \ref{prop:isoLL}}

The proof of Theorem \ref{theo:multigllaw} for regular observables (Part 2) follows very similar steps to those in the proof of Part 1, with the only exception that the local laws for chains with one and two regular observables have to be propagated together. We thus omit this proof for the sake of brevity.
\qed

\newcommand{\com}{\color{magenta}}

\section{Zig step: Proof of Proposition \ref{prop:Zig}} \label{sec:zig}

In the current section we present the proof of Proposition \ref{prop:Zig}. Firstly we do the zig step for average two-resolvent chains in Section \ref{subsec:Zig_2Gav}. This is done self-consistently, i.e. without involving isotropic quantities
or longer chains. However the single resolvent local law is used, which states that for any fixed $\zeta>0$ and for any $z\in\C\setminus\R$ such that $N\vert\Im z\vert\rho(z)\ge N^\zeta$, it holds that
\begin{equation}
\label{eq:singllaw}
\big|\langle (G(z)-M(z))A\rangle\big|\prec \frac{1}{N|\Im z|}, \qquad\quad \big|\langle{\boldsymbol x}, (G(z)-M(z)){\boldsymbol y}\rangle\big|\prec \sqrt{{\frac{\rho}{N|\Im z|}}}.
\end{equation}
Note that \eqref{eq:singllaw} coincides with \eqref{eq:singllaw0} when $\Re z$ is in the bulk, however \eqref{eq:singllaw} is more general since it is uniform in the spectrum.  The local law \eqref{eq:singllaw} was proven near the edge in \cite{SlowCor} and was later extended to the cusp regime in \cite{CuspUniv}. In fact, for the proof of Proposition \ref{prop:Zig} we do not need \eqref{eq:singllaw} itself, but just a weaker statement that \eqref{eq:singllaw} propagates along the zig flow, which can be directly proven by the methods described below in Section~\ref{subsec:Zig_2Gav}. Thus our proof can be easily made independent of \cite{SlowCor, CuspUniv}, but for simplicity in the current presentation we will rely on them as they are already available.

Later in Section \ref{subsec:Zig_iso} we work with isotropic two- and three-resolvent chains and prove \eqref{eq:2G_iso}, \eqref{eq:3G_iso} relying on the result of Section \ref{subsec:Zig_2Gav}. Finally, in Section \ref{subsec:Zig_regular} we explain how the proofs of \eqref{eq:Zig}--\eqref{eq:3G_iso} should be modified in the setting when one or several of observables are regular in the sense of Definition~\ref{def:regulardef}.

Throughout the entire section we will assume without loss of generality that all matrices $A_j,B_j$, $j=1,2$ are bounded in operator norm by 1, i.e. $\lVert A_j\rVert\le 1$, $\lVert B_j\rVert\le 1$. Also by ${\boldsymbol x}, {\boldsymbol y}\in\C^N$ we will mean unit vectors. Moreover, for simplicity we present the proof for $s=0$ and $t=T$. 
To keep the presentation short we often omit the time dependence in $G_{j,s}$ and simply write $G_j$. That is, we use the shorthand notation
\begin{equation*}
G_j = (W_s+D_{j,s}-z_{j,s})^{-1},\quad j\in [1,2],
\end{equation*}
whenever the time $s$ is clear from the context. For all other time dependent variables, such as $z_{j,s}$, $D_{j,s}$, and $\ell_s$, we keep the time dependence explicitly.

\subsection{Average two-resolvent chains: Proof of \eqref{eq:Zig} in Proposition \ref{prop:Zig}}\label{subsec:Zig_2Gav}

By Itô calculus, 
	for any deterministic observables $R_1,R_2\in\C^{N\times N}$, recalling \eqref{Gjt}, \eqref{eq:def_flow}
	and \eqref{eq:OU},  we have   
	\begin{equation}
		\begin{split}
			\dif \langle G_{1,t}R_1G_{2,t}R_2\rangle &= \dif\mathcal{E}_t+\langle G_{1,t}R_1G_{2,t}R_2\rangle \dif t + \langle G_{1,t}R_1G_{2,t}\rangle \langle G_{2,t}R_2G_{1,t}\rangle\dif t\\
			&\quad +\langle G_{1,t}-M_{1,t}\rangle \langle G_{1,t}^2R_1G_{2,t}R_2\rangle \dif t + \langle G_{2,t}-M_{2,t}\rangle \langle G_{1,t}R_1G_{2,t}^2R_2\rangle\dif t,\\
			&\quad+\frac{\boldsymbol1(\beta=1)}{N}\bigg[ \langle G_{1,t}^\mathfrak{t}G_{1,t}R_1G_{2,t}R_2G_{1,t}\rangle\dif t+\langle G_{2,t}^\mathfrak{t}G_{2,t}R_2G_{1,t}R_1G_{2,t}\rangle\dif t \\
			&\qquad \qquad \qquad \quad +\langle (G_{1,t}R_1G_{2,t})^\mathfrak{t}G_{2,t}R_2G_{1,t}\rangle \dif t \bigg] \,,
			\end{split}
		\label{eq:GG_evol}
	\end{equation}
where the \emph{martingale term} in the first line of \eqref{eq:GG_evol} is given by
	\begin{equation}\label{eq:def_qv}
\dif\mathcal{E}_t=\frac{1}{\sqrt{N}}\sum\limits_{a,b=1}^N \partial_{ab}\langle G_{1,t}R_1G_{2,t}R_2\rangle \dif B_{ab}.
	\end{equation}
	Here $\partial_{ab}=\partial_{w_{ab}(t)}$ stands for the directional derivative in the direction of $w_{ab}(t)$ (here $w_{ab}(t)$ denote the entries of $W_t$), $\beta=1$, $\beta=2$ denote the real and complex case, respectively, and $\mathfrak{t}$ denotes the transposition. 
	 From now on to keep the presentation short and simple we only consider the complex case $\beta=2$, since the real case $\beta=1$ is very similar, it only requires to estimate a few more terms in \eqref{eq:GG_evol}, whose estimate does not require any new idea. We refer to \cite{edgeETH} for a similar case when the additional terms present in the real case were estimated carefully.
    
The differential in \eqref{eq:GG_evol} is complemented by the time derivative of the corresponding deterministic approximation (recall the shorthand notation $M_{12,t}^{R}$ from \eqref{eq:M2_t-dep}) given in the next lemma. Its proof is completely analogous to \cite[Lemma 5.5]{mesoCLT} and hence omitted.
	
    \begin{lemma}[Time derivative of $M_{12}$]\label{lem:M_evol} For any $t\in [0,T]$ it holds that
    \begin{equation}
    \partial_t \langle M_{12,t}^{R_1}R_2\rangle = \langle M_{12,t}^{R_1}R_2\rangle + \langle M_{12,t}^{R_1}\rangle\langle M_{21,t}^{R_2}\rangle.
    \label{eq:M_evol}
    \end{equation}
    \end{lemma}		
    	Then, using the shorthand notation 
    		\begin{equation}
    		g^{R_1,R_2}_t:=\left\langle \left(G_{1,t}R_1G_{2,t}-M_{12,t}^{R_1}\right)R_2\right\rangle,
        \label{eq:g_t}   	
    	\end{equation}
    	we find, subtracting \eqref{eq:M_evol} from \eqref{eq:GG_evol}, that
	\begin{equation}
			\dif g_t^{R_1,R_2} = \left(1+(2-k(R_1,R_2))\langle M_{12,t}^I\rangle \right) g_t^{R_1,R_2}\dif t +\dif\mathcal{E}_t + \mathcal{F}_t\dif t \,. 
		\label{eq:dif_g}
	\end{equation}
	Here, we introduced the notation $\mathcal{F}_t=\mathrm{Lin}_t + \mathrm{Err}_t$ for the {\it forcing term}, where
	the \emph{linear term} and \emph{error term} are given by 
	\begin{equation}\label{eq:Err}
		\begin{split}
						&\mathrm{Lin}_t = k(R_1)\langle M_{12,t}^{R_1}\rangle g_t^{I,R_2} + k(R_2)\langle M_{21,t}^{R_2}\rangle g_t^{R_1,I},\\
			&\mathrm{Err}_t = g_t^{I,R_2}g_t^{R_1,I} + \langle G_{1,t}-M_{1,t}\rangle \langle G_{1,t}^2R_1G_{2,t}R_2\rangle + \langle G_{2,t}-M_{2,t}\rangle \langle G_{1,t}R_1G_{2,t}^2R_2\rangle\,, 
		\end{split}
	\end{equation}
	respectively.
Moreover, we denoted
	\begin{equation}
		k(R_1,\ldots,R_m):=\#\lbrace j\in [1,m]: R_j\neq I\rbrace
		\label{eq:k}
	\end{equation}
	for deterministic $R_1,\ldots,R_m\in\mathbb{C}^{N\times N}$. 
	
Recall the exponent $\epsilon>0$ which is fixed in Theorem \ref{theo:multigllaw}. The current Setting \ref{setting} depends on $\epsilon$ through the definition of spectral domains \eqref{eq:specdom}. Take any $\xi_0,\xi_1,\xi_2\in (0,\epsilon/10)$
	such that $\xi_0<\xi_1/2<\xi_2/4$ and define the stopping time
	\begin{equation}
		\begin{split}
			\label{eq:stoptime}
			\tau^{R_1,R_2}&:=\inf\lbrace t\in [0,T]: \, \max_{s\in [0,t]}\max_{z_{j,0}\in\Omega^j_0}\alpha_s^{-1}\left\vert g_s^{R_1,R_2}\right\vert \ge N^{2\xi_{k(R_1,R_2)}}\rbrace, \quad\nc	\\
			\tau&:=\min\lbrace \tau^{R_1,R_2}:\, R_1,R_2\in\mathfrak{S}\rbrace\,, \quad \text{with} \quad \mathfrak{S}:=\lbrace I,B_1,B_1^*,B_2,B_2^*\rbrace \,, 
		\end{split}
	\end{equation}
	where we introduced the shorthand notation 
	\begin{equation*}
\alpha_t:=\frac{1}{N\eta_{1,t}\eta_{2,t}}\wedge \frac{1}{\sqrt{N\ell_t}\gamma_t} \,. 
	\end{equation*}
	We point out that both $g_s$ and $\alpha_s$ in \eqref{eq:stoptime} depend on the $z_{j,s}$'s and thus on the $z_{j,0}$'s via the flow as its initial condition. 
	
	In the analysis of \eqref{eq:dif_g} the following two quantities play significant role 
	
	\begin{minipage}{0.5\textwidth}
		\begin{align}
			f_r:= 2\Re\langle M_{12,r}^I\rangle\vee 0,
			\label{eq:def_f}
		\end{align}
	\end{minipage}
	\begin{minipage}{0.5\textwidth}
		\begin{align}
			\beta_r:=\vert 1-\langle M_{1,r}M_{2,r}\rangle\vert.
			\label{eq:def_beta}
		\end{align} 
	\end{minipage}
	
	These functions depend on time $r\in [0,T]$ and initial conditions $z_{j,0}\in\Omega^j_0$, $D_{j,0}\in\mathbb{C}^{N\times N}$, $j\in[1,2]$, but we will omit the dependence on initial conditions in notations when this does not cause an ambiguity. Also note that \eqref{eq:def_beta} is the time-dependent version of \eqref{eq:smallestev} where $\beta(z_1,z_2)$ is defined. Clearly $f_t$ is essentially the coefficient of $g_t^{R_1, R_2}$  in the linear ODE \eqref{eq:dif_g} 
	with forcing terms, so its exponential plays the role of the propagator. 
	 We stress that the notation $k(R_1,\ldots,R_m)$ introduced in \eqref{eq:k} serves only the purpose of covering all possible cases $R_1,R_2\in\mathfrak{S}$ in one formula \eqref{eq:dif_g}. We do not exploit the fact that for $k(R_1,R_2)>0$ the propagator with $1+(1-k(R_1,R_2)/2)f_t$ in the rhs.~of \eqref{eq:dif_g}
	 becomes smaller than $1+f_t$, but rather estimate the propagator from above by
	 the exponential of  $1+f_t$ in all cases.
	
	We now state two important technical lemmas whose proofs are postponed to Section~\ref{sec:techzig} of the Supplementary Material \cite{supplement} and after concluding the proof of \eqref{eq:Zig}, respectively. Lemma \ref{lem:prop} controls the propagator of \eqref{eq:dif_g}.

	\begin{lemma}[Bound on the propagator]
		\label{lem:prop}
		We have the following: 
		\begin{enumerate}
			\item For any spectral pairs $\nu_1, \nu_2$ it holds that
			\begin{equation}
				2\vert \langle M_{\nu_1,\nu_2}^I\rangle\vert\le\pi\rho_1/\eta_1+\pi\rho_2/\eta_2, \quad \text{with} \quad \rho_j(z)=\pi^{-1}\vert\langle \Im M_j(z)\rangle\vert,\, j\in[1,2].
				\label{eq:prop_bound_basic}
			\end{equation}
			\item For any $z_{j,0}\in\Omega_0^j$, $j\in[1,2]$, there exists $s_0=s_0(z_{1,0},z_{2,0})\in [0,T]$ such that $f_r>0$ for all $r<s_0$ and $f_r=0$ for all $r>s_0$. Note that $s_0$ may be an endpoint of $[0,T]$.
			\item For any $s,t\in[0,T]$, $s<t$, we have 
			\begin{subequations}
			    \begin{align}
					&\int_s^t f_r\dif r\le \log \frac{\eta_{1,s}\eta_{2,s}}{\eta_{1,t}\eta_{2,t}},
					\label{eq:prop_int_bound1}\\
					&\int_s^t f_r\dif r=2\log\frac{\beta_{s\wedge s_0}}{\beta_{t\wedge s_0}}.
					\label{eq:prop_int_bound2}
				\end{align}
			\end{subequations}
			\item For any $s,t\in [0,T]$, $s\le t$, it holds $\beta_s\sim \beta_t + (t-s)$.
		\end{enumerate}
	\end{lemma}
	
The following lemma controls the forcing terms of \eqref{eq:dif_g}, i.e.~the martingale term, the linear term and error term.  
	\begin{lemma}[Bound on the forcing terms] \label{lem:2G_av_error} Consider $R_1,R_2\in\mathfrak{S}$. Denote the quadratic variation of the martingale term $\dif \mathcal{E}_t$ \eqref{eq:def_qv} %in \eqref{eq:dif_g}
	 by 
		\begin{equation}
			\mathrm{QV}\big[g_t^{R_1,R_2} \big]:= \frac{1}{N}\sum\limits_{a,b=1}^N \left\vert\partial_{ab}\langle G_{1,t}R_1G_{2,t}R_2\rangle\right\vert^2.
			\label{eq:2G_QV}
		\end{equation} 
		Then for any $\zeta>0$ it holds, with very high probability, that
		\begin{equation}
			\begin{split}
		\left(\int_0^{t\wedge\tau}\mathrm{QV}\big[g_s^{R_1,R_2} \big] \dif s\right)^{1/2} &+ \int_0^{t\wedge\tau}|\mathcal{F}_s|\dif s \\
		&\lesssim \alpha_{t\wedge\tau} \left(k(R_1)N^{2\xi_{k(R_2)}}+k(R_2)N^{2\xi_{k(R_1)}} + N^\zeta\right)\log N
			\end{split}
			\label{eq:2G_av_error}
		\end{equation}
		uniformly in $t\in[0,T]$, $z_{j,0}\in\Omega^j_0$ and $\lVert B_j\rVert\le 1$, $j\in [1,2]$.
	\end{lemma}
	
	In the following, we will consider \eqref{eq:dif_g} as a system of equations for $g^{R_1,R_2}_t$, $R_1,R_2\in\mathfrak{S}$. For each choice of $R_1,R_2\in\mathfrak{S}$, we use the \emph{stochastic Gronwall argument} from \cite[Lemma 5.6]{gumbel} with \eqref{eq:2G_av_error} as an input to show that $\tau^{R_1,R_2}>\tau$ unless $\tau^{R_1,R_2}=T$. 
	This would readily imply that $\tau=T$ with very high probability, i.e.~\eqref{eq:Zig} holds.  

	Take any $R_1,R_2\in\mathfrak{S}$ and denote $g_s:=g_s^{R_1,R_2}$, $\xi:=\xi_{k(R_1,R_2)}$.  Consider \eqref{eq:2G_av_error} for some $\zeta<\xi$. 
	Due to the choice of $\xi_j$, $j\in[0,2]$ 
	  the rhs.~of \eqref{eq:2G_av_error} is upper bounded by 
	  $\alpha_{t\wedge\tau}N^{\xi}$, where we ignored the irrelevant $\log N$ factor. Then \cite[Lemma 5.6]{gumbel} with $d=1$ applied for the scalar equation \eqref{eq:dif_g} asserts that for any arbitrary small $\zeta>0$ and for any $t\ge 0$ we have
	\begin{equation}
	\begin{split}
		\sup_{0\le s\le t\wedge \tau}\left\vert g_s\right\vert^2 \lesssim &\left\vert g_0\right\vert^2 + N^{2\xi+3\zeta}\alpha_{t\wedge\tau}^2\\
		& +\int_0^{t\wedge\tau} \left(\vert g_0\vert^2 + N^{2\xi+3\zeta}\alpha_s^2\right) f_s \exp \left(2\left(1+N^{-\zeta}\right)\int_s^{t\wedge\tau}(1+f_r)\dif r\right)\dif s.
    \end{split}
		\label{eq:stoch_gronwall}
	\end{equation}
	It follows from \eqref{eq:Zig_init} that $\vert g_0\vert^2\lesssim N^{3\zeta}\alpha_0^2\le N^{3\zeta}\alpha_s^2$ with very high probability. Also \eqref{eq:prop_int_bound1} implies that
	\begin{equation*}
		\exp\left(2N^{-\zeta}\int_s^{t\wedge\tau}(1+f_r)\dif r\right)\le \exp\left(CN^{-\zeta}\log N\right)\lesssim 1.
	\end{equation*}
	Therefore, \eqref{eq:stoch_gronwall} simplifies to
	\begin{equation}
		\sup_{0\le s\le t\wedge \tau}\left\vert g_s\right\vert^2 \lesssim N^{2\xi+3\zeta}\alpha_{t\wedge\tau}^2 +N^{2\xi+3\zeta}\int_0^{t\wedge\tau}\alpha_s^2 f_s \exp \left(2\int_s^{t\wedge\tau}f_r\dif r\right)\dif s,
		\label{eq:stoch_gronwall_2}
	\end{equation}
	where we incorporated the contribution from the integral of $1$ in the exponent into the implicit multiplicative constant.
	
	Take $\zeta<\xi/3$. Then for the purpose of showing that $\tau=T$ with very high probability it suffices to verify the inequality
	\begin{equation}
		\int_0^{t\wedge\tau}\alpha_s^2 f_s \exp \left(2\int_s^{t\wedge\tau}f_r\dif r\right)\dif s\lesssim \alpha_{t\wedge\tau}^2\log N.
		\label{eq:2G_av_propagation}
	\end{equation}
	We first check that the lhs.~of \eqref{eq:2G_av_propagation} has an upper bound of order $\log N/(N\eta_{1,t\wedge\tau}\eta_{2,t\wedge\tau})^2$. In order to see this, we employ \eqref{eq:prop_bound_basic} and \eqref{eq:prop_int_bound1} along with $\alpha_s\le 1/(N\eta_{1,s}\eta_{2,s})$ and find that
	\begin{equation}
	\begin{split}
		&\int_0^{t\wedge\tau}\alpha_s^2 f_s \exp \left(2\int_s^{t\wedge\tau}f_r\dif r\right)\dif s \\
		&\quad\le\left(\frac{1}{N\eta_{1,t\wedge\tau}\eta_{2,t\wedge\tau}}\right)^2 \int_0^{t\wedge\tau}\left(\frac{\rho_{1,s}}{\eta_{1,s}}+\frac{\rho_{2,s}}{\eta_{2,s}}\right)\dif s\lesssim  \frac{\log N}{(N\eta_{1,t\wedge\tau}\eta_{2,t\wedge\tau})^2}.
		\label{eq:prop_1st_regime}
	\end{split}
	\end{equation}
	To establish the upper bound of order $\log N/(N\ell_{t\wedge\tau}\gamma^2_{t\wedge\tau})$
	 for the lhs.~of \eqref{eq:2G_av_propagation} we split the region of integration into two parts $[0,s_*]$ and $[s_*,t\wedge\tau]$, where $s_*$ is defined as
	\begin{equation}
		s_*:=\inf \left\lbrace s\in [0,t\wedge\tau]: \, \min\lbrace \eta_{1,s}/\rho_{1,s},\eta_{2,s}/\rho_{2,s}\rbrace\le \gamma_{t\wedge\tau}\right\rbrace.
		\label{eq:s*}
	\end{equation}
	Since $\eta_{j,s}/\rho_{j,s}$, $j\in [2]$, are monotonically decreasing functions in $s$, it holds that $\eta_{j,s}/\rho_{j,s}\le \gamma_{t\wedge\tau}$, $j\in [2]$, for $s\in [s_*,t\wedge\tau]$. Another property of $s_*$ which will be used is that 
	\begin{equation}
		t\wedge\tau-s_*\lesssim \gamma_{t\wedge\tau}.
		\label{eq:s*_ineq}
	\end{equation}
	We postpone the proof of \eqref{eq:s*_ineq} until the end of the proof of Part 1 of Proposition \ref{prop:Zig}. In combination with \eqref{eq:g_def} and the fourth statement of Lemma \ref{lem:prop}, \eqref{eq:s*_ineq} gives that
	\begin{equation}
		\beta_s\sim\beta_{t\wedge\tau},\quad \forall s\in [s_*,t\wedge\tau].
		\label{eq:b_sim}
	\end{equation}
	Armed with \eqref{eq:b_sim}, we are now ready to complete the proof of \eqref{eq:2G_av_propagation}. We may assume w.l.o.g.~that $t\le s_0$, since $f_s=0$ for $s>s_0$ 
	(recall Lemma \ref{lem:prop}~(2)). First, in the regime $s\in [0,s_*]$ we use that $\exp\big(\int_{s_*}^{t\wedge\tau}f_r\dif r\big)\sim 1$ by means of \eqref{eq:prop_int_bound2} and \eqref{eq:b_sim}, and thus an estimate similar to \eqref{eq:prop_1st_regime}  yields
	\begin{equation}
	\label{eq:propest1}
		\int_0^{s_*}\alpha_s^2 f_s \exp \left(2\int_s^{t\wedge\tau}f_r\dif r\right)\dif s\lesssim \frac{\log N}{(N\eta_{1,s_*}\eta_{2,s_*})^2}\lesssim \frac{\log N}{N\ell_{s_*}\gamma_{s_*}^2}\lesssim\frac{\log N}{N\ell_{t\wedge\tau}\gamma_{t\wedge\tau}^2}.
	\end{equation}
	Second, in the regime $s\in[s_*,t\wedge\tau]$ use \eqref{eq:prop_int_bound2}, $\alpha_s\le 1/(\sqrt{Nl_s}\gamma_s)$ and the bound $f_s\lesssim \beta_s^{-1}$ to get
	\begin{equation}
	\label{eq:propest2}
		\int_{s_*}^{t\wedge\tau} \alpha_s^2 f_s e^{2\int_s^{t\wedge\tau} f_r\dif r}\dif s \lesssim \frac{1}{N\ell_{t\wedge\tau}\gamma_{t\wedge\tau}^2} \cdot \frac{1}{\beta_{t\wedge \tau}^4}\int_{s_*}^{t\wedge\tau} \beta_s^3\dif s\sim \frac{\beta_{t\wedge\tau}^3(t\wedge\tau-s_*)}{N\ell_{t\wedge\tau}\gamma_{t\wedge\tau}^2\beta_{t\wedge\tau}^4}\lesssim \frac{1}{N\ell_{t\wedge\tau}\gamma_{t\wedge\tau}^2}.
	\end{equation}
	Here we used \eqref{eq:b_sim} in the last but one inequality and \eqref{eq:s*_ineq} in the last one. This finishes the proof of \eqref{eq:2G_av_propagation}. 
	
	Now we verify \eqref{eq:s*_ineq}. For any $r,s\in [0,T]$ from the definition of the characteristic flow we have that
	\begin{equation}
		e^r\eta_{j,r}/\rho_{j,r} - \mathrm{e}^s\eta_{j,s}/\rho_{j,s} = -(\mathrm{e}^r-\mathrm{e}^s)\pi/2.
		\label{eq:eta/rho}
	\end{equation}
	For $r=t\wedge\tau$, $s=s_*$ and $j$ such that $\gamma_{t\wedge\tau}\ge \eta_{j,s_*}/\rho_{j,s_*}$ we find that 
	\begin{equation*}
		\vert t\wedge\tau-s_*\wedge\tau\vert \sim \vert e^{t\wedge\tau}-e^{s_*}\vert  \lesssim \left\vert \eta_{j,t\wedge\tau}/\rho_{j,t\wedge\tau}\right\vert + \left\vert\eta_{j,s_*}/\rho_{j,s_*}\right\vert \lesssim \gamma_{t\wedge\tau}. 
	\end{equation*}
This concludes the proof of the average part \eqref{eq:Zig} of Part 1 of Proposition \ref{prop:Zig}. \qed

\vspace{2mm}

To prepare for the proof of Lemma \ref{lem:2G_av_error} in the next proposition 
 we show that  the spectral domains $\Omega_t^j$ and $\Omega_{\kappa, t}^j$ (see \eqref{eq:specdom} and \eqref{eq:specdom_bulk}) 
for $t\in [0,T]$ and $j \in [2]$ satisfy the \emph{ray property}. 
Informally this means that  for every $z$ in these domains with $\Im z > 0$ (resp.~$\Im z <0$) the vertical ray going off toward $\Re z+\ii \infty$ (resp.~$\Re z- \ii \infty$) is essentially 
contained in the domain. Since the result holds both for $j=1,2$, we will neglect $j$ in notations.  The proof of Lemma~\ref{lem:ray} is given in Section \ref{sec:techzig} of the Supplementary Material \cite{supplement}.

\begin{lemma}[Ray property for time dependent spectral domains] \label{lem:ray}  Fix a (large) $L>0$ and let $D\in \C^{N\times N}$ be a self-adjoint deformation with $\lVert D\rVert\le L$. Then we have the following. 
	\begin{itemize}
\item[(i)] \textnormal{[Unrestricted spectral domains]} For any $t\in [0,T]$, $z\in\Omega_t$ and $x\ge 0$ such that $\vert\Im z\vert+x\le N^{100}$ it holds that $z+{\mathrm sgn}(\Im z)\ii x\in\Omega_t$. That is, for $\Im z>0$ ($\Im z <0$) the vertical ray which starts at $z$,  goes up (down) and
leaves $\Omega_t$ only after reaching points with imaginary part larger than $N^{100}$ (smaller than $- N^{100}$). %{\cor [Oleksii: Add bound on $D$]}
\vspace{1mm}
\item[(ii)] \textnormal{[Bulk-restricted spectral domains]} Fix a (small) $\kappa>0$. Then there exists $t_*\in [0,T]$ such that the previous part of the statement holds for $\Omega_{\kappa,t}$ for any $t\in [t_*,T]$. Namely, For any $t\in [t_*,T]$, $z\in\Omega_{\kappa,t}$ and $x\ge 0$ such that $\vert\Im z\vert+x\le N^{100}$ it holds that $z+{\mathrm sgn}(\Im z)\ii x\in\Omega_{\kappa,t}$. Moreover, $T-t_*\sim 1$ with implicit constants which depend only on $\kappa$ and $L$.
	\end{itemize}
\end{lemma}

\begin{proof}[Proof of Lemma~\ref{lem:2G_av_error}] Recall that the target bound in \eqref{eq:2G_av_error} consists of three parts: The quadratic variation $\mathrm{QV}$ of the martingale term and the two contributions $\mathrm{Lin}$ and $\mathrm{Err}$ to $\mathcal{F}$. We will discuss each part separately. 
	
Before going into the proof, we point out that all bounds below hold with very high probability and for times $s\in [0,t\wedge\tau]$. We often omit in notations the dependence of resolvent chains and their deterministic approximations on time when this does not lead to an ambiguity.
	\\[2mm]
	\noindent\underline{Bound on $\mathrm{QV}$:} By computing the derivatives $\partial_{ab}$ in \eqref{eq:2G_QV}, using Schwarz inequality and Ward identity, we get 
	\begin{equation}
		\mathrm{QV}\big[g_s^{R_1,R_2}\big]\lesssim \frac{1}{N^2}\left( \eta_{1,s}^{-2}\langle \Im G_1R_1G_2R_2\Im G_1R_2^*G_2^*R_1^*\rangle + \eta_{2,s}^{-2}\langle \Im G_2R_2G_1R_1\Im G_2R_1^*G_1^*R_2^*\rangle\right).
		\label{eq:2G_QV_bound}
	\end{equation}
	In the following, we will focus on the first of the two terms in \eqref{eq:2G_QV_bound}, since the estimates for the second one are identical. 
	Firstly we give an upper bound which does not depend on $\gamma$:
	\begin{equation*}
		\langle \Im G_1R_1G_2R_2\Im G_1R_2^*G_2^*R_1^*\rangle \le \langle\Im G_1\rangle\lVert R_1G_2R_2\Im G_1R_2^*G_2^*R_1^*\rVert \lesssim \rho_{1,s}/(\eta_{1,s}\eta_{2,s}^2),
	\end{equation*} 
	where we used the averaged version of the single--resolvent local law from \eqref{eq:singllaw}. We thus conclude the bound
	\begin{equation*}
		\frac{1}{N^2}\int_0^{t\wedge\tau}\!\!\eta_{1,s}^{-2}\langle \Im G_1R_1G_2R_2\Im G_1R_2^*G_2^*R_1^*\rangle\dif s\lesssim \frac{1}{N^2}\int_0^{t\wedge\tau} \frac{\rho_{1,s}}{\eta_{1,s}^3\eta_{2,s}^2}\dif s\lesssim \left(\frac{1}{N\eta_{1,t\wedge\tau}\eta_{2,t\wedge\tau}}\right)^2.
	\end{equation*}
	
	Next, we aim to reduce the average four resolvent chain from \eqref{eq:2G_QV_bound} to a product of  average two resolvent chains. In order to do so, we introduce the shorthand notations $S:=R_2\Im G_1R_2^*$, $T:=R_1^*\Im G_1R_1$ and note that $S,T\ge 0$. Let $\lbrace \lambda_i^{(2)}\rbrace_{i\in[N]}$ be the eigenvalues of $W+D_2$ and ${\boldsymbol u}_i^2$ the corresponding normalized eigenvectors. By spectral decomposition of $G_2$ we can write
	\begin{equation}
		\begin{split}
			\vert \langle G_2SG_2^*T\rangle\vert &= \frac{1}{N}\left\vert \sum\limits_{i,j\in [N]} \frac{\langle {\boldsymbol u}_i^2,S{\boldsymbol u}_j^2\rangle\langle {\boldsymbol u}_j^2,T{\boldsymbol u}_i^2\rangle}{\big(\lambda_i^{(2)}-z_2\big)\big(\lambda_j^{(2)}-\bar{z}_2\big)}\right\vert\lesssim \frac{1}{N}\sum\limits_{i,j\in [N]} \frac{\langle {\boldsymbol u}_i^2, S{\boldsymbol u}_i^2\rangle\langle {\boldsymbol u}_j^2, T{\boldsymbol u}_j^2\rangle}{\big\vert\lambda_i^{(2)}-z_2\big\vert\cdot\big\vert\lambda_j^{(2)}-z_2\big\vert}\\[2mm]
			&= N\langle \vert G_2\vert S\rangle\langle \vert G_2\vert T\rangle = N\langle \Im G_1 R_2^* \vert G_2\vert R_2\rangle \langle \Im G_1 R_1 \vert G_2\vert R_1^*\rangle.
		\end{split}
		\label{eq:mart_reduct}
	\end{equation}
	In the end of the first line we used the positive definiteness of $S,T$ and the elementary estimate
	\begin{equation*}
		\begin{split}
			\langle {\boldsymbol u}_i^2,S{\boldsymbol u}_j^2\rangle\langle {\boldsymbol u}_j^2,T{\boldsymbol u}_i^2\rangle &\le \left(\langle {\boldsymbol u}_i^2,S{\boldsymbol u}_i^2\rangle\langle {\boldsymbol u}_j^2,S{\boldsymbol u}_j^2\rangle\langle {\boldsymbol u}_i^2,T{\boldsymbol u}_i^2\rangle\langle {\boldsymbol u}_j^2,T{\boldsymbol u}_j^2\rangle\right)^{1/2}\\
			&\lesssim \langle {\boldsymbol u}_i^2,S{\boldsymbol u}_i^2\rangle\langle {\boldsymbol u}_j^2,T{\boldsymbol u}_j^2\rangle + \langle {\boldsymbol u}_j^2,S{\boldsymbol u}_j^2\rangle\langle {\boldsymbol u}_i^2,T{\boldsymbol u}_i^2\rangle.
		\end{split}	
	\end{equation*}
	
	In order to deal with absolute values of resolvents we employ the integral representation  \cite[Eq.~(5.4)]{multiG}:	
	\begin{equation}
		\vert G(E+\ii \eta)\vert = \frac{2}{\pi}\int_0^\infty \frac{\Im G(E+\ii \sqrt{\eta^2+x^2})}{\sqrt{\eta^2+x^2}}\dif x.
		\label{eq:|G|_int}
	\end{equation}
	of $|G|$ in terms of $\Im G$ along the ray $z + \ii \, \mathrm{sgn}(z) x$ for $x \ge 0$ (cf.~Lemma \ref{lem:ray}). 
	Hence, using \eqref{eq:|G|_int} for the first factor on the rhs.~of \eqref{eq:mart_reduct} we get
	\begin{equation}
		\langle \Im G_1 R_2^* \vert G_2\vert R_2\rangle = \frac{2}{\pi}\int_0^\infty \langle \Im G_1 R_2^* \Im G_2(E_{2,s}+\ii\zeta_{2,s,x})R_2\rangle \zeta_{2,s,x}^{-1}\dif x,
		\label{eq:|G|_int_applied}	
	\end{equation}
	where we abbreviated 
	 $\zeta_{2,s,x}:=(\eta_{2,s}^2+x^2)^{1/2}$. We now split the region of integration $[0,\infty)$ into two parts: $S_1$ corresponds to the regime $\zeta_{2,s,x}\le N^{100}$ and $S_2$ is the complementary regime, i.e.~$S_2:=[0,\infty)\setminus S_1$. Now, for any $x\in S_1$, by Lemma \ref{lem:ray} it holds that $E_{2,s}+\ii\zeta_{2,s,x}\in \Omega_s^2$. Thus we conclude
	\begin{equation}\label{eq:gammab}
		\langle \Im G_1 R_2^* \Im G_2(E_{2,s}+\ii\zeta_{2,s,x})R_2\rangle \lesssim \frac{1}{\gamma(z_{1,s},E_{2,s}+\ii\zeta_{2,s,x})}\left(1+ \frac{N^{2\xi_2}}{\sqrt{N\ell(z_{1,s},E_{2,s}+\ii\zeta_{2,s,x})}} \right)
	\end{equation}	 
	where we abbreviated $\ell(z,z'):=\vert\Im z\vert\rho_{1,s}(z)\wedge \vert\Im z'\vert\rho_{2,s}(z')$ for $z,z'\in\C\setminus\R$. Along with the vague monotonicity of $\gamma$ in imaginary part \eqref{eq:g_monot} this inequality implies that 
\begin{equation}
\int_{S_1} \langle \Im G_1 R_2^* \Im G_2(E_{2,s}+\ii\zeta_{2,s,x})R_2\rangle \zeta_{2,s,x}^{-1}\dif x\lesssim \frac{\log N}{\gamma(z_{1,s},z_{2,s})}.
\label{eq:|G|_int_S1}	
\end{equation}
In the complementary regime we simply bound the integrand of \eqref{eq:|G|_int_applied} by the product of operator norms of resolvents. This gives an upper bound of order $\eta^{-1}_{1,s}N^{-100}$ for the integral over $S_2$. In particular, this is smaller than $\gamma^{-1}_s$ since $\gamma$ is a bounded function of $(z_1,z_2,D,D_2)\in \big(\C\setminus\R\big)^2\times \big(\C^{N \times N}\big)^2$ (see Setting~\ref{setting} and Definition \ref{def:gamma}).\nc

	Arguing similarly for the second factor in \eqref{eq:mart_reduct} we get
	\begin{equation*}
		\frac{1}{N^2}\int_0^{t\wedge\tau}\frac{1}{\eta_{1,s}^2}\langle \Im G_1R_1G_2R_2\Im G_1R_2^*G_2^*R_1^*\rangle\dif s\lesssim\frac{1}{N}\int_0^{t\wedge\tau}\frac{1}{\eta_{1,s}^2\gamma_s^2}\dif s\lesssim \frac{1}{N\ell_{t\wedge\tau}\gamma_{t\wedge\tau}}.
	\end{equation*}    	
	This concludes the desired bound on the quadratic variation.
	\\[2mm]
	\noindent\underline{Bound on $\mathrm{Lin}_t$.} Recalling the definition from \eqref{eq:Err}, in order to verify
	\begin{equation*}
		\int_0^{t\wedge\tau}\mathrm{Lin}_s\dif s \lesssim \alpha_{t\wedge\tau} \left(k(R_1)N^{2\xi_k(R_2)}+k(R_2)N^{2\xi_{k(R_1)}}\right)\log N
	\end{equation*}
	it is sufficient to notice that $\alpha_s$ increases along the flow and that by \eqref{eq:stab_bound} and \eqref{eq:int_bounds} we have 
	\begin{equation*}
		\int_0^{t\wedge\tau}\big| \langle M_{12,s}^{R_j}\rangle\big|\dif s \lesssim \int_0^{t\wedge\tau}\left(\frac{\rho_{1,s}}{\eta_{1,s}}\wedge 1\right)\dif s\lesssim \log N,\quad j\in[1,2].
	\end{equation*}
	\\[2mm]
	\noindent\underline{Bound on $\mathrm{Err}_t$.} For the first term in $\mathrm{Err}_t$ 
	in \eqref{eq:Err} by means of \eqref{eq:int_bounds} we easily find
	\begin{equation*}
		\int_0^{t\wedge\tau} \left\vert g_s^{I,R_2}g_s^{R_1,I}\right\vert \dif s\lesssim \frac{N^{2\xi_{k(R_1)}+2\xi_{k(R_2)}}}{N}\int_0^{t\wedge\tau} \frac{\alpha_s}{\eta_{1,s}\eta_{2,s}}\dif s\lesssim \frac{N^{2\xi_{k(R_1)}+2\xi_{k(R_2)}}}{N\ell_{t\wedge\tau}}\alpha_{t\wedge\tau}\lesssim \alpha_{t\wedge\tau}.
	\end{equation*}
	The remaining two terms in $\mathrm{Err}_t$ can be treated completely analogously, hence we focus on the first of the two for concreteness. 
	
	As the first step, we separate the first $G_1$ from the rest of the factors in $\langle G_1^2R_1G_2R_2\rangle$ via a Cauchy-Schwarz inequality followed by a Ward identity: 
	\begin{equation}
	\begin{split}
		\vert \langle G_1^2R_1G_2R_2\rangle\vert&\le \frac{\langle \Im G_1\rangle^{1/2}\langle \Im G_1R_1G_2R_2R_2^*G_2^*R_1^*\rangle^{1/2}}{\eta_{1,s}}\\
		&\le  \frac{\langle \Im G_1\rangle \lVert R_1G_2R_2R_2^*G_2^*R_1^*\rVert^{1/2}}{\eta_{1,s}}\le \frac{\rho_{1,s}}{\eta_{1,s}\eta_{2,s}}.
		\label{eq:3G_av_Err1}
	\end{split}
	\end{equation}
	The last estimate follows from the usual averaged single-resolvent local law for $\Im G_1$  \eqref{eq:singllaw}
	 and holds with very high probability. In order to get an upper bound for $\langle G_1^2R_1G_2R_2\rangle$ in terms of $\gamma$ we use the {\it reduction bound}
	\begin{equation}
		\vert \langle G_1^2R_1G_2R_2\rangle\vert 		\le \langle \vert G_1\vert R_1\vert G_2\vert R_1^*\rangle^{1/2} \langle \vert G_1\vert G_1^*R_2^*\vert G_2\vert R_2G_1\rangle^{1/2} \le \langle \vert G_1\vert R_2^*\vert G_2\vert R_2\rangle \eta_{1,s}^{-1},
		\label{eq:2G_av_err_red1}
	\end{equation}
	obtained analogously to \eqref{eq:mart_reduct}, where in the final estimate we additionally used  the
	commutativity of $\vert G_1\vert^{1/2}$, $G_1$ and $G_1^*$ together with $\lVert G_1G_1^*\rVert\le\eta_{1,s}^{-2}$. 
By means of \eqref{eq:|G|_int}, using similar arguments as around \eqref{eq:|G|_int_S1}
we hence find
	\begin{equation}
		\langle \vert G_1\vert R_1\vert G_2\vert R_1^*\rangle\lesssim \gamma_s^{-1}\log N, \quad \langle \vert G_1\vert R_2^*\vert G_2\vert R_2\rangle\lesssim \gamma_s^{-1}\log N\,, 
		\label{eq:int_with_2G}
	\end{equation}
and thus
	\begin{equation}
		\vert \langle G_1^2R_1G_2R_2\rangle\vert \lesssim (\eta_{1,s}\gamma_s)^{-1}\log N.
		\label{eq:2G_av_err_red}	
	\end{equation}
	Finally, combining \eqref{eq:3G_av_Err1} and \eqref{eq:2G_av_err_red} with the single-resolvent local law $\vert \langle G_{1,s}-M_{1,s}\rangle\vert\lesssim N^\zeta/(N\eta_{1,s})$ we find, with very high probability that
	\begin{equation}
	\label{eq:finbavnonreg}
		\int_0^{t\wedge\tau}\vert \langle G_{1,s}-M_{1,s}\rangle \langle G_{1,s}^2R_1G_{2,s}R_2\rangle\vert \dif s\lesssim \int_0^{t\wedge\tau} \frac{N^\zeta}{N\eta_{1,s}}\left(\frac{\rho_{1,s}}{\eta_{1,s}\eta_{2,s}}\wedge \frac{1}{\eta_{1,s}\gamma_s}\right)\dif s\lesssim N^\zeta \alpha_{t\wedge\tau}.
	\end{equation}
	This finishes the proof of Lemma \ref{lem:2G_av_error}. 
\end{proof}

\subsection{Isotropic two- and three-resolvent chains: Proof of \eqref{eq:2G_iso}--\eqref{eq:3G_iso} in Proposition~\ref{prop:Zig}}\label{subsec:Zig_iso}
Consider deterministic matrices $B_1,B_2\in\C^{N\times N}$ and unit vectors $\boldsymbol{x}, \boldsymbol{y}\in\C^N$. The argument below proves \eqref{eq:2G_iso}, \eqref{eq:3G_iso}  uniformly in $B_1,B_2$, $\boldsymbol{x}, \boldsymbol{y}$.  
For notational simplicity we omit the dependence of $z_j$ and $G_j$ on $t$. Furthermore, to keep the presentation simple, we give the proof only in the complex case ($\beta = 2$) just as in Section \ref{subsec:Zig_2Gav} and again refer to \cite{edgeETH}  for a detailed treatment of the case $\beta = 1$. 

To start with, the analog of \eqref{eq:dif_g}  for isotropic two resolvents is (recall \eqref{eq:k} for the definition of $k(R_1)$)
\begin{equation}
\begin{split}
&\dif \big(G_{1,t}R_1G_{2,t}-M_{12,t}^{R_1}\big)_{\boldsymbol{vw}}\\
&\quad = \left(1+(1-k(R_1))\langle M_{12,t}^I\rangle\right)\big(G_{1,t}R_1G_{2,t}-M_{12,t}^{R_1}\big)_{\boldsymbol{vw}}\dif t + \dif \mathcal{E}_t^{(2)} + \mathcal{F}_t^{(2)}\dif t,\\
\end{split}
\label{eq:2G_iso_dif}
\end{equation}
for any deterministic vectors $\boldsymbol{v}, \boldsymbol{w}$, 
where $\dif \mathcal{E}_t^{(2)}$ is the martingale term 
\begin{equation*}
\dif \mathcal{E}_t^{(2)}= \frac{1}{\sqrt{N}}\sum\limits_{a,b=1}^N \partial_{ab} \left(G_{1,t}R_1G_{2,t}\right)_{\boldsymbol{vw}}\dif B_{ab},
\end{equation*}
and the forcing term $\mathcal{F}_t^{(2)} = \mathrm{Lin}^{(2)}_t + \mathrm{Err}^{(2)}_t$ is the sum of the linear term $\mathrm{Lin}^{(2)}_t$ and the error term $\mathrm{Err}^{(2)}_t$,
\begin{equation}
\begin{split}
&\mathrm{Lin}_t^{(2)}:= k(R_1)\langle M_{12,t}^{R_1}\rangle \big(G_{1,t}G_{2,t}-M_{12,t}^I\big)_{\boldsymbol{vw}},\\
&\mathrm{Err}_t^{(2)}:=\big\langle G_{1,t}R_1G_{2,t} -  M_{12,t}^{R_1}\big\rangle \left(G_{1,t}G_{2,t}\right)_{\boldsymbol{vw}} + \langle G_{1,t}-M_{1,t}\rangle \left(G_{1,t}^2R_1G_{2,t}\right)_{\boldsymbol{vw}}\\
&\qquad \qquad+\langle G_{2,t}-M_{2,t}\rangle \left(G_{1,t}R_1G_{2,t}^2\right)_{\boldsymbol{vw}}\,,
\end{split}
\end{equation}
respectively. Recalling the short notation $M_{121,t}^{R_1,R_2}$ from \eqref{eq:M3t-dep} we similarly get that 
\begin{equation}
	\begin{split}
		&\dif \big(G_{1,t}R_1G_{2,t}R_2G_{1,t}-M_{121,t}^{R_1,R_2}\big)_{\boldsymbol{vw}}\\
		& = \big(\tfrac{3}{2}+(2-k(R_1,R_2))\langle M_{12,t}^I\rangle\big)\big(G_{1,t}R_1G_{2,t}R_2G_{1,t}-M_{121,t}^{R_1,R_2}\big)_{\boldsymbol{vw}}\dif t+ \dif\mathcal{E}_t^{(3)}+\mathcal{F}^{(3)}_t\dif t \,, 
	\end{split}
	\label{eq:3G_iso_dif}
\end{equation}
where now the martingale term is given by 
\begin{equation*}
\dif\mathcal{E}^{(3)}_t = \frac{1}{\sqrt{N}}\sum\limits_{a,b=1}^N \partial_{ab}\left(G_{1,t}R_1G_{2,t}R_2G_{1,t}\right)_{\boldsymbol{vw}}\dif B_{ab}
\end{equation*}
and the summands of $\mathcal{F}^{(3)}_t = \mathrm{Lin}^{(3)}_t+\sum\limits_{i=1}^3 \mathrm{Err}^{(3)}_{i,t}$ read
\begin{equation*}
	\begin{split}
		\mathrm{Lin}^{(3)}_t &= k(R_1)\langle M_{12,t}^{R_1}\rangle (G_{1,t}G_{2,t}R_2G_{1,t} - M_{121,t}^{I,R_2})_{\boldsymbol{vw}}\\
		& + k(R_2)\langle M_{21,t}^{R_2}\rangle \big(G_{1,t}R_1G_{2,t}G_{1,t}-M_{121,t}^{R_1,I}\big)_{\boldsymbol{vw}},\\[1mm]
		\mathrm{Err}^{(3)}_{1,t}&=\big\langle G_{1,t}R_1G_{2,t}R_2G_{1,t} -M_{121,t}^{R_1,R_2}\big\rangle \left(G_{1,t}^2\right)_{\boldsymbol{vw}} + \langle M_{121,t}^{R_1,R_2} \rangle \left(G_{1,t}^2 -M_{11,t}^I\right)_{\boldsymbol{vw}},\\[1mm]
		\mathrm{Err}^{(3)}_{2,t}&=\big\langle G_{1,t}R_1G_{2,t} - M_{12,t}^{R_1}\big\rangle(G_{1,t}G_{2,t}R_2G_{1,t})_{\boldsymbol{vw}} \\
		&+ \big\langle G_{2,t}R_2G_{1,t} - M_{21,t}^{R_2}\big\rangle (G_{1,t}R_1G_{2,t}G_{1,t})_{\boldsymbol{vw}},\\[1mm]
		\mathrm{Err}^{(3)}_{3,t} &= \langle G_{1,t}-M_{1,t}\rangle (G_{1,t}^2R_1G_{2,t}R_2G_{1,t})_{\boldsymbol{vw}} + \langle G_{2,t}-M_{2,t}\rangle (G_{1,t}R_1G_{2,t}^2R_2G_{1,t})_{\boldsymbol{vw}}\\[0.5mm]
		&+ \langle G_{1,t}-M_{1,t}\rangle (G_{1,t}R_1G_{2,t}R_2G_{1,t}^2)_{\boldsymbol{vw}} \,. 
	\end{split}
\end{equation*}

In the following analysis, we will need several tolerance exponents $\theta_0,\theta_1, \xi_0,\xi_1,\xi_2\in (0,\epsilon/10)$, which we required to satisfy the relations 
\begin{equation} \label{eq:exponentrelation}
\xi_0<\xi_1<\xi_2<2\xi_1<\theta_0<\theta_1\,.
\end{equation}  
 
We then define the stopping times
\begin{equation*}
	\begin{split}
		\tau^{R_1} := &\inf \Bigl\{ t\in[0,T]:\\
		& \max_{s\in[0,t]}\max_{\boldsymbol{v},\boldsymbol{w}\in\lbrace \boldsymbol{x},\boldsymbol{y}\rbrace }\max_{z_{j,0}\in\Omega_0^j} \sqrt{N\ell_s\eta_{*,s}\gamma_s}\left\vert \left(G_{1,s}R_1G_{2,s}-M_{12,s}^{R_1}\right)_{\boldsymbol{vw}}\right\vert \ge N^{2\theta_{k(R_1)}}\Bigr\},\nc\\
		\tau^{R_1,R_2} := &\inf \Bigl\{ t\in[0,T]:\\
		& \max_{s\in[0,t]}\max_{\boldsymbol{v},\boldsymbol{w}\in\lbrace \boldsymbol{x},\boldsymbol{y}\rbrace}\max_{z_{j,0}\in\Omega_0^j} \ell_s\gamma_s\left\vert \left(G_{1,s}R_1G_{2,s}R_2G_{1,s}^{(*)}\right)_{\boldsymbol{vw}}\right\vert \ge N^{2\xi_{k(R_1,R_2)}}\Bigr\},\nc\\
		\tau :=&\min\left\lbrace \tau^{R_1}, \tau^{R_1,R_2}:\, R_1,R_2\in \mathfrak{S}\right\rbrace \,, \quad \text{recalling}\quad \mathfrak{S}=\lbrace I,B_1,B_1^*,B_2,B_2^*\rbrace
	\end{split}
\end{equation*}
 from \eqref{eq:stoptime}. 
As in Section \ref{subsec:Zig_2Gav}, the goal is to show that $\tau = T$. First note that $\tau > 0$ by initial conditions \eqref{eq:init_iso2}, \eqref{eq:init_iso3}. 

To prove our goal, we control the terms on the rhs.~of \eqref{eq:2G_iso_dif} and \eqref{eq:3G_iso_dif}. In particular we claim that uniformly in $t\in[0,T]$ we have
\begin{subequations}
\begin{equation}
\left(\int_0^{t\wedge\tau} \hspace{-2mm}\mathrm{QV}_s^{(2)}\dif s\right)^{1/2} \hspace{-4mm}+\hspace{-1mm}\int_0^{t\wedge\tau} \hspace{-2mm}\vert\mathcal{F}_s^{(2)}\vert\dif s \lesssim \frac{N^{\xi_{2k(R_1)}}+k(R_1)N^{2\theta_0}}{\sqrt{N\ell_{t\wedge\tau}\eta_{*,t\wedge\tau}\gamma_{t\wedge\tau}}} \log N
,\label{eq:iso_err_bounds_a}
\end{equation}
\begin{equation}
\left(\int_0^{t\wedge\tau}\hspace{-2mm}\mathrm{QV}_s^{(3)}\dif s\right)^{1/2} \hspace{-4mm}+\hspace{-1mm}\int_0^{t\wedge\tau} \hspace{-2mm}\vert\mathcal{F}_s^{(3)}\vert\dif s \lesssim \frac{\sum_{j=1,2}k(R_j)N^{2\xi_{k(R_{3-j})}}\!+\!N^{\xi_0}\!+\!k(R_1,R_2)N^{\xi_2} }{\ell_{t\wedge\tau}\gamma_{t\wedge\tau}}\log N  
\label{eq:iso_err_bounds_b}
\end{equation}
\end{subequations}
%\begin{subequations}
%	\begin{align}
%		\left(\int_0^{t\wedge\tau} \hspace{-2mm}\mathrm{QV}_s^{(2)}\dif s\right)^{1/2} \hspace{-4mm}+\hspace{-1mm}\int_0^{t\wedge\tau} \hspace{-2mm}\vert\mathcal{F}_s^{(2)}\vert\dif s &\lesssim \frac{N^{\xi_{2k(R_1)}}+k(R_1)N^{2\theta_0}}{\sqrt{N\ell_{t\wedge\tau}\eta_{*,t\wedge\tau}\gamma_{t\wedge\tau}}} \log N,\label{eq:iso_err_bounds_a}\\
%		\left(\int_0^{t\wedge\tau}\hspace{-2mm}\mathrm{QV}_s^{(3)}\dif s\right)^{1/2} \hspace{-4mm}+\hspace{-1mm}\int_0^{t\wedge\tau} \hspace{-2mm}\vert\mathcal{F}_s^{(3)}\vert\dif s &\lesssim \frac{\sum_{j=1,2}k(R_j)N^{2\xi_{k(R_{3-j})}}\!+\!N^{\xi_0}\!+\!k(R_1,R_2)N^{\xi_2} }{\ell_{t\wedge\tau}\gamma_{t\wedge\tau}}\log N  \label{eq:iso_err_bounds_b}
%	\end{align}
%\end{subequations}
with very high probability,
where $\mathrm{QV}_s^{(2)}$ and $\mathrm{QV}_s^{(3)}$ are quadratic variations of $\dif\mathcal{E}^{(2)}_s$ and $\dif\mathcal{E}^{(3)}_s$ respectively.

For brevity, we omit the proof of \eqref{eq:iso_err_bounds_a}. From the proof of \eqref{eq:iso_err_bounds_b}, we discuss only the quadratic variation term (first term in the lhs.~of \eqref{eq:iso_err_bounds_b}) and $\mathrm{Err}_3^{(3)}$ (part of the second term in the lhs.~of \eqref{eq:iso_err_bounds_b}).
Along the way, the relations in \eqref{eq:exponentrelation} are used several times in order to accommodate error terms originating from the quadratic variation and the error terms $\mathrm{Lin}^{(2)}_t$ and $\mathrm{Lin}^{(3)}_t$.
We leave the rest of the technicalities to the reader and refer to \cite{nonHermdecay} where they are carefully carried out. 
However we point out that there are no new methods needed for analysis of the terms which we do not discuss here.

Firstly, for $\mathrm{QV}_s^{(3)}$ we have  
\begin{equation}
	\begin{split}
		N\cdot \mathrm{QV}_s^{(3)} &= \sum\limits_{a,b=1}^N \left\vert\partial_{ab}\left( G_1R_1G_2R_2G_1\right)_{\boldsymbol{vw}}\right\vert^2  \\
		&\lesssim {\eta_{1,s}^{-2}}\left(\Im G_1\right)_{\boldsymbol{vv}}\left(G_1^*R_2^*G_2^*R_1^*\Im G_1R_1G_2R_2G_1\right)_{\boldsymbol{ww}} \\
		& \quad +{\eta_{2,s}^{-2}}\left(G_1R_1\Im G_2R_1^*G_1^*\right)_{\boldsymbol{vv}}\left(G_1^*R_2^*\Im G_2R_2G_1\right)_{\boldsymbol{ww}} \\
		&\quad +{\eta_{1,s}^{-2}}\left(G_1R_1G_2R_2\Im G_1R_2^*G_2^*R_1^*G_1^*\right)_{\boldsymbol{vv}}\left(\Im G_1\right)_{\boldsymbol{ww}}.
	\end{split}
	\label{eq:3G_iso_qv}
\end{equation}

The long resolvent chains in the first and third term on the rhs.~of \eqref{eq:3G_iso_qv} have to be reduced to shorter ones via a suitable reduction inequality. 
In the first term on the rhs.~of \eqref{eq:3G_iso_qv}, we employ the polar decomposition of $G_2$, i.e. represent $G_2$ as $\vert G_2\vert U$, where $U$ is a unitary matrix. Denoting 
\begin{equation*}
{\boldsymbol x}:=U\vert G_2\vert^{1/2}R_2G_1{\boldsymbol w}\quad\text{and}\quad S:=\vert G_1\vert^{1/2}R_1^*\Im G_1R_1\vert G_1\vert^{1/2}
\end{equation*}
and using that $S\ge 0$ we get
\begin{equation}
	\begin{split}
&\left(G_1^*R_2^*G_2^* R_1^*\Im G_1R_1G_2R_2G_1\right)_{\boldsymbol{ww}} =  \left(S\right)_{\boldsymbol{xx}} \\
&\quad\le N\langle S\rangle \lVert x\rVert^2 = N \langle \Im G_1 R_1\vert G_2\vert R_1^*\rangle (G_1^*R_2^*\vert G_2\vert R_2G_1)_{\boldsymbol{ww}} \,. 
	\end{split}
\label{eq:3G_iso_qv_1st}
\end{equation}
By using the integral representation \eqref{eq:|G|_int} for $\vert G_2\vert$ in both factors in the rhs.~of \eqref{eq:3G_iso_qv_1st} we then find
\begin{equation}
	\langle \Im G_1 R_1\vert G_2\vert R_1^*\rangle \lesssim \frac{\log N}{\gamma_s} \quad \text{and} \quad (G_1^*R_2^*\vert G_2\vert R_2G_1)_{\boldsymbol{ww}} \lesssim \frac{N^{2\xi_2}\log N}{\ell_s\gamma_s}
\label{eq:2,3|G|}
\end{equation}
with very high probability. Here, to obtain the upper bound for $\langle \Im G_1 R_1\vert G_2\vert R_1^*\rangle$ we employed \eqref{eq:Zig}, which was already proven in Section \ref{subsec:Zig_2Gav}. Note that in the proof of the second part of \eqref{eq:2,3|G|} we encounter the resolvent chains of the form 
\begin{equation*}
(G_1^*R_2^*\vert \widetilde{G}_2\vert R_2G_1)_{\boldsymbol{ww}},\quad \text{where}\quad\widetilde{G}_2=\widetilde{G}_{2,s} := (W_s-D_{2,s}-z_{2,s}-ix)^{-1},\,x\cdot\mathrm{sgn}(\Im z)\ge 0.
\end{equation*}
These chains are bounded by $(\ell\gamma)^{-1}$ with very high probability, where $\ell$ and $\gamma$ are evaluated at $(z_{1,s},z_{2,s}+\ii x)$. So in order to argue similarly to \eqref{eq:|G|_int_applied}-\eqref{eq:gammab} we need to know that $\ell(z_{1,s},z_{2,s})\lesssim \ell(z_{1,s},z_{2,s}+ix)$. This indeed holds since $\eta \mapsto \eta\rho(E+\ii \eta)$ is an increasing function, which can be easily seen from the Stieltjes representation of $\rho(E+ \ii \eta)$.

For the third term in the rhs. of \eqref{eq:3G_iso_qv} the argument is the same, while for the second term we use Proposition \ref{prop:norm_M_bounds} in combination with the bound on the fluctuation of 3G isotropic chain which is available for $s\le\tau$.  Thus using \eqref{eq:int_bounds} we have
\begin{equation*} 
	\begin{split}
		\int_0^{t\wedge\tau} \mathrm{QV}^{(3)}_s\dif s&\lesssim \int_0^{t\wedge\tau}\left(\frac{\max\lbrace N^{2\xi_{2k(R_1)}},N^{2\xi_{2k(R_2)}}\rbrace}{\eta_{1,s}^2\ell_s\gamma_s^2}(\log N)^2+\frac{N^{-1+2\xi_{2k(R_1)}+2\xi_{2k(R_2)}}}{\eta_{2,s}^2\ell_s^2\gamma_s^2} \right)\dif s\\
		&\lesssim \frac{N^{2\xi_0}+k(R_1,R_2)N^{2\xi_2}}{\left(\ell_{t\wedge\tau}\gamma_{t\wedge\tau}\right)^2}\left((\log N)^2+\frac{N^{2\xi_2}}{N\eta_{2,t\wedge\tau}\rho_{2,t\wedge\tau}}\right)\\
		&\lesssim \left(\frac{\left(N^{\xi_0}+k(R_1,R_2)N^{\xi_2}\right)\log N}{\ell_{t\wedge\tau}\gamma_{t\wedge\tau}}\right)^2.
	\end{split}
\end{equation*}
Next we give the upper bound for $\mathrm{Err}_{3,t}^{(3)}$ providing the argument for the last term, for the other terms in $\mathrm{Err}^{(3)}_{3,t}$ the proof is similar. 
Use the usual averaged single resolvent local law from \eqref{eq:singllaw} for the first factor $\langle G_1-M_1\rangle$. In the second factor we employ the reduction estimate
\begin{equation*}
\left\vert \left(G_1R_1G_2R_2G_1^2\right)_{\boldsymbol{vw}}\right\vert \le \frac{N}{\eta_{1,s}}\left(\langle \vert G_1\vert R_1\vert G_2\vert R_1^*\rangle\langle \Im G_1 R_2^*\vert G_2\vert R_2\rangle \vert G_1\vert_{\boldsymbol{vv}}\left(\Im G_1\right)_{\boldsymbol{ww}}\right)^{1/2}.
\end{equation*}
Applying \eqref{eq:|G|_int} to the absolute values of resolvents and arguing similarly to \eqref{eq:|G|_int_applied} we get that
\begin{equation*}
	\int_0^{t\wedge\tau} \left\vert \langle G_{1,s}-M_{1,s}\rangle (G_{1,s}^2R_1G_{2,s}R_2G_{1,s})_{\boldsymbol{vw}}\right\vert\dif s \lesssim \int_0^{t\wedge\tau} \frac{N^{\zeta}}{N\eta_{1,s}}\cdot\frac{N(\log N)^2}{\eta_{1,s}\gamma_s}\dif s\lesssim \frac{N^{2\zeta}}{\ell_{t\wedge\tau}\gamma_{t\wedge\tau}}
\end{equation*}
for any $\zeta>0$.

Once we have \eqref{eq:iso_err_bounds_a},\eqref{eq:iso_err_bounds_b} in hand, we argue similarly to \eqref{eq:stoch_gronwall}\,--\,\eqref{eq:stoch_gronwall_2}.Thus in order to complete the proof of \eqref{eq:2G_iso}-\eqref{eq:3G_iso} it suffices to verify the inequalities (recall \eqref{eq:def_f} for the definition of $f_r$)
\begin{subequations}
	\begin{equation}
		\int_0^{t\wedge\tau}\frac{1}{N\ell_s\eta_{*,s}\gamma_s} f_s \exp{\left(\int_s^{t\wedge\tau}f_r\dif r\right)}\dif s\lesssim  \frac{\log N}{N\ell_{t\wedge\tau}\eta_{*,t\wedge\tau}\gamma_{t\wedge\tau}},
		\label{eq:prop_2G_iso}
	\end{equation}
	\begin{equation}
		\int_0^{t\wedge\tau}\frac{1}{\left(\ell_{s}\gamma_s\right)^2} f_s \exp{\left(2\int_s^{t\wedge\tau}f_r\dif r\right)}\dif s\lesssim  \frac{\log N}{\left(\ell_{t\wedge\tau}\gamma_{t\wedge\tau}\right)^2},
		\label{eq:prop_3G_iso}
	\end{equation}
\end{subequations}
where \eqref{eq:prop_2G_iso} corresponds to the propagation of the upper bound on the lhs. of \eqref{eq:2G_iso} and \eqref{eq:prop_3G_iso} is the analog of \eqref{eq:3G_iso}. The proof of \eqref{eq:prop_2G_iso}, \eqref{eq:prop_3G_iso} is analogous to the proof of \eqref{eq:2G_av_propagation} and is based on splitting of the interval of integration into $[0,s_*]$ and $[s_*,t\wedge\tau]$, where $s_*$ is defined in \eqref{eq:s*}. The only difference is that in the regime $s\in [0,s_*]$ one needs to use the bound $\gamma_s\ge \eta_{j,s}/\rho_{j,s}$, $j\in[2]$.

This concludes the proof of the isotropic parts \eqref{eq:2G_iso}--\eqref{eq:3G_iso} of Part 1 of Proposition \ref{prop:Zig}. \qed

\subsection{Modifications for the regular case: Proof of Part 2 of Proposition~\ref{prop:Zig}}\label{subsec:Zig_regular} 
Several steps in this proof are very similar to the ones presented in Sections~\ref{subsec:Zig_2Gav}--\ref{subsec:Zig_iso} we thus omit several details and present the proof only in the averaged case to illustrate the main differences in the simplest possible setting. Moreover, we work in the bulk-restricted spectral domains \eqref{eq:specdom_bulk} unlike in Sections~\ref{subsec:Zig_2Gav}--\ref{subsec:Zig_iso} where the proof is presented uniformly in the spectrum. In particular, it holds that $\ell_t\sim\eta_{1,t}\wedge\eta_{2,t}\wedge 1$.

Fix matrices $A_1, A_2\in\C^{N\times N}$ and take any $R_1,R_2\in \lbrace I,A_1,A_1^*,A_2,A_2^*\rbrace$. For initial conditions $z_{j,0}\in\Omega_{\kappa,0}^j$, $D_{j,0}\in\mathfrak{D}_j$, $j=1,2$, we will use the shorthand notation $\mathring{R}_j^{12}=\mathring{R}_j^{\nu_{1,t},\nu_{2,t}}$ and $\mathring{R}_j^{21}=\mathring{R}_j^{\nu_{2,t},\nu_{1,t}}$ whenever the time $t$ can be unambiguously determined from the context. Here we denoted $\nu_{j,t}:=(z_{j,t},D_{j,t})$ where $z_{j,t}, D_{j,t}$ is the solution to the characteristic flow equation \eqref{eq:def_flow} with initial conditions $z_{j,0}, D_{j,0}$.

We first consider the case when one observable is regularized and compute the differential $\dif g_t^{\mathring{R}_1^{12},R_2}$. Similarly to \eqref{eq:dif_g} we have
\begin{equation}
\begin{split}
\label{eq:flowreg}
\dif g_t^{\mathring{R}_1^{12},R_2} &=g_t^{\mathring{R}_1^{12},R_2}\dif t+\langle M_{12,t}^{\mathring{R}_1^{12}}\rangle g_t^{I,R_2} \dif t+ \langle M_{21,t}^{R_2}\rangle g_t^{\mathring{R}_1^{12},I}\dif t+\dif\mathcal{E}_t + \mathrm{Err}_t\dif t+\mathrm{Reg}_t^{(1)}\dif t, \\
\dif\mathcal{E}_t &=\frac{1}{\sqrt{N}}\sum\limits_{a,b=1}^N \partial_{ab}\langle G_{1,t}\mathring{R}_1^{12}G_{2,t}R_2\rangle \dif B_{ab},\\
\mathrm{Err}_t &= g_t^{I,R_2}g_t^{\mathring{R}_1^{12},I} + \langle G_{1,t}-M_{1,t}\rangle \langle G_{1,t}^2\mathring{R}_1^{12}G_{2,t}R_2\rangle + \langle G_{2,t}-M_{2,t}\rangle \langle G_{1,t}\mathring{R}_1^{12}G_{2,t}^2R_2\rangle,\\
\mathrm{Reg}_t^{(1)} &= -\partial_t \left[ \phi(\nu_{1,t},\nu_{2,t})\right] \frac{\langle M_{1,t}R_1M_{2,t}^{(*)}\rangle}{\langle M_{1,t}M_{2,t}^{(*)}\rangle}\langle G_{1,t}G_{2,t}R_2\rangle,
\end{split}
\end{equation}
for the definition of $\phi$ see \eqref{eq:def_phi}. In the third line in \eqref{eq:flowreg} the star above $M_{2,t}$ is present if and only if $\Im z_{1,t}\Im z_{2,t}>0$. The only difference of \eqref{eq:flowreg} from \eqref{eq:dif_g} is the additional error term $\mathrm{Reg}_t^{(1)}$ which comes from the differentiation of $\mathring{R}_1^{\nu_{1,t},\nu_{2,t}}$ in $t$. We point out that only the artificial cutoff gives a contribution to $\mathrm{Reg}_t^{(1)}$. If the regular component \eqref{eq:regulardef} was defined without $\phi$, then $\mathring{R}_1^{\nu_{1,t},\nu_{2,t}}$ would be independent of $t$ (see also Lemma \ref{lem:flow_properties}(ii)). Note that $\mathrm{Reg}_t^{(1)}=0$ when $\phi(\nu_{1,t},\nu_{2,t})\in \{0,1\}$ and in the complementary regime $\widehat{\gamma}(z_{1,t},z_{2,t})\sim 1$. Employing \eqref{eq:Zig_init} which is already proven in Sec.~\ref{subsec:Zig_2Gav} we get
\begin{equation}
\int_0^t \big| \mathrm{Reg}_s^{(1)}\big|\dif s\lesssim \frac{1}{\sqrt{N\ell_t}}, \quad \forall t\in [0,T].
\label{eq:reg_err}
\end{equation}

Beside \eqref{eq:flowreg} we also consider the case when both observables are regularized according to Definition~\ref{def:regulardef}. These two cases have to be considered together since their equations are coupled. The differential $\dif g_t^{\mathring{R}_1^{12},\mathring{R}_2^{21}}$ is completely analogous to \eqref{eq:flowreg} except for the term $\mathrm{Reg}_t^{(1)}$ which should be replaced by
\begin{equation*}
\mathrm{Reg}_t^{(2)}\!\!:=-\partial_t \!\!\left[ \phi(\nu_{1,t},\nu_{2,t})\right]\!\! \left(\frac{\langle M_{1,t}R_1M^{(*)}_{2,t}\rangle}{\langle M_{1,t}M_{2,t}^{(*)}\rangle}\langle G_{1,t}G_{2,t}\mathring{R}_2^{21}\rangle\! + \!\frac{\langle M_{2,t}R_2M^{(*)}_{1,t}\rangle}{\langle M_{2,t}M_{1,t}^{(*)}\rangle}\left\langle G_{1,t}\mathring{R}_1^{12}G_{2,t}\right\rangle\right)
\end{equation*}
with the same notational convention about $(*)$ as in \eqref{eq:flowreg}. It is easy to see that $\mathrm{Reg}^{(2)}$ satisfies the bound~\eqref{eq:reg_err}.

From now on, to further simplify the presentation, in the case when only one among $R_1,R_2$ is regularized we assume that the matrix which is not regularized equals to identity. If this is not the case, then one can proceed as in Section~\ref{subsec:Zig_2Gav} where $k(R_1,R_2)$ was introduced in \eqref{eq:k} in order to distinguish between identity and non-identity observables.

Introduce the stopping time
\begin{equation}
\begin{split}
\tau^{R_1}&:=\inf\left\{t\in [0,T]:   \max_{s\in[0,t]}\max_{z_{j,0}\in\Omega^j_{\kappa,0}} \alpha_{1,s}^{-1}\left(\big| g_s^{\mathring{R}_1^{12},I}\big|+\big| g_s^{I,\mathring{R}_1^{21}}\big|\right)\ge N^{2\xi_1}\right\rbrace,\\
\tau^{R_1,R_2}&:=\inf\left\{t\in [0,T]:   \max_{s\in[0,t]}\max_{z_{j,0}\in\Omega^j_{\kappa,0}} \alpha_{2,s}^{-1}\big| g_s^{\mathring{R}_1^{12},\mathring{R}_2^{21}}\big|\ge N^{2\xi_2}\right\rbrace,\\
\tau&:=\min\left\lbrace \tau^{R_1},\tau^{R_1,R_2}: R_1,R_2\in\lbrace A_1,A_1^*,A_2,A_2^*\rbrace\right\rbrace
\end{split}
\label{eq:defstoptime}
\end{equation}
for some small $0<\xi_1<\xi_2<\epsilon/10$, where 
\begin{equation}
\label{eq:esterrtermnew}
\alpha_{1,s}:=\frac{1}{N\eta_{1,s}\eta_{2,s}}\wedge\frac{1}{\sqrt{N\ell_s\gamma_s}}, \qquad\quad \alpha_{2,s}:=\frac{1}{N\eta_{1,s}\eta_{2,s}}\wedge\frac{1}{\sqrt{N\ell_s}}.
\end{equation}
The estimates for $g_t^{I,\mathring{R}_1^{21}}$ and $g_t^{\mathring{R}_1^{12},I}$ are completely analogous, in the following we thus consider only $g_t^{\mathring{R}_1^{12},I}$. Note that in the definition of the stopping time $\tau$ in \eqref{eq:defstoptime} we only consider quantities with at least one regular observable, since the case of no regular observables already follows by the results of Part 1 of this proof.

The argument around \eqref{eq:reg_err} shows that whenever $\phi \neq 1$, it contributes with controllable and irrelevant error terms ${\mathrm Reg}^{(1)}_t$ and ${\mathrm Reg}^{(2)}_t$ to \eqref{eq:flowreg} and to the analogue of \eqref{eq:flowreg} in the case of two regularized observables, respectively. Hence, for simplicity we may assume that for initial conditions $\nu_{j,0}=(z_{j,0},D_{j,0})$, $j=1,2$, it holds that 
\begin{equation}
\phi(\nu_{1,t},\nu_{2,t})=1,\quad \forall t\in [0,T].
\label{eq:reg_ass}
\end{equation}
The main advantage of this simplification is that in this way the concept of regularization becomes time independent. More precisely, recalling the definition of the regular component \eqref{eq:regulardef} and using Lemma \ref{lem:flow_properties}(i) along with \eqref{eq:reg_ass} we see that the regularization with respect to $(\nu_{1,t},\nu_{2,t})$ does not depend on $t$, i.e.
\begin{equation*}
\mathring{R}_1^{\nu_{1,t},\nu_{2,t}}=\mathring{R}_1^{\nu_{1,T},\nu_{2,T}} \quad\text{and}\quad \mathring{R}_2^{\nu_{2,t},\nu_{1,t}}=\mathring{R}_2^{\nu_{2,T},\nu_{1,T}}
\end{equation*}
for all $t\in [0,T]$. We will further assume that in \eqref{eq:defstoptime} $R_1$ is $(\nu_{1,T},\nu_{2,T})$-regular and $R_2$ is $(\nu_{2,T},\nu_{1,T})$-regular. 

The fact that we can achieve the bound $1/(N\eta_{1,t}\eta_{2,t})$ for $g_t^{A_1,I}$ follows directly by the arguments in Section~\ref{subsec:Zig_2Gav} for any general observable $A_1$. In the remainder of the proof we thus focus on proving the bounds $1/\sqrt{N\ell_t\gamma_t}$ and $1/\sqrt{N\ell_t}$ in the case of one or two regular observables, respectively. Throughout this section we use the properties of the characteristic flow from Lemma~\ref{lem:flow_properties} even if we do not mention it explicitly.

First, we notice that the first term in the rhs. of the differential equation in \eqref{eq:flowreg} can be neglected as it only amounts to a negligible rescaling $e^{-t}g_t^{\mathring{R}_1^{12},R_2}$. Then, we consider the stochastic term in \eqref{eq:flowreg}. To estimate this term we first bound its quadratic variation, denoted by $\mathrm{QV}[\cdot]$ as follows (we only write one representative term):
\begin{equation}
\begin{split}
\label{eq:estqvnew}
\mathrm{QV}\big[g_t^{A_1,I}\big]&\lesssim \frac{1}{N^2\eta_{1,t}^2}\langle \Im G_1 A_1G_2 \Im G_1 A_1 G_2^*\rangle \lesssim \frac{1}{N\eta_{1,t}^2}\langle \Im G_1 A_1 |G_2| A_1\rangle\langle \Im G_1 |G_2|\rangle, \\
\mathrm{QV}\big[g_t^{A_1,A_2}\big]&\lesssim  \frac{1}{N^2\eta_{1,t}^2}\langle \Im G_1 A_1 G_2 A_2\Im G_1 A_2 G_2^* A_1\rangle\\
&\lesssim \frac{1}{N\eta_{1,t}^2}\langle \Im G_1 A_1 |G_2| A_1\rangle\langle \Im G_1 A_2 |G_2| A_2\rangle,
\end{split}
\end{equation}
where we used the reduction inequalities from \eqref{eq:mart_reduct}. Here we restricted the argument to the case $R_1=A_1$ when one observable is regularized and $R_1=A_1$, $R_2=A_2$ when both are regularized. In general, one needs to consider $R_1,R_2\in\lbrace A_1,A_1^*,A_2,A_2^*\rbrace$, but all these cases are analogous to the one considered in \eqref{eq:estqvnew} and thus omitted. Notice that in the rhs. of \eqref{eq:estqvnew} also $|G|$ appeared. Products of traces with some $G$'s replaced by $|G|$ were already handled in \eqref{eq:|G|_int_applied} using the integral representation \eqref{eq:|G|_int}, however, the situation here is more delicate as we need to ensure that it is still possible to gain the additional smallness coming from $A_1,A_2$ being regular along the whole vertical line \eqref{eq:|G|_int_applied}. This analysis was already performed in full detail in \cite[Eqs.~(6.3)--(6.10)]{OptLowerBound}, we thus not repeat it here. We point out that in \cite{OptLowerBound} this was done for fixed spectral parameters, however, given Lemma~\ref{lem:reg_reg}, the fact that $z_t$ now changes in time does not cause any complication as assuming \eqref{eq:reg_ass} the notion of regularity does not depend on time. Proceeding as in \cite[Eqs.~(6.3)--(6.10)]{OptLowerBound}, using \eqref{eq:boundM12}, \eqref{eq:M2_reg}, and \eqref{eq:Zig}, we thus conclude
\begin{equation}
\mathrm{QV}\big[g_t^{A_1,I}\big]\lesssim \frac{1}{N\eta_{1,t}^2\gamma_t}, \qquad\qquad\quad \mathrm{QV}\big[g_t^{A_1,A_2}\big]\lesssim \frac{1}{N\eta_{1,t}^2}.
\end{equation}
By the path-wise Burkholder-Davis-Gundy inequality (see \cite[Appendix B.6, Eq. (18)]{BDG} with $c = 0$ for continuous martingale) we thus obtain
\begin{equation}
\label{eq:eststochnew}
\sup_{0\le t\le T}\left|\int_0^{t\wedge\tau} \dif \mathcal{E}_s\right|\lesssim N^{\xi_j}\alpha_{j,t\wedge\tau},
\end{equation}
with $j=1$ in the case of one regularized observable and $j=2$ when both $R_1,R_2$ are regularized. This convention will be used throughout this proof even if not mentioned explicitly.

Next, proceeding as in \eqref{eq:3G_av_Err1}--\eqref{eq:finbavnonreg}, using the bound \eqref{eq:M2_reg} for the deterministic terms, it is easy to see that
\begin{equation}
\label{eq:esterrterm}
\int_0^{t\wedge \tau} \mathrm{Err}_s\,\dif s \prec \frac{N^{\xi_j}}{N\ell_{t\wedge\tau}}\alpha_{j,t\wedge\tau}+\frac{N^{4{\xi_j}}}{N\gamma_{t\wedge\tau}}.
\end{equation}
We point out that also in the proof of \eqref{eq:esterrterm} we need to use the integral representation \eqref{eq:|G|_int_applied} as discussed above (see, e.g., \eqref{eq:|G|_int_S1}); we omit the details for brevity.

Combining \eqref{eq:esterrtermnew} and \eqref{eq:eststochnew}, for any $0\le s \le t$, by integrating 
the differential equation \eqref{eq:flowreg}  from $s$ to $t\wedge\tau$, we thus obtain
\begin{equation}
\begin{split}
\label{eq:onllintermnew}
 g_{t\wedge\tau}^{A_1,I}&=g_s^{A_1,I}+\int_s^{t\wedge\tau}\langle M_{12,r}^{A_1}\rangle g_r^{I,I}\dif r+\int_s^{t\wedge\tau}\langle M_{12,r}^I\rangle g_r^{A_1,I}\dif r+\mathcal{O}\left(N^{\xi_1} \alpha_{1,t\wedge\tau}\right), \\
 g_{t\wedge\tau}^{A_1,A_2}&=g_s^{A_1,A_2}+\int_s^{t\wedge\tau}\langle M_{12,r}^{A_1}\rangle g_r^{I,A_2}\dif r+\int_s^{t\wedge\tau}\langle M_{12,r}^{A_2}\rangle g_r^{A_1,I}\dif r+\mathcal{O}\left(N^{\xi_2} \alpha_{2,t\wedge\tau}\right).
 \end{split}
\end{equation}
The terms in \eqref{eq:onllintermnew} evaluated at time $s$ are estimated using \eqref{eq:Zig_inithr} and \eqref{eq:Zig_init_reg}, respectively. Using \eqref{eq:M2_reg} for the first integral in the first line of \eqref{eq:onllintermnew} and for both integrals in the second line of \eqref{eq:onllintermnew}, we thus obtain
\begin{equation}
 g_{t\wedge\tau}^{A_1,I}=\int_s^{t\wedge\tau}\langle M_{12,r}^I\rangle g_r^{A_1,I}\dif r+\mathcal{O}\left(N^{\xi_1} \alpha_{1,t\wedge\tau}\right), \qquad\qquad\quad  g_{t\wedge\tau}^{A_1,A_2}=\mathcal{O}\left(N^{\xi_2} \alpha_{2,t\wedge\tau}\right).
\end{equation}
To conclude the estimate of  $g_{t\wedge\tau}^{A_1,I}$ we apply Gronwall inequality and obtain
\begin{equation}
g_{t\wedge\tau}^{A_1,I}\lesssim N^{\xi_1} \alpha_{1,t\wedge\tau}+N^{\xi_1}\int_s^{t\wedge\tau} \alpha_{1,r} \frac{1}{\gamma_r}  \frac{\beta_r}{\beta_{t\wedge\tau}}\,\dif r\lesssim N^{\xi_1} \alpha_{1,t\wedge\tau},
\end{equation}
where in the first inequality we used \eqref{eq:def_f}--\eqref{eq:prop_int_bound2} and the second inequality follows by computations similar to \eqref{eq:propest1}--\eqref{eq:propest2}. This shows that $\tau=T$ and thus it concludes the proof.
\qed

\section{Zag step: Proof of (un)conditional Gronwall estimates in Lemmas \ref{lem:uncondGron2iso}--\ref{lem:condGron2av}} \label{sec:zag}
In this section, we prove the Gronwall estimates from Section \ref{subsec:GFTzag}. 
\subsection{Conditional Gronwall estimates: Proof of Lemmas \ref{lem:condGron2iso} and \ref{lem:condGron2av}} \label{subsec:condGron}
We begin by proving the conditional Gronwall estimates in Lemmas \ref{lem:condGron2iso} and \ref{lem:condGron2av}. 
\begin{proof}[Proof of Lemma \ref{lem:condGron2iso}]
	By Ito's formula and a cumulant expansion, we find that 
	\begin{equation} \label{eq:cumexIto}
		\frac{\dif}{\dif t} \E |R_t|^{2p} \lesssim \sum_{k=3}^{K} \frac{1}{N^{k/2}}  \sum_{a,b} \sum_{l = 0}^k \left| \E (\partial_{ab}^l \partial_{ba}^{k-l} |R_t|^{2p}) \right| + \mathcal{O}(N^{-100p})
	\end{equation}
	where we truncated the expansion at $K = \mathcal{O}(p)$ and used the trivial bound $\Vert G \Vert \le \eta^{-1}$ to estimate the error term in \eqref{eq:cumexIto}. 
	
	Throughtout the proof, we will frequently use that 
	\begin{equation*}
		\frac{1}{\sqrt{N \eta} \eta^{1/2} \gamma^{1/2}} \lesssim \mathcal{E}_1 \lesssim \mathcal{E}_0 \lesssim \frac{1}{\sqrt{N \eta} \eta} \,, \quad  \mathcal{E}_0/\mathcal{E}_1 \lesssim \eta^{-1/6}\,, \quad \text{and} \quad \eta \lesssim \gamma
	\end{equation*}
	as well as $\eta \lesssim 1$ and $N \eta \ge 1$, without further mentioning. 
	\\[2mm]
	\underline{{Third order terms:}} We begin by estimating the term of order $k=3$ in \eqref{eq:cumexIto}, as these are the most delicate ones. Distributing the derivatives according to the Leibniz rule, we see that there are three types of terms, namely (i) $(\partial^3R) |R|^{2p-1}$, (ii) $(\partial^2R) (\partial R) |R|^{2p-2}$, and (iii) $(\partial R)^3 |R|^{2p-3}$. For ease of notation, we shall henceforth drop the subscript $t$ of $R_t$ as well as the index of $G_i$, whenever it does not lead to confusion, or is irrelevant. Moreover, we will not distinguish between $R$ and $\overline{R}$ (and hence $G$ and $G^*$) as their treatment is exactly the same. 
	
	For terms of type (i), we focus on two exemplary constellations of indices; other terms are estimated analogously and are hence omitted. First, we consider 
	\begin{equation} \label{eq:GFT1isotype1.1}
		N^{-3/2} \Big|\sum_{a,b} G_{\boldsymbol x a} G_{bb} G_{aa} (G_1 B_1 G_2)_{b \boldsymbol y}\Big| |R|^{2p-1}\,. 
	\end{equation}
	For each of the four factors within the sum in \eqref{eq:GFT1isotype1.1}, we now employ either the the isotropic single resolvent law $G_{\boldsymbol u \boldsymbol v} = M_{\boldsymbol u \boldsymbol v} + \mathcal{O}_\prec \big((N \eta)^{-1/2}\big)$ or \eqref{eq:GFT1isoinput}. The resulting eight terms are then estimated by application of Schwarz inequalities (for the off-diagonal terms $M_{\boldsymbol x a}$ and $(M_{12}^{B_1})_{b \boldsymbol y}$) and \emph{isotropic resummation},
	e.g.~as
	\begin{equation} \label{eq:example1}
		N^{-3/2} \Big|\sum_{a,b} M_{\boldsymbol x a} M_{bb} M_{aa} (M_{12}^{B_1})_{b \boldsymbol y}\Big|  \lesssim N^{-1/2}\sqrt{\sum_a |M_{\boldsymbol x a}|^2 }\sqrt{\sum_b \big|(M_{12}^{B_1})_{b \boldsymbol y} \big|^2} \lesssim \frac{1}{\sqrt{N} \, \gamma}
	\end{equation}
	or, now using isotropic resummation for $(G-M)_{\boldsymbol x a}$, 
	\begin{equation*}
		N^{-3/2}\! \Big|\sum_{a,b} (G-M)_{\boldsymbol x a} M_{bb} M_{aa} (M_{12}^{B_1})_{b \boldsymbol y}\Big|\!\! \prec N^{-1} \big| (G-M)_{\boldsymbol x \boldsymbol m} \big|\sqrt{\sum_b \big|(M_{12}^{B_1})_{b \boldsymbol y} \big|^2}  \lesssim \frac{1}{N \sqrt{\eta} \, \gamma} \,, 
	\end{equation*}
	where we denoted $\boldsymbol m = (M_{aa})_{a\in [N]}$ and used that $\Vert \boldsymbol m \Vert \lesssim \sqrt{N}$,
	or
	\begin{equation*}
		N^{-3/2} \Big|\sum_{a,b} M_{\boldsymbol x a} (G-M)_{bb} M_{aa} (G_1 B_1 G_2 - M_{12}^{B_1})_{b \boldsymbol y}\Big| \prec \sqrt{\sum_a |M_{\boldsymbol x a}|^2}  \frac{\mathcal{E}_0}{\sqrt{N \eta}} \lesssim \frac{\mathcal{E}_0}{\sqrt{N \eta}}\,. 
	\end{equation*}
	In the above estimates we frequently used the bound 
	\begin{equation} \label{eq:Mboundcond}
		\big\Vert M_{12}^{B_1} \Vert \lesssim \Vert B_1 \Vert \, \gamma^{-1}
	\end{equation}
	from Proposition \ref{prop:norm_M_bounds}. 
	
	Collecting all the terms, we thus find by application of Young's inequality and using $\eta \lesssim 1$, that
	\begin{equation*}
\E 	\, [	\eqref{eq:GFT1isotype1.1}] \lesssim N^{\xi/2p} \left(\frac{1}{\sqrt{N \eta} \, \gamma} + \frac{\mathcal{E}_0}{\sqrt{N \eta}}\right) \E |R|^{2p-1} \lesssim \left(1 + \frac{1}{\sqrt{N}\eta^{3/2}}\right) \left( \E |R|^{2p} + N^\xi \mathcal{E}_1^{2p}\right) 
	\end{equation*}
for any $\xi > 0$. 
	Secondly, we consider 
	\begin{equation} \label{eq:GFT1isotype1.2}
		N^{-3/2} \Big|\sum_{a,b} G_{\boldsymbol x a} G_{ab} (G_1B_1G_2)_{bb} G_{a \boldsymbol y}\Big| |R|^{2p-1}\,. 
	\end{equation}
	Following the strategy explained below \eqref{eq:GFT1isotype1.1}, we estimate, e.g., 
	\begin{equation*} 
		N^{-3/2} \Big|\sum_{a,b} M_{\boldsymbol x a} (G-M)_{ab} (M_{12}^{B_1})_{bb} (G-M)_{a \boldsymbol y}\Big| \prec \frac{1}{N \eta \gamma} \sqrt{\sum_a |M_{\boldsymbol x a}|^2} \lesssim \mathcal{E}_1
	\end{equation*}
	or 
	\begin{equation*} 
		N^{-3/2} \Big|\sum_{a,b} (G-M)_{\boldsymbol x a} (G-M)_{ab} (G_1 B_1 G_2 - M_{12}^{B_1})_{bb} (G-M)_{a \boldsymbol y}\Big|\prec  \frac{N^{1/2}\mathcal{E}_0}{(N \eta)^{3/2}}   \lesssim \frac{\mathcal{E}_1}{\sqrt{N }\eta^{7/6}} \,, 
	\end{equation*}
	such that we conclude for any $\xi > 0$, just as above, 
	\begin{equation*}
	\E \, [	\eqref{eq:GFT1isotype1.2} ]\lesssim  \left(1 + \frac{1}{\sqrt{N}\eta^{3/2}}\right) \left(\E  |R|^{2p} + N^\xi \mathcal{E}_1^{2p}\right) \,.
	\end{equation*}
	
	For terms of type (ii), we again focus on two exemplary constellations of indices and omit the other ones, as they can be treated analogously. First, we consider
	\begin{equation} \label{eq:GFT1isotype2.1}
		N^{-3/2} \Big|\sum_{a,b} G_{\boldsymbol x a}  (G_1 B_1 G_2)_{b \boldsymbol y} G_{\boldsymbol x a} G_{bb} (G_1 B_1 G_2)_{a \boldsymbol y}\Big| |R|^{2p-2} \,,
	\end{equation}
	which we estimate as described below \eqref{eq:GFT1isotype1.1}. An exemplary term (ignoring $|R|^{2p-2}$) is bounded as
	\begin{equation} \label{eq:example2}
		\begin{split}
			N^{-3/2} &\Big|\sum_{a,b} M_{\boldsymbol x a}  (M_{12}^{B_1})_{b \boldsymbol y} (G-M)_{\boldsymbol x a} M_{bb} (G_1 B_1 G_2- M_{12}^{B_1})_{a \boldsymbol y}\Big| \\
			&\prec N^{-1} \sqrt{\sum_a |M_{\boldsymbol x a}|^2} \left|  (M_{12}^{B_1})_{\boldsymbol m \boldsymbol y} \right| \frac{\mathcal{E}_0}{\sqrt{N \eta}}
			\lesssim N^{-1/2}  \frac{\mathcal{E}_0}{\sqrt{N \eta} \gamma} \lesssim \mathcal{E}_1^2 \,,
		\end{split}
	\end{equation}
	where we used $\Vert \boldsymbol m \Vert \lesssim \sqrt{N}$ and \eqref{eq:Mboundcond}. 
	Secondly, we consider
	\begin{equation} \label{eq:GFT1isotype2.2}
		N^{-3/2} \Big|\sum_{a,b} G_{\boldsymbol x a}  (G_1 B_1 G_2)_{b \boldsymbol y} G_{\boldsymbol x b} (G_1 B_1 G_2)_{aa} G_{b \boldsymbol y}\Big| |R|^{2p-2} \,.
	\end{equation}
	Again, an exemplary term (following the strategy below \eqref{eq:GFT1isotype1.1}) can be estimated as
	\begin{equation*}
		\begin{split}
			N^{-3/2} &\Big|\sum_{a,b} (G-M)_{\boldsymbol x a}  (G_1 B_1 G_2 - M_{12}^{B_1})_{b \boldsymbol y} M_{\boldsymbol x b} (G_1 B_1 G_2- M_{12}^{B_1})_{aa} (G-M)_{a \boldsymbol y}\Big| \\
			&\prec  \sqrt{\sum_b |M_{\boldsymbol x b}|^2} \,  \frac{\mathcal{E}_0^2}{N \eta}
			\lesssim \frac{\mathcal{E}_1^2}{N \eta^{4/3}} \,. 
		\end{split}
	\end{equation*}
	In total, for terms of type (ii) we find, by means of Young's inequality, that, for any $\xi > 0$, 
	\begin{equation*}
	\E \, [	\eqref{eq:GFT1isotype2.1} + \eqref{eq:GFT1isotype2.2} ]\lesssim  \left(1 + \frac{1}{\sqrt{N}\eta^{3/2}}\right) \left( \E |R|^{2p} + N^\xi \mathcal{E}_1^{2p}\right) \,.
	\end{equation*}
		
	Lastly, for third order terms in \eqref{eq:cumexIto}, we turn to terms of type (iii). One exemplary and representative index constellation is given by 
	\begin{equation} \label{eq:GFT1isotype3.1}
		N^{-3/2} \Big|\sum_{a,b} G_{\boldsymbol x a}  (G_1 B_1 G_2)_{b \boldsymbol y} G_{\boldsymbol x a}  (G_1 B_1 G_2)_{b \boldsymbol y}  G_{\boldsymbol x b}  (G_1 B_1 G_2)_{a \boldsymbol y}\Big| |R|^{2p-3} \,,
	\end{equation}
	which we again estimate as described below \eqref{eq:GFT1isotype1.1}, e.g., as (neglecting $|R|^{2p-3}$)
	\begin{equation*} 
			N^{-3/2} \Big|\!\sum_{a,b} (G-M)_{\boldsymbol x a}  (G_1 B_1 G_2- M_{12}^{B_1})_{b \boldsymbol y} M_{\boldsymbol x a}  (M_{12}^{B_1})_{b \boldsymbol y}  (G-M)_{\boldsymbol x b}  (M_{12}^{B_1})_{a \boldsymbol y}\Big|\!\!\prec \frac{N^{-1}\mathcal{E}_0}{N \eta \gamma^2} \!\lesssim \mathcal{E}_1^3\,.
	\end{equation*}
	
	In total, for terms of type (iii) we find, by means of Young's inequality, that, for any $\xi > 0$,
	\begin{equation*}
	\E 	[\eqref{eq:GFT1isotype3.1}]  \lesssim \left(1 + \frac{1}{\sqrt{N}\eta^{3/2}}\right) \left( \E |R|^{2p} + N^\xi \mathcal{E}_1^{2p}\right) \,.
	\end{equation*}
	Therefore, collecting all the estimates for terms of type (i), (ii), and (iii), we have
	\begin{equation*}
		N^{-3/2}  \sum_{a,b} \sum_{l = 0}^3 \left| \E (\partial_{ab}^l \partial_{ba}^{k-l} |R|^{2p}) \right| \lesssim  \left(1 + \frac{1}{\sqrt{N}\eta^{3/2}}\right) \left( |R|^{2p} + N^\xi \mathcal{E}_1^{2p}\right) \,. 
	\end{equation*}
	\\[2mm]
	\underline{{Higher order terms:}} We now discuss the higher order terms in \eqref{eq:cumexIto} with $k \ge 4$ and distinguish two cases: First, we consider the case where the $k$ derivatives hit $m \le k-2$ different factors of $R$'s. Afterwards, we discuss the remaining case $m\in \{k-1, k\}$ (note that necessarily $m \le k$). 
	
	Indeed, for $m \le k-2$ different $R$ factors that are hit by a derivative, we employ the estimates ($\boldsymbol u, \boldsymbol v$ are arbitrary vectors of bounded norm)
	\begin{equation} \label{eq:input}
		|G_{\boldsymbol u \boldsymbol v}| \prec 1 \quad \text{and} \quad \left| (G_1 B_1 G_2)_{\boldsymbol u \boldsymbol v}\right| \prec \gamma^{-1} + \mathcal{E}_0
	\end{equation}
	for all but two off-diagonal terms. In this way, modulo changing one or more of the $a,b$ or $\boldsymbol x, \boldsymbol y$ indices to $b, a$ or $\boldsymbol y, \boldsymbol x$, respectively (which are all treated completely analogously), and ignoring the ``untouched" $|R|^{2p-m}$ factor, we arrive at 
	\begin{equation} \label{eq:higher1}
		N^{-k/2} \sum_{ab} |G_{\boldsymbol x a}| \big| G_{b\boldsymbol y}\big| \big(\gamma^{-1} + \mathcal{E}_0\big)^{m} 
	\end{equation}
	for $m \ge 2$, or, for $m=1$,
	\begin{equation} \label{eq:higher2}
		N^{-k/2} \sum_{ab} |G_{\boldsymbol x a}| \big| (G_1 B_1 G_2)_{b \boldsymbol y}\big| \,. 
	\end{equation}
	Following the strategy explained below \eqref{eq:GFT1isotype1.1}, we then find  
	\begin{equation*} 
		\begin{split}
			\eqref{eq:higher1} + \eqref{eq:higher2} &\prec \left(\frac{1}{N^{(k-2)/2} \gamma} +  \frac{\mathcal{E}_0}{N^{(k-3)/2} \eta^{1/2}} \right) \mathbf{1}(m=1)+ \frac{\big( \gamma^{-1} + \mathcal{E}_0\big)^{m}}{N^{(k-2)/2} \eta}  \mathbf{1}(m \ge 2)\\
			& \lesssim \frac{1}{N^{(k-2)/2} \eta \gamma^m} + \frac{\mathcal{E}_0 \, \mathbf{1}(m=1)}{N^{(k-3)/2} \eta^{1/2} } +  \frac{\mathcal{E}_0^m}{N^{(k-3)/2} \eta^{1/2}} \lesssim \left(1 + \frac{1}{\sqrt{N} \eta}\right) \mathcal{E}_1^m\,. 
		\end{split}
	\end{equation*}
	
	Next, for $m \in \{k-1, k\}$, we note that (by simple combinatorics) there are at least two $R$'s, which are hit by a derivative exactly once. Therefore, using \eqref{eq:input} for all the terms originating from the other $m-2$ differentiated $R$'s, and ignoring the ``untouched" $|R|^{2p-m}$ factor, we arrive at
	\begin{equation} \label{eq:higher3}
		N^{-k/2} \sum_{ab} |G_{\boldsymbol x a}| \big| (G_1 B_1 G_2)_{b \boldsymbol y}\big| |G_{\boldsymbol x a}| \big| (G_1 B_1 G_2)_{b \boldsymbol y}\big| \big(\gamma^{-1} + \mathcal{E}_0\big)^{m-2}
	\end{equation}
	or with $a,b$ in the last two terms interchanged. Similarly to above, we now estimate 
	\begin{equation*}
		\eqref{eq:higher3} \prec \left(\frac{1}{N^{k/2} \eta \gamma^2} + \frac{\mathcal{E}_0}{N^{(k-1)/2} \eta \gamma} +  \frac{\mathcal{E}_0^2}{N^{(k-2)/2} \eta } \right) \, \big(\gamma^{-1} + \mathcal{E}_0\big)^{m-2}\!\! \lesssim \left(1 + \frac{1}{\sqrt{N} \eta^{7/6}}\right) \mathcal{E}_1^m \,. 
	\end{equation*}
	Therefore, collecting all the terms of order $k \ge 4$, we have, by means of Young's inequality and using $\eta \lesssim 1$, 
	\begin{equation*}
		N^{-k/2}  \sum_{a,b} \sum_{l = 0}^k \left| \E (\partial_{ab}^l \partial_{ba}^{k-l} |R|^{2p}) \right| \lesssim \left(1 + \frac{1}{\sqrt{N}\eta^{3/2}}\right) \left(\E  |R|^{2p} + N^\xi\mathcal{E}_1^{2p}\right) \,,
	\end{equation*}
for any $\xi > 0$. 
	This concludes the proof of \eqref{eq:condGron2iso}. 
\end{proof}

\begin{proof}[Proof of Lemma \ref{lem:condGron2av}]
	Just as in \eqref{eq:cumexIto}, we compute by Ito's formula and a cumulant expansion (truncated at order $K = \mathcal{O}(p)$)
	\begin{equation} \label{eq:cumexIto2}
		\frac{\dif}{\dif t} \E |R_t|^{2p} \lesssim \sum_{k=3}^{K} \frac{1}{N^{k/2}}  \sum_{a,b} \sum_{l = 0}^k \left| \E (\partial_{ab}^l \partial_{ba}^{k-l} |R_t|^{2p}) \right| + \mathcal{O}(N^{-100p}) \,. 
	\end{equation}
	Just as in the proof of Lemma \ref{lem:condGron2iso}, for ease of notation, we shall henceforth drop the subscript $t$ of $R_t$ as well as the index of $G_i$, whenever it does not lead to confusion, or is irrelevant. Moreover, we will not distinguish between $R$ and $\overline{R}$ (and hence $G$ and $G^*$) as their treatment is exactly the same. We will first focus on the case where $\mathcal{E}_1 = 1/(\sqrt{N \eta} \gamma)$ (recall \eqref{eq:condGron2av}). 
	
	By direct computation, using \eqref{eq:GFT2avinput} and $\eta \lesssim \gamma$, we find that (the $N^{-1}$ comes from the normalized trace in the definition of $R$)
	\begin{equation} \label{eq:2avfund}
		\big| \partial_{ab}^l \partial_{ba}^{k-l} R \big| \prec \frac{1}{N \eta \gamma}
	\end{equation}
	for all $k \in \N$, $l \in [k] \cup \{0\}$. 
	
	Let $m \le k$ be the number of $R$-factors in \eqref{eq:cumexIto2}, that are hit by a derivative. For $k=3$ and $m \ge 2$, as well as $k \ge 4$ and $m \in [k]$ in \eqref{eq:cumexIto2}, the estimate \eqref{eq:2avfund} allows to bound these terms as (recall $\mathcal{E}_1$ from \eqref{eq:condGron2av})
	\begin{equation} \label{eq:avkge4}
		N^{-(k-4)/2} \left(\frac{1}{N \eta \gamma}\right)^m |R|^{2p-m} \lesssim \frac{1}{\sqrt{N} \eta} \left(|R|^{2p} + \mathcal{E}_1^{2p}\right)
	\end{equation}
	where we bounded the $a,b$ summations in \eqref{eq:cumexIto2} trivially, employed Young's inequality and used $N \eta \ge 1$. 
	
	The remaining case with $k=3$ and $m=1$ is now discussed separately. Note that, by explicit computation, in this case there is at least one off-diagonal term, i.e.~of the form $G_{ab}$, $(GBG)_{ab}$, or $(GBGBG)_{ab}$, resulting from three derivatives hitting a single $R$ factor. In the first case, using \eqref{eq:GFT2avinput} together with a Schwarz inequality, a Ward identity, and Young's inequality, we can bound these terms as
	\begin{equation*}
		N^{-5/2} \frac{1}{\eta \gamma}\sum_{ab} |G_{ab}| |R|^{2p-1}\prec \frac{1}{N \eta^{3/2} \gamma} |R|^{2p-1} \lesssim \frac{1}{\sqrt{N} \eta} \left( |R|^{2p} + \mathcal{E}_1^{2p}\right) \,. 
	\end{equation*}
	In the second case, the bound works completely analogously, using 
	\begin{equation*}
		N^{-5/2} \gamma^{-1}\sum_{ab} |(G_1 B G_2)_{ab}| \lesssim \frac{1}{N \eta^{1/2} \gamma} \sqrt{\big(G_1 B \Im G_2 B^* G_1^*\big)_{aa}} \prec \frac{1}{N \eta \gamma^{3/2}} \lesssim \frac{1}{\sqrt{N} \eta } \, \mathcal{E}_1
	\end{equation*}
	instead. In third case, however, we need to use \emph{isotropic resummation}: Since $(GBGBG)_{ab}$ is the only off-diagonal term (otherwise one could apply one of the first two cases), we necessarily deal with a term having the following index structure (ignoring the untouched $|R|^{2p-1}$)
	\begin{equation} \label{eq:avoddd}
		N^{-5/2} \sum_{ab} (G_1B_1G_2B_2G_1)_{ab} G_{aa} G_{bb} \,. 
	\end{equation} 
	We now write $G_{aa} = M_{aa} + \mathcal{O}_\prec((N \eta)^{-1/2})$, and similarly for $G_{bb}$, and estimate the resulting four terms separately. For the $M_{aa} M_{bb}$-term, we can isotropically sum up both indices $a,b$ as 
	\begin{equation*}
		N^{-5/2} \Big| \sum_{ab} (G_1B_1G_2B_2G_1)_{ab} M_{aa} M_{bb} \Big| \lesssim N^{-5/2} \left|(G_1B_1G_2B_2G_1)_{\boldsymbol m \boldsymbol m}\right|\prec \frac{1}{N^{3/2}\eta \gamma} \lesssim \mathcal{E}_1
	\end{equation*} 
	where we denoted $\boldsymbol m = (M_{aa})_{a\in [N]}$ and used that $\Vert \boldsymbol m \Vert \lesssim \sqrt{N}$. For the other three terms, we use \eqref{eq:GFT2avinput} and estimate the $a,b$ summations trivially such that we find them to be bounded by 
	$$(N \eta^{3/2} \gamma)^{-1} \lesssim \mathcal{E}_1/(\sqrt{N} \eta)\,.$$
	Thus, collecting all the terms and employing Young's inequality, we conclude \eqref{eq:condGron2av} for the case $\mathcal{E}_1 = 1/(\sqrt{N \eta} \gamma)$. 
	
	In the other case, when $\mathcal{E}_1 = 1/(N \eta_1 \eta_2)$, we only need to estimate the terms with $k=3$ slightly more carefully. In fact, for $k \ge 4$ the bound \eqref{eq:avkge4} is sufficient, since, by definition of $\gamma$ in Definition~\ref{def:gamma}, it holds that $\gamma \gtrsim \eta_1 \vee \eta_2$. Now, the main difference compared to the discussion above is that since $\mathcal{E}_1 = 1/(N \eta_1 \eta_2)$ has $N$ (instead of $\sqrt{N}$ as in the first case) in the denominator, the summations over $a$ and $b$ have to be carried out more effectively, i.e.~by exploiting as many off-diagonal terms as possible and by \emph{isotropic resummation}. In order to do this, we schematically decompose a diagonal resolvent chain as $G_{aa} = M_{aa} + \text{fluctuation}$, similarly to \eqref{eq:avoddd}. This is sufficient to treat all the terms arising for $k=3$ and $m=1$. 
	
	For $m=2,3$, however, there is an additional twist if the only off-diagonal terms are of the form $(G_1B_1G_2B_2G_1)_{ab} $, since we have no effective decomposition for longer isotropic chains. In this case, for $m=3$, we estimate 
	\begin{equation} \label{eq:2ndcasespecial}
		\begin{split}
&N^{-9/2} \Big| \sum_{a,b} \big((G_1B_1G_2B_2G_1)_{ab}\big)^3 \Big| \\
\prec \, &N^{-7/2} \frac{1}{\eta \gamma} \max_a \, (G_1B_1G_2B_2G_1G_1^*B_2^*G_2^*B_1^*G_1^*)_{aa} \\
\lesssim  \, &\frac{1}{N^{7/2}\eta_1^3 \eta_2^2 } \max_a \, (G_1B_1\Im G_2B_1^*G_1^*)_{aa} \lesssim \frac{1}{\sqrt{N} \eta} \, \frac{1}{(N \eta_1  \eta_2)^3} = \frac{\mathcal{E}_1^3}{\sqrt{N} \eta} \,. 
		\end{split}
	\end{equation}
	To go to the second line, we estimate one of the three factors $(G_1B_1G_2B_2G_1)_{ab}$ by \eqref{eq:GFT2avinput}. Next, we used the operator norm bound $\Vert B_2 G_1 G_1^* B_2^*\Vert \lesssim \eta_1^{-2}$ and a Ward identity. In the penultimate step, we used \eqref{eq:GFT2avinput} and the fact that $\gamma \gtrsim \eta_1 \vee \eta_2$. 
	Similar terms arising for $m=2$ are treated analogously to \eqref{eq:2ndcasespecial} and are hence left to the reader. 
	
	This concludes the proof of Lemma \ref{lem:condGron2av}. 
\end{proof}

\subsection{Unconditional Gronwall estimate: Proof of Lemma \ref{lem:uncondGron2iso}} \label{subec:uncondGron}

	The proof of Lemma \ref{lem:uncondGron2iso} is very similar to that of Lemma \ref{lem:condGron2iso} and we freely use the simplified notations introduced there. The only difference compared to Lemma \ref{lem:condGron2iso} is the following: In that proof we used the input estimate
	\begin{equation} \label{eq:splitting}
		\big(G_1 B_1 G_2\big)_{\boldsymbol u \boldsymbol v} = \big(M_{12}^{B_1}\big)_{\boldsymbol u \boldsymbol v} + \mathcal{O}_\prec (\mathcal{E}_0)
	\end{equation}
	from \eqref{eq:GFT1isoinput} and effectively summed up the $M$-term (see, e.g., \eqref{eq:example1}). In the current proof, we not use the splitting in \eqref{eq:splitting} but instead employ the trivial estimate $\big|(G_1 B_1 G_2)_{\boldsymbol u \boldsymbol v}\big| \prec \eta^{-1}$ (as follows by a Schwarz inequality together with a Ward identity and a single resolvent local law) or sum it up, e.g., as
	\begin{equation} \label{eq:longSchwarz}
		\sum_{a} |(G_1 B_1 G_2)_{\boldsymbol x a}| \le N^{1/2} \Big(\sum_a |(G_1 B_1 G_2)_{\boldsymbol x a}|^2 \Big)^{1/2} \prec N^{1/2}\eta^{-3/2}\,,
	\end{equation}
	where the final estimate follows from a Ward identity and \eqref{eq:3Gisotrivial}. 
	
	To illustrate the changes in a more concrete example, we consider \eqref{eq:GFT1isotype2.1}, and estimate it as
	\begin{equation*} 
		\begin{split}
			N^{-3/2} &\Big|\sum_{a,b} G_{\boldsymbol x a}  (G_1 B_1 G_2)_{b \boldsymbol y} M_{\boldsymbol x a} G_{bb} (G_1 B_1 G_2)_{a \boldsymbol y}\Big| \\
			&\prec N^{-3/2} \frac{1}{\eta}\sum_a |G_{\boldsymbol x a}|^2 \sum_b |(G_1 B_1 G_2)_{b \boldsymbol y}| \prec \frac{1}{N \eta^3 } = \mathcal{E}_0^2 
		\end{split}
	\end{equation*}
	by using a Schwarz inequality together with a Ward identity, the bound $|G_{\boldsymbol u\boldsymbol v}| \prec 1$, and \eqref{eq:longSchwarz}. 
	
	All the other terms can be treated with completely analogous simple modifications, hence we omit their detailed discussion. \qed 

\appendix

\section{Proofs of additional technical results}

\subsection{Proof of Proposition~\ref{prop:stab} and about its optimality}
\label{sec:addstabop}
In this section we first demonstrate the optimality of the lower  bound on $\beta_*$ from \eqref{eq:stab_bound} given in Proposition~\ref{prop:stab} and then present the proof of Proposition \ref{prop:stab} itself. Throughout this section, we will use the shorthand notation
$$
\Delta^2:= \langle (D_1-D_2)^2\rangle \,. 
$$

\begin{proposition}[Optimality of the stability bound in the bulk]\label{prop:stab_opt} Fix a (small) $\kappa>0$ and a (large) $L>0$. Let $D_1,D_2\in \C^{N\times N}$ be traceless Hermitian matrices with $\lVert D_l\rVert\le L$, $l=1,2$. Then uniformly in $E_1,E_2\in \R$ with $\max\lbrace\rho_1(E_1),\rho_2(E_2)\rbrace\ge \kappa$ it holds that
	\begin{equation}
		\beta_*(E_1+\ii 0,E_2+\ii 0)\sim \widehat{\gamma}(E_1+\ii 0,E_2+\ii 0).
		\label{eq:stab_opt}
	\end{equation}
	In \eqref{eq:stab_opt} implicit constants depend only on $\kappa$ and $L$.
\end{proposition}
\begin{proof}
	In the regime $\Delta^2+\vert E_1-E_2\vert\le c$, the estimate \eqref{eq:stab_opt} follows from a straightforward perturbative calculation for $\beta(E_1+\ii 0,E_2-\ii 0)$. Here, the implicit constant $c>0$ depends only on $\kappa$ and $L$. In the complementary regime, we have $\widehat{\gamma}\sim 1$ and also $\beta_*\sim 1$ by Proposition \ref{prop:stab}. Therefore, it holds that $\beta_*\sim\widehat{\gamma}$. The rest of the proof of Proposition \ref{prop:stab_opt} is elementary and thus omitted.
\end{proof}

\begin{proof}[Proof of Proposition~\ref{prop:stab}:] Assume for simplicity that $\mathbf{I}_1=\mathbf{I}_2=\R$. Since $\lVert D_j\rVert\le L$, we have that ${\mathrm supp}\rho_j\subset [-L-2,L+2]$ for $j=1,2$. 
	In the following, we will distinguish the two cases (i) $\max\lbrace \vert z_1\vert,\vert z_2\vert\rbrace \ge L+3$ and (ii) $\max\lbrace \vert z_1\vert,\vert z_2\vert\rbrace \le L+3$. 
	\\[2mm]
	\underline{Case (i):}	We will show that $\beta_*(z_1,z_2)\sim 1$ and $\widehat{\gamma}(z_1,z_2)\sim 1$, which imply, in particular \eqref{eq:stab_bound} and \eqref{eq:rho_stab}. Assume w.l.o.g.~that $\vert z_1\vert \ge L+3$. Denote 
	\begin{equation*}
		d_1:={\mathrm dist}(z_1,{\mathrm supp}\rho_1) = \min\lbrace \vert z_1-x\vert: x\in{\mathrm supp}\rho_1\rbrace.
	\end{equation*}
	Using the integral representation
	\begin{equation}
		\langle \Im M_1\rangle = \int_\R \frac{\eta_1}{\vert x-z_1\vert^2}\rho_1(x)\dif x,
		\label{eq:ImM_int_rep}
	\end{equation}
	we find that $\langle \Im M_1(z_1)\rangle\le \eta_1/d_1^2$. Therefore,
	\begin{equation*}
		\langle M_1M_1^*\rangle = \frac{\langle\Im M_1\rangle}{\eta_1+\langle \Im M_1\rangle}\le \frac{1}{1+d^2_1}.
	\end{equation*}    	
	This allows us to show that $\beta_*(z_1,z_2)\sim 1$, as follows from
	\begin{equation*}
		\begin{split}
			1&\gtrsim \beta_*(z_1,z_2)\ge 1-\max\lbrace \vert \langle M_1M_2\rangle\vert,\vert\langle M_1M_2^*\rangle\rbrace\\
			& \ge 1-\langle M_1M_1^*\rangle^{1/2}\langle M_2M_2^*\rangle^{1/2}\ge 1-\frac{1}{\sqrt{1+d^2_1}} \gtrsim 1.
		\end{split}
	\end{equation*}
	Here we used that $\langle M_2M_2^*\rangle\le 1$ and $d_1\ge 1$. Moreover, $\eta_1/\rho_1\gtrsim d^2_1$, which implies $\widehat{\gamma}(z_1,z_2)\sim 1$. Thus $\beta_*(z_1,z_2)\sim \widehat{\gamma}(z_1,z_2)$.
	\\[2mm]
	\underline{Case (ii):}
	For $\vert z_j\vert\le L+3$, $j=1,2$, we split the proof in two parts: the lower bound on $\beta_*$, and the upper bound on $\beta_*$. 
	\\[1mm]
	\underline{Lower bound on $\beta_*$.} Taking into account \cite[Proposition 4.2]{echo} it is sufficient to show that $\LT\lesssim \beta_*$. Subtracting \eqref{eq:MDE} for $M_1$ from \eqref{eq:MDE} for $M_2^*$ we get that
	\begin{equation*}
		z_1-\bar{z}_2-\frac{\langle M_1(D_1-D_2)M_2^*\rangle}{\langle M_1M_2^*\rangle}=\frac{(1-\langle M_1M_2^*\rangle)(\langle M_1\rangle -\langle M_2^*\rangle)}{\langle M_1M_2^*\rangle}. 
	\end{equation*}
	Therefore, we can rewrite $\LT$ as
	\begin{equation}
		\LT = \left\vert \frac{(1-\langle M_1M_2^*\rangle)(\langle M_1\rangle -\langle M_2^*\rangle)}{\langle M_1M_2^*\rangle}\right\vert \wedge 1.
		\label{eq:LT_2nd_form}
	\end{equation}
	If $\vert \langle M_1M_2^*\rangle\vert\ge 1/2$, \eqref{eq:LT_2nd_form} implies the bound $\LT\lesssim \vert 1-\langle M_1M_2^*\rangle\vert$, where we used that $\vert \langle M_1\rangle-\langle M_2^*\rangle\vert\lesssim 1$. In the complementary regime, i.e.~when $\vert \langle M_1M_2^*\rangle\vert <1/2$ we have  $\beta(z_1,\bar{z}_2)>1/2\gtrsim \LT$. 
	
	Now we prove that $\LT\lesssim \beta(z_1,z_2)$. First, consider the case $\vert\langle M_1M_2^*\rangle\vert\ge 1/2$. Again it is convenient to work with $\LT$ represented in the form  \eqref{eq:LT_2nd_form}. For the first factor in the numerator of \eqref{eq:LT_2nd_form} it holds that
	\begin{equation}
		\vert 1-\langle M_1M_2^*\rangle\vert \le \vert 1-\langle M_1M_2\rangle\vert +2\lVert M_1\rVert\cdot\vert \langle \Im M_2\rangle\vert \lesssim \vert 1-\langle M_1M_2\rangle\vert^{1/2}. 
		\label{eq:LT_stab_aux1}
	\end{equation}
	In the last step we used \eqref{eq:rho_stab} which is proven in \cite[Proposition 4.2]{echo}. For the second factor we use the bound
	\begin{equation}
		\begin{split}
			\vert \langle M_1\rangle -\langle M_2^*\rangle\vert^2 &\le \langle (M_1-M_2^*)(M_1^*-M_2)\rangle \\
			&=\langle M_1M_1^*\rangle +\langle M_2M_2^*\rangle -2\Re \langle M_1M_2\rangle \lesssim \vert 1-\langle M_1M_2\rangle\vert.
		\end{split}
		\label{eq:LT_stab_aux2}
	\end{equation}
	Therefore, \eqref{eq:LT_2nd_form} along with \eqref{eq:LT_stab_aux1} and \eqref{eq:LT_stab_aux2} implies $\LT\lesssim \beta(z_1,z_2)$. 
	
	Second, we consider the case $\vert \langle M_1M_2^*\rangle\vert < 1/2$. Then
	\begin{equation}
		\vert 1-\langle M_1M_2\rangle\vert \ge \vert 1-\langle M_1M_2^*\rangle\vert - 2\vert \langle M_1\Im M_2\rangle\vert\ge 1/2-2C_0\vert\langle\Im M_1\rangle\vert
		\label{eq:same_h_aux1}
	\end{equation}
	for some constant $C_0$. 
	In case that $\vert\langle \Im M_1\rangle\vert<1/(8C_0)$,  \eqref{eq:same_h_aux1} shows that $\beta(z_1,z_2)\ge 1/4\gtrsim \LT$. If $\vert\langle \Im M_1\rangle\vert\ge 1/(8C_0)$, we use \eqref{eq:rho_stab} to get $\beta(z_1,z_2) \ge \vert \langle\Im M_1\rangle\vert^2\gtrsim 1\gtrsim \LT$.
	\\[1mm]
	\underline{Upper bound on $\beta_*$.} Firstly we have
	\begin{equation}
		\beta_*\le \vert 1-\langle M_1M_2^*\rangle\vert\le \vert 1-\langle M_1M_1^*\rangle\vert + \vert \langle M_1^*(M_1-M_2)\rangle\vert.
		\label{eq:beta*upper}
	\end{equation}
	The first term on the rhs.~of \eqref{eq:beta*upper} has an upper bound of order $\widehat{\gamma}$, as follows from
	\begin{equation}
		\vert 1-\langle M_1M_1^*\rangle\vert = \frac{\eta_1}{\eta_1+\langle\Im M_1\rangle}\lesssim \frac{\eta_1}{\rho_1}\wedge 1\le \widehat{\gamma}.
		\label{eq:beta*upper1}
	\end{equation}
	The second term on the rhs.~of \eqref{eq:beta*upper} can be rewritten as
	\begin{equation}
		\left\vert \frac{(z_1-z_2-\langle M_1(D_1-D_2)M_2\rangle)\langle M_1^*M_1M_2\rangle}{1-\langle M_1M_2\rangle}\right\vert\lesssim \frac{\vert E_1-E_2\vert +\eta_1+\eta_2+\Delta}{\beta_*}\lesssim \frac{\widehat{\gamma}^{1/2}}{\beta_*}.
		\label{eq:beta*upper2}
	\end{equation}
	Now, combining \eqref{eq:beta*upper} with \eqref{eq:beta*upper1} and \eqref{eq:beta*upper2} we get that $\beta_*\lesssim \widehat{\gamma}^{1/4}$.
	
	This concludes the proof of Proposition \ref{prop:stab}. 
\end{proof}

\subsection{Proof of Proposition \ref{prop:gamma}}\label{sec:cone}

Before we turn to the proof of Proposition \ref{prop:gamma}, we explain some sufficient condition for $M$ being bounded on the whole complex plane. 

\begin{remark}[Sufficient condition for \eqref{eq:M_bound_I} with $\mathbf{I}=\R$]\label{rmk:Mbdd}
	As pointed out below Proposition \ref{prop:stab}, the bound \eqref{eq:M_bound_I} holds trivially in the bulk of the spectrum. We now give some sufficient conditions to ensure that \eqref{eq:M_bound_I} holds uniformly in the spectrum. Denote the eigenvalues of any self-adjoint deformation $D$ by $\lbrace d_j\rbrace_{j=1}^N$ labeled in increasing order, $d_j\le d_k$ for $j<k$. Fix a large positive constant $L>0$. The set $\mathcal{M}_L$ of admissible self-adjoint deformations $D$ is defined as follows: we say that $D\in \mathcal{M}_L$ if $\lVert D\rVert\le L$ and there exists an $N$-independent partition $\lbrace I_s\rbrace_{s=1}^m$ of $[0,1]$ in at most $L$ segments such that for any $s\in[1,m]$ and any $j,k\in[1,N]$ with $j/N,k/N\in I_s$ we have $\vert d_j-d_k\vert \le L\vert j/N- k/N\vert^{1/2}$. Since the operator $\mathcal{S}=\langle\cdot\rangle$ is flat, condition $D\in\mathcal{M}_L$ implies that $D$ satisfies \eqref{eq:M_bound_I} for $\mathbf{I}=\R$ with some $C_0<\infty$ by means of \cite[Lemma~9.3]{shape}.
\end{remark}

\begin{proof}[Proof of Proposition \ref{prop:gamma}]
	In order to prove Proposition \ref{prop:gamma}, we need to verify the properties of an \emph{admissible control parameter} from Definition \ref{def:gamma}. Note that in 
	Proposition \ref{prop:stab} we have already shown that $\widehat{\gamma}$ satisfies \eqref{eq:g_def}, i.e.~$\widehat{\gamma}$ is a lower bound on the stability operator. It thus remains to check items (2) and (3) of Definition \ref{def:gamma}, i.e.~monotonicity in time and vague monotonicity in imaginary part. In the rest of the proof, let $z_1,z_2\in\HH$ and $w_2:=z_2+\ii x$ with $x\ge 0$.\\[2 mm]
	\underline{Monotonicity in time:} In order to prove monotonicity in time, we claim that
	\begin{subequations}
		\begin{equation}
			\langle (D_{1,s}-D_{2,s})^2\rangle\sim \langle (D_{1,t}-D_{2,t})^2\rangle,
			\label{eq:d_in_time}
		\end{equation}
		\begin{equation}
			\LT_s\lesssim \LT_t + t-s,\quad \LT_t\lesssim \LT_s+t-s,
			\label{eq:LT_in_time}
		\end{equation}
		\begin{equation}
			\vert E_{1,s}-E_{2,s}\vert^2 \lesssim \vert E_{1,t}-E_{2,t}\vert^2 + (t-s)^2,\quad \vert E_{1,t}-E_{2,t}\vert^2 \lesssim \vert E_{1,s}-E_{2,s}\vert^2 + (t-s)^2,
			\label{eq:E_in_time}
		\end{equation}
		\begin{equation}
			\frac{\eta_{j,s}}{\rho_{j,s}}\wedge 1\sim \frac{\eta_{j,t}}{\rho_{j,t}}\wedge 1+t-s, \quad j\in [2],
			\label{eq:eta/rho_in_time}
		\end{equation}
	\end{subequations}
	uniformly in $s,t\in [0,T]$, $s\le t$. 
	
	The first assertion \eqref{eq:d_in_time} is a direct consequence of \eqref{eq:def_flow}, \eqref{eq:E_in_time} follows from \eqref{eq:z_t} and \eqref{eq:eta/rho_in_time} follows from \eqref{eq:eta/rho}.  To verify \eqref{eq:LT_in_time}, we again use \eqref{eq:z_t} for $z_{1,s},z_{2,s}\in\HH$ to get
	\begin{equation*}
		\begin{split}
			z_{1,t}-\bar{z}_{2,t}-\frac{\langle M_{1,t}(D_{1,t}-D_{2,t})M_{2,t}^*\rangle}{\langle M_{1,t}M_{2,t}^*\rangle} &= \ee^{-\frac{t-s}{2}}\left(z_{1,s}-\bar{z}_{2,s}-\frac{\langle M_{1,s}(D_{1,s}-D_{2,s})M_{2,s}^*\rangle}{\langle M_{1,s}M_{2,s}^*\rangle}\right)\\
			&\quad -2\left(\langle M_{1,s}\rangle - \langle M_{2,s}^*\rangle\right)\sinh\frac{t-s}{2}.
		\end{split}
	\end{equation*}
	Armed with \eqref{eq:d_in_time}-\eqref{eq:eta/rho_in_time} we obtain $\widehat{\gamma}_s+t-s\sim \widehat{\gamma}_t+t-s$. Moreover, by \eqref{eq:eta/rho_in_time} it holds that $\widehat{\gamma}_s\gtrsim t-s$, and thus $\widehat{\gamma}_s\sim \widehat{\gamma}_t+t-s$.
	\\[2mm]
	\underline{Vague monotonicity in space:} Note that $\widehat{\gamma}$ has the symmetry 
	\begin{equation*}
		\widehat{\gamma}(z_1,z_2,D_1,D_2)=\widehat{\gamma}(z_2,z_1,D_2,D_1).
	\end{equation*}
	Thus it is sufficient to prove the first part of \eqref{eq:g_monot}. In the following, we will distinguish between the two cases (i) $\vert \langle M_1(z_1)M_2^*(z_2)\rangle\vert \ge 1/2$ and (ii) $\vert \langle M_1(z_1)M_2^*(z_2)\rangle\vert < 1/2$. The exact choice of the threshold separating this two cases is not important, $1/2$ may be replaced by any  $c \in (0,1)$. The proof in case (ii) is much simpler, since it corresponds to the situation when $\beta_*(z_1,z_2)\gtrsim 1$ and one only needs to show that $\beta_*(z_1,w_2)\gtrsim 1$. The proof in case (i), however, is much more involved.\\[2 mm]
	\noindent\textit{Case (i):}  For $\vert \langle M_1(z_1)M_2^*(z_2)\rangle\vert \ge 1/2$, we first note that the integral representation \eqref{eq:ImM_int_rep} implies $\Im z_2/\rho_2(z_2)\le \Im w_2/\rho_2(w_2)$, i.e.~we have monotonicity of this summand in the definition of \eqref{eq:def_gamma_0}.

	It is thus left to show that
	\begin{equation}
		\LT(z_1,z_2)\lesssim \widehat{\gamma}(z_1,w_2).
		\label{eq:LT_ineq}
	\end{equation}
	
	First, suppose that $\vert \langle M_1(z_1)M_2^*(w_2)\rangle\vert\ge 1/2$.  If $\LT (z_1,z_2)\le \Delta^2$, then \eqref{eq:LT_ineq} obviously holds. Thus we may assume that $\LT(z_1,z_2)>\Delta^2$. Using the shorthand notations $M_j:=M_j(z_j)$, $j\in[2]$, $\widetilde{M}_2:=M_2(w_2)$ and $\Sigma:=D_1-D_2$, it is easy to see that
	\begin{equation}
		\begin{split}
			&\vert \LT(z_1,z_2)-\LT(z_1,w_2)\vert \\
			&\qquad\le \vert z_2-w_2\vert + \left\vert \frac{\langle M_1\Sigma(M_2^*-\widetilde{M}_2^*)\rangle}{\langle M_1M_2^*\rangle}\right\vert+ \left\vert \frac{\langle M_1\Sigma M_2^*\rangle\langle M_1(M_2^*-\widetilde{M}_2^*)\rangle}{\langle M_1M_2^*\rangle \langle M_1\widetilde{M}_2^*\rangle}\right\vert\\
			&\qquad = \vert z_2-w_2\vert + \left\vert \frac{\langle M_1\Sigma M_2^*\widetilde{M}_2^*\rangle(z_2-w_2)}{\langle M_1M_2^*\rangle(1-\langle M_2^*\widetilde{M}_2^*\rangle)}\right\vert + \left\vert \frac{\langle M_1\Sigma M_2^*\rangle\langle M_1M_2^*\widetilde{M}_2^*\rangle(z_2-w_2)}{\langle M_1M_2^*\rangle \langle M_1\widetilde{M}_2^*\rangle(1-\langle M_2^*\widetilde{M}_2^*\rangle)}\right\vert\\
			&\qquad \le \vert z_2-w_2\vert + (2L^3 +4L^5)\Delta \left\vert \frac{z_2-w_2}{1-\langle M_2\widetilde{M}_2\rangle}\right\vert.
		\end{split}
		\label{eq:LT_regularity_aux}
	\end{equation} 
	If $\vert\langle M_2\widetilde{M}_2\rangle\vert>1/2$, then $\left\vert (z_2-w_2)(1-\langle M_2\widetilde{M}_2\rangle)^{-1}\right\vert\sim \vert z_2-w_2\vert$. In the complementary case, $\vert\langle M_2\widetilde{M}_2\rangle\vert\le 1/2$, note that
	\begin{equation*}
		\left\vert \frac{z_2-w_2}{1-\langle M_2\widetilde{M}_2\rangle}\right\vert = \left\vert\frac{\langle M_2\rangle -\langle\widetilde{M}_2\rangle}{\langle M_2\widetilde{M}_2\rangle}\right\vert.
	\end{equation*}
	Since $D_2$ satisfies \eqref{eq:M_bound_I} with $\mathbf{I}=\R$, there exists $C_0'>0$ which depends only on $L$ such that 
	\begin{equation}
		\vert \langle M_2(\xi)\rangle-\langle M_2(\zeta)\rangle\vert\le C_0'\vert \xi-\zeta\vert^{1/3}
		\label{eq:M_holder}
	\end{equation}
	for any $\xi,\zeta\in\HH$ with $\vert \xi\vert,\vert\zeta\vert<L$. Therefore, in both cases, $\vert\langle M_2\widetilde{M}_2\rangle\vert > 1/2$ and $\vert\langle M_2\widetilde{M}_2\rangle\vert\le 1/2$, we have
	\begin{equation}
		\vert \LT(z_1,z_2)-\LT(z_1,w_2)\vert \le \vert z_2-w_2\vert + C_1\Delta\vert z_2-w_2\vert^{1/3}
		\label{eq:LT_regularity}
	\end{equation}
	for some constant $C_1$ which only depends on $L$. Next we distinguish between several regimes based on the relation of $\vert z_2-w_2\vert$, $\Delta$ and $\rho_2(w_2)$.
	
	\textbf{(1)} First, assume that $\vert z_2-w_2\vert \ge \Delta^{3/2}$. Then, as a consequence of \eqref{eq:LT_regularity}, we have
	\begin{equation*}
		\vert \LT(z_1,z_2)-\LT(z_1,w_2)\vert \le (C_1+1)\vert z_2-w_2\vert.
	\end{equation*}
	This immediately implies \eqref{eq:LT_ineq} in the case $\vert z_2-w_2\vert <\LT(z_1,z_2)/(2(C_1+1))$. In the complementary regime we have
	\begin{equation*}
		\Im w_2/\rho_2(w_2)\ge\Im w_2\ge \vert w_2-z_2\vert \ge (2(C_1+1))^{-1}\LT(z_1,z_2),
	\end{equation*}
	which allows to conclude \eqref{eq:LT_ineq} as well. 
	
	\textbf{(2)} Next, assume that $\vert z_2-w_2\vert < \Delta^{3/2}$ and $\rho_2(w_2)< C_2(\Im w)^{1/3}$, where $C_2>2C_0$ is a large positive constant depending only on $L$. From \eqref{eq:LT_regularity} we have
	\begin{equation*}
		\vert \LT(z_1,z_2)-\LT(z_1,w_2)\vert \le (C_1+1)\Delta\vert z_2-w_2\vert^{1/3}
	\end{equation*}
	which gives \eqref{eq:LT_ineq} for $\vert z_2-w_2\vert\le (\LT(z_1,z_2)/(2(C_1+1)\Delta))^3$. If $w_2$ does not satisfy this inequality, then it holds that
	\begin{equation}
		\LT(z_1,z_2)/2<(C_1+1)\Delta\vert z_2-w_2\vert^{1/3}\le (C_1+1)\LT^{1/2}(z_1,z_2)\vert z_2-w_2\vert^{1/3}.
		\label{eq:g_delta_bound}
	\end{equation}
	Therefore,
	\begin{equation*}
		(\Im w_2)^{2/3} \ge \vert w_2-z_2\vert^{2/3}\ge (2(C_1+1))^{-2}\LT(z_1,z_2).
	\end{equation*}
	In combination with the bound $\rho_2(w_2)<C_2(\Im w_2)^{1/3}$ this implies \eqref{eq:LT_ineq}.
	
	\textbf{(3)} Finally, assume that $\vert z_2-w_2\vert < \Delta^{3/2}$ and $\rho_2(w_2)\ge C_2(\Im w)^{1/3}$. It follows from \eqref{eq:M_holder} that for any $\zeta$ from the segment $\mathcal{I}$ connecting $z_2$ and $w_2$ we have $\rho_2(\zeta)\ge \rho_2(w_2)/2$. Hence
	\begin{equation*}
		\left\vert\langle M_2(z_2)\rangle -\langle M_2(w_2)\rangle\right\vert = \left\vert \int_{\mathcal{I}} \frac{\langle M^2_2(\zeta)\rangle}{1-\langle M^2_2(\zeta)\rangle}\dif\zeta\right\vert\le\frac{\vert z_2-w_2\vert}{\min_{\zeta\in\mathcal{I}}\vert 1-\langle M^2_2(\zeta)\rangle\vert}\le \frac{C_3\vert z_2-w_2\vert}{\rho_2(w)^2},
	\end{equation*}
	where $C_3$ depends only on $L$. Combine this bound with \eqref{eq:LT_regularity_aux}. 
	The case when the lhs.~of \eqref{eq:LT_regularity_aux} has an upper bound of order $\vert z_2-w_2\vert$ was already considered above. Thus we may assume that
	\begin{equation}
		\vert \LT(z_1,z_2)-\LT(z_1,w_2)\vert \le C_4 \Delta\frac{\vert z_2-w_2\vert}{\rho_2(w)^2},
		\label{eq:large_rho_regularity}
	\end{equation}
	where $C_4>0$ depends only on $L$. If the rhs.~of \eqref{eq:large_rho_regularity} is bounded from above by $\LT(z_1,z_2)/2$, we conclude the desired \eqref{eq:LT_ineq}. Otherwise, similarly to \eqref{eq:g_delta_bound} we get
	\begin{equation*}
		\LT(z_1,z_2)<\left(2C_4 \frac{\vert z_2-w_2\vert}{\rho^2_2(w_2)}\right)^2\le \left(2C_4\right)^2 \frac{\Im w_2}{\rho_2^2(w_2)}\cdot\frac{\Im w_2}{\rho(w_2)}\lesssim \frac{\Im w_2}{\rho(w_2)}
	\end{equation*}
	since $\rho_2(w_2)\ge C_2(\Im w)^{1/3}$.

	After having treated the case $\vert \langle M_1(z_1)M_2^*(w_2)\rangle\vert \ge  1/2$, we may assume that
	\begin{equation*}
		\vert \langle M_1(z_1)M_2^*(w_2)\rangle\vert< 1/2.
	\end{equation*}
	In this case, we have $\beta(z_1,\bar{w}_2)\ge 1/2$. Notice that
	\begin{equation}
		\vert\beta(z_1,\bar{w}_2)-\beta(z_1,w_2)\vert \le 2L\rho_2(w_2).
		\label{eq:beta_reg}
	\end{equation}
	If $\rho_2(w_2)<1/(8L)$, then \eqref{eq:beta_reg} gives $\beta(z_1,w_2)\ge 1/4$. Otherwise by \cite[Proposition 4.2]{echo} it holds that $\beta(z_1,w_2)\gtrsim \rho_2(w_2)^2 \gtrsim 1$. This means that $\beta_*(z_1,w_2)\sim 1$. Therefore, by Proposition \ref{prop:stab} $\gamma_0(z_1,w_2)\sim 1$, which immediately implies \eqref{eq:LT_ineq}.\\[2 mm]
	\noindent\textit{Case (ii):} In order to verify \eqref{eq:g_monot} in the case  $\vert\langle M_1(z_1)M_2^*(z_2)\rangle\vert< 1/2$, it is sufficient to show that $\beta_*(z_1,w_2)\sim 1$. Indeed, once we have this, the bound $\beta_*(z_1,w_2)\lesssim \widehat{\gamma}^{1/4}(z_1,w_2)$ from Proposition \ref{prop:stab} gives that $\widehat{\gamma}(z_1,w_2)\sim 1$, i.e. \eqref{eq:g_monot} holds. Using the H{\"o}lder 1/3-regularity \eqref{eq:M_holder} of $\langle M_2\rangle$ in a similar way as in the argument above \eqref{eq:M_holder} we get that $\beta(z_1,\bar{w}_2)\sim 1$ for $\vert z_2-w_2\vert\le c$ for some small positive constant $c\sim 1$ which depends only on $L$. For  $\vert z_2-w_2\vert> c$ by \eqref{eq:stab_bound} we have
	\begin{equation*}
		\beta(z_1,\bar{w}_2)\gtrsim \Im w_2/\rho_2(w_2)\ge \Im w_2\gtrsim 1.
	\end{equation*}
	Thus we have shown the existence of a (small) constant $c_0>0$ which depends only on $L$ such that $\beta(z_1,\bar{w}_2)\ge c_0$. Similarly to the proof around \eqref{eq:beta_reg} we argue that $\beta(z_1,w_2)\sim 1$. 
	
	This concludes the proof of Proposition \ref{prop:gamma}. 
\end{proof}

\subsection{Proof of Proposition \ref{prop:norm_M_bounds}:}\label{sec:norm_M_bounds} The proof is split in two parts. \\[2mm]
\textbf{Part 1:} The bound \eqref{eq:M2_gen} is the direct consequence of \eqref{eq:defM12}. In order to verify \eqref{eq:M3_gen} note that
\begin{equation*}
	\lVert M_{\nu_1,\nu_2,\nu_1}^{B_1,B_2}\rVert \lesssim \frac{\lVert B_1\rVert\cdot\lVert B_2\rVert}{\vert 1-\langle M_1M_2\rangle\vert^2\vert 1-\langle M_1^2\rangle\vert}.
\end{equation*}
Then use the lower bounds $\vert 1-\langle M_1M_2\rangle\vert\gtrsim \eta_1/\rho_1$ from Proposition \ref{prop:stab} and $\vert 1-\langle M_1^2\rangle\vert\gtrsim \rho_1^2$ from \eqref{eq:rho_stab} to get the desired result. For the upper bound on $\lVert M_{\nu_1,\nu_2,\bar{\nu}_1}^{B_1,B_2}\rVert$ the argument is similar, but one needs to use instead $\vert 1-\langle M_1M_2\rangle\vert\vee \vert 1-\langle M_1M_2^*\rangle\vert\gtrsim \rho_1^2$ and $\vert 1-\langle M_1M_1^*\rangle\vert\gtrsim \eta_1/\rho_1$ from Proposition \ref{prop:stab}.\\[2 mm]
\noindent\textbf{Part 2:} At first we prove \eqref{eq:M2_reg}. Inverting $\mathcal{B}_{12}$ defined in \eqref{eq:stabop} and using \eqref{eq:defM12} we get that
\begin{equation*}
	M_{\nu_1,\nu_2}^{A_1}=M_1A_1M_2+\frac{\langle M_1A_1M_2\rangle}{1-\langle M_1M_2\rangle}M_1M_2.
\end{equation*}
If $\Im z_1\Im z_2>0$, then by \eqref{eq:rho_stab} $\beta(z_1,z_2)\gtrsim \kappa^2$, so \eqref{eq:M2_reg} holds. Assume further that $\Im z_1\Im z_2<0$. Since $A_1$ is $(\nu_1,\nu_2)$-regular, either $\phi (\nu_1,\nu_2)$ defined in \eqref{eq:def_phi} vanishes or $\langle M_1A_1M_2\rangle=0$. In the first case $\widehat{\gamma}\sim 1$, so by Proposition \ref{prop:stab} $\beta(z_1,z_2)\sim 1$. In the second case $M_{\nu_1,\nu_2}^{A_1}=M_1A_1M_2$. In both cases $\lVert M_{\nu_1,\nu_2}^{A_1}\rVert\lesssim\lVert A_1\rVert$, i.e. \eqref{eq:M2_reg} holds.

The proofs of \eqref{eq:M3_reg_reg} and of the part of \eqref{eq:M3_reg_gen} which addresses $\lVert M_{\nu_1,\nu_2,\nu_1}^{A_1,B_2}\rVert$ go along the same lines. The only non-trivial bound is an upper bound \eqref{eq:M3_reg_gen} on $\lVert M_{\nu_1,\nu_2,\bar{\nu}_1}^{A_1,B_2}\rVert$ in the case when $\Im z_1\Im z_2>0$ and $\langle M_1A_1M_2^*\rangle=0$. Using explicit formulas for $\mathcal{B}_{13}^{-1}$ and for two-resolvent deterministic approximations we see that it is sufficient to verify the following cancellation between two terms:
\begin{equation}
	\left\vert\frac{\langle M_1M_1^*A_1M_2\rangle}{1-\langle M_1^*M_2\rangle}+\frac{\langle M_1A_1M_2\rangle\langle M_1M_1^*M_2\rangle}{(1-\langle M_1^*M_2\rangle)(1-\langle M_1M_2\rangle)}\right\vert\lesssim \frac{1}{\sqrt{\vert 1-\langle M_1^*M_2\rangle\vert}}.
	\label{eq:3G_cancel}
\end{equation}
By \eqref{eq:rho_stab} $\vert 1-\langle M_1M_2\rangle\vert \sim 1$. We further rewrite \eqref{eq:3G_cancel} as
\begin{equation*}
	\vert \langle M_1A_1M_2\rangle (1-\langle M_1^*M_2\rangle) - \langle M_1^*A_1M_2\rangle (1-\langle M_1M_2\rangle)\vert \lesssim \sqrt{\vert 1-\langle M_1^*M_2\rangle\vert},
\end{equation*}
which immediately follows from Lemma \ref{lem:reg_reg} applied to $y_1=y_2=0$. This finishes the verification of \eqref{eq:M3_reg_gen}. \qed

\subsection{Proof of Lemma \ref{lem:reg_reg}:}\label{sec:reg_reg} Let $w_1,w_2\in\C\setminus\R$ be any spectral parameters and denote $\nu_j^{\#}:=(w_j,D_j)$, $j=1,2$, and $\mathcal{A}:=\mathring{A}^{\nu_1,\nu_2}$. We have
\begin{equation*}
	\mathring{\mathcal{A}}^{\nu_1^{\#},\nu_2^{\#}} = \mathring{A}^{\nu_1^{\#},\nu_2^{\#}}+ (\mathcal{A}-A)(1-\phi(\nu_1^{\#},\nu_2^{\#})).
\end{equation*}
Using the fact that $\widehat{\gamma}(w_1,w_2)\sim 1$ when $\phi(\nu_1^{\#},\nu_2^{\#})\neq 1$ we get
\begin{equation*}
	\left\lVert \mathring{\mathcal{A}}^{\nu_1^{\#},\nu_2^{\#}} - \mathring{A}^{\nu_1^{\#},\nu_2^{\#}}\right\rVert \lesssim \lVert A\rVert \widehat{\gamma}(w_1,w_2).
\end{equation*}
Thus we may assume that $A=\mathcal{A}$, i.e. that $A$ is $(\nu_1,\nu_2)$-regular.

As usual we will denote $M_l(z):=M^{D_l}(z)$ for $l=1,2$. Since $A$ is $(\nu_1,\nu_2)$-regular, either (i) $\phi(\nu_1,\nu_2)=0$ or (ii) $\langle M_1(z_1)AM_2^*(z_2)\rangle=0$. In case (i), it is a direct consequence of the definition \eqref{eq:def_phi} of $\phi$ that $\phi(\nu_1',\nu_2')=0$. Therefore, since the lhs.~of \eqref{eq:reg_reg} vanishes, \eqref{eq:reg_reg} trivially holds. Thus, we will henceforth assume that $\langle M_1(z_1)AM_2^*(z_2)\rangle=0$. 

In the following, we will focus on showing that
\begin{equation}
	\lVert \mathring{A}^{\nu'_2,\nu'_1}-A\rVert\lesssim\lVert A\rVert\sqrt{\widehat{\gamma}(z_1',z_2')}
	\label{eq:reg_reg_spec}
\end{equation}
since the argument for the other bounds claimed in Lemma \ref{lem:reg_reg} are similar and thus are omitted. Firstly note that \eqref{eq:reg_reg_spec} is trivial in the case $\phi(\nu_1',\nu_2')=0$. In the complementary regime, where $\phi(\nu_1', \nu_2') \neq 0$, we have $\vert \langle M_2(z_2')M_1^*(z_1')\rangle\vert\sim 1$ and it is sufficient to prove that
\begin{equation}
	\vert \langle M_2(z_2')AM_1^*(z_1')\rangle\vert\lesssim \lVert A\rVert\sqrt{\widehat{\gamma}(z_1',z_2')}.
	\label{eq:reg_reg_aux1}
\end{equation}
Using the $(\nu_1,\nu_2)$-regularity of $A$ we rewrite the lhs.~of \eqref{eq:reg_reg_aux1} as
\begin{equation}
	\langle M_2(z_2')AM_1^*(z_1')\rangle =\langle M_2(z_2')A(M_1(z_1')-M_2(z_2))^*\rangle - \langle (M_1(z_1)-M_2(z_2'))AM_2^*(z_2)\rangle.
	\label{eq:reg_reg_aux2}
\end{equation}
Subtracting \eqref{eq:MDE} for $M_2(z_2)$ from \eqref{eq:MDE} for $M_1(z_1')$ we get
\begin{equation}
	\begin{split}
		M_1(z_1')-M_2(z_2)&= \frac{(z_1'-z_2)-\langle M_1(z_1')(D_1-D_2)M_2(z_2)\rangle}{1-\langle M_1(z_1')M_2(z_2)\rangle}M_1(z_1')M_2(z_2)\\
		& - M_1(z_1')(D_1-D_2)M_2(z_2).
		\label{eq:MDE_dif}
	\end{split}
\end{equation}
Since $\rho_2(z_2)\ge\kappa$, the denominator in \eqref{eq:MDE_dif} has a lower bound of order one by the means of \eqref{eq:rho_stab}. Plugging \eqref{eq:MDE_dif} into the first term on the rhs.~of \eqref{eq:reg_reg_aux2} we arrive at
\begin{equation*}
	\begin{split}
		\vert \langle M_2(z_2')A(M_1(z_1')-M_2(z_2))^*\rangle\vert &\lesssim \lVert A\rVert\left(\vert z_1'-z_2\vert +\langle (D_1-D_2)^2\rangle^{1/2}\right)\\
		&\lesssim \lVert A\rVert\widehat{\gamma}(z_1',z_2)\lesssim \lVert A\rVert\widehat{\gamma}(z_1',z_2').
	\end{split}
\end{equation*}
In the last step we used that $\widehat{\gamma}$ is an admissible control parameter (cf.~Proposition \ref{prop:gamma}) and hence satisfies the monotonicity property \eqref{eq:g_monot}. By a similar argument for the second term on the rhs.~of \eqref{eq:reg_reg_aux2} we conclude \eqref{eq:reg_reg_aux1} and thus the proof of Lemma \ref{lem:reg_reg}. \qed

\subsection{Proofs of technical results in the proof of Proposition~\ref{prop:Zig}}
\label{sec:techzig}

In this section we present the proofs of Lemma~\ref{lem:prop} and Lemma \ref{lem:ray}. 

\begin{proof}[Proof of Lemma~\ref{lem:prop}] We will verify each item in Lemma~\ref{lem:prop} separately. 
	\\[1mm]
	\underline{Item (1):} In order to prove \eqref{eq:prop_bound_basic} it is sufficient to show that
	\begin{equation}
		\frac{\langle M_1M_1^*\rangle^{1/2}\langle M_2M_2^*\rangle^{1/2}}{1-\langle M_1M_1^*\rangle^{1/2}\langle M_2M_2^*\rangle^{1/2}}\le \frac{\pi}{2}\left( \frac{\rho_1}{\eta_1} + \frac{\rho_2}{\eta_2} \right)
		\label{eq: f1}
	\end{equation}
	since $\vert\langle M_1M_2\rangle\vert\le \langle M_1M_1^*\rangle^{1/2}\langle M_2M_2^*\rangle^{1/2}$. Using the shorthand notations $x:=\pi\rho_1/\eta_1 > 0$ and $y:=\pi\rho_2/\eta_2 >0 $, we have
	\begin{equation*}
		\langle M_1M_1^*\rangle\langle M_2M_2^*\rangle = xy(x+1)^{-1}(y+1)^{-1}.
	\end{equation*}
	Then \eqref{eq: f1} is equivalent to
	\begin{equation*}
		\left(\left(1+1/x\right)^{1/2}\left(1+1/y\right)^{1/2}-1\right)^{-1}\le (x+y)/2\,, 
	\end{equation*}
	which can be rewritten as
	\begin{equation*}
		(x+y)^2 +(x+y)^2\left(1/x+1/y\right) + (x+y)^2/(xy)\ge (x+y)^2 + 4(x+y) + 4.
	\end{equation*}
	This inequality holds true since $(x+y)(1/x+1/y)\ge 4$ and $(x+y)^2\ge 4xy$. Thus,  \eqref{eq:prop_bound_basic} holds.
	\\[1mm]
	\underline{Item (2):} Under the characteristic flow, $M_{j,t}$ evolves as $M_{j,t}=\ee^{t/2}M_{j,0}$, cf.~Lemma \ref{lem:flow_properties}~(i). Thus
	\begin{equation*}
		\Re \langle M_{12,r}^I\rangle = \ee^r \frac{\Re\left[ \langle M_{1,0}M_{2,0}\rangle\right] - \ee^r \vert \langle M_{1,0}M_{2,0}\rangle\vert^2}{\vert 1-\langle M_{1,r}M_{2,r}\rangle\vert^2}.
	\end{equation*}
	Since $\Re\left[ \langle M_{1,0}M_{2,0}\rangle\right] - \ee^r \vert \langle M_{1,0}M_{2,0}\rangle\vert^2$ is monotonically decreasing in $r$ and the denominator is positive, the second statement of Lemma \ref{lem:prop} holds.
	\\[1mm]
	\underline{Item (3):}  In order to conclude \eqref{eq:prop_int_bound1}, we integrate \eqref{eq:prop_bound_basic} to get 
	\begin{equation*}
		\int_s^t f_r\dif r\le 2\int_s^t\vert \langle M_{12,r}^I\rangle\vert \dif r\le \int_s^t \left(\frac{\pi\rho_{1,r}}{\eta_{1,r}}+\frac{\pi\rho_{2,r}}{\eta_{2,r}}\right)\dif r\le \log \frac{\eta_{1,s}\eta_{2,s}}{\eta_{1,t}\eta_{2,t}}.
	\end{equation*}
	Here we used that $\pi\rho_{j,r}\le -\partial_r\eta_{j,r}$, $j=1,2$. To derive \eqref{eq:prop_int_bound2}, assume for notational simplicity that $s_0\ge t$.  Then
	\begin{equation*}
		\frac{1}{2}\int_s^t f_r\dif r=\Re\int_s^t \frac{\ee^r\langle M_{1,0}M_{2,0}\rangle}{1-\ee^r\langle M_{1,0}M_{2,0}\rangle}\dif r = \Re \log \frac{1-\langle M_{1,s}M_{2,s}\rangle}{1-\langle M_{1,t}M_{2,t}\rangle}=\log \frac{\beta_s}{\beta_t}.
	\end{equation*}
	\\[1mm]
	\underline{Item (4):}  For any $0\le s\le t\le T$ it holds that
	\begin{equation*}
		\begin{split}
			\beta_s &= \vert 1 -\langle M_{1,s}M_{2,s}\rangle\vert = \vert 1 - \ee^{s-t}\langle M_{1,t}M_{2,t}\rangle\vert\\
			& = \vert \ee^{s-t}(1-\langle M_{1,t}M_{2,t}\rangle) + (1-\ee^{s-t})\vert\sim \beta_t + t-s, 
		\end{split}	
	\end{equation*}
	where in the last implication we used that $1-\ee^{s-t}\ge 0$ and $\Re (1-\langle M_{1,t}M_{2,t}\rangle)\ge 0$.
\end{proof}

\begin{proof}[Proof of Lemma \ref{lem:ray}] We prove the two parts of Lemma \ref{lem:ray} separately. 
	\\[1mm]
	\underline{Part (i):} 
	At first we show that the constraint $\vert \Re z\vert\le N^{200}$ may be removed from the definition \eqref{eq:specdom} of $\Omega_T$. More precisely, we prove that if 
	\begin{equation}
		\vert\Im z\vert\rho_T(z)\ge N^{-1+\epsilon}\quad\text{and}\quad\vert\Im z\vert\le N^{100},
		\label{eq:specdom_cond}
	\end{equation}
	then $\vert\Re z\vert\le N^{200}$. Assume the opposite, i.e. that there exists $z=E+\ii \eta\in\C\setminus\R$ as in \eqref{eq:specdom_cond} such that $\vert\Re z\vert>N^{200}$. We have
	\begin{equation*}
		N^{-1+\epsilon}\lesssim\vert\Im z\vert\rho_T(z)=\frac{1}{\pi}\int_{\R}\frac{\eta^2}{(x-E)^2+\eta^2}\rho(x)\dif x\sim \frac{\eta^2}{\eta^2+E^2}.
	\end{equation*} 
	In the last step we used that $(x-E)^2+\eta^2\sim E^2+\eta^2$ for any $x\in{\mathrm supp}\,\rho_T$ once the distance from $E$ to the support of $\rho$ has a lower bound of order 1. Therefore it holds that
	\begin{equation*}
		\vert E\vert\lesssim\vert\eta\vert N^{(1-\epsilon)/2}\le N^{200},
	\end{equation*}
	which contradicts to the assumption $\vert E\vert >N^{200}$.
	
	Now we are ready to prove the first part of Lemma \ref{lem:ray}.  The ray property of $\Omega_T$ follows from the monotonicity of the function
	\begin{equation*}
		[0, \infty) \ni \eta \mapsto 	\eta\rho_T(E+\ii \eta)=\frac{1}{\pi}\int_\R\frac{\eta^2}{(x-E)^2+\eta^2}\rho_T(x)\dif x
	\end{equation*}
	for any fixed $E$. Moreover, since this function increases from $0$ at $\eta=0$ to $1$ at $\eta\to+\infty$, for any $E\in\R$ there exists a unique $\eta=\eta(E)>0$ such that 
	\begin{equation}
		\eta(E)\rho_T(E+\ii\eta(E))=N^{-1+\epsilon}.
		\label{eq:eta(E)}
	\end{equation}
	In particular, the part of the boundary of $\Omega_T\cap\HH$ which is not introduced by the constraint $\vert\Im z\vert\le N^{100}$ is a graph of a function $E \mapsto \eta(E)$. Differentiating the defining equation \eqref{eq:eta(E)} for $\eta(E)$ in $E$, we get that
	\begin{equation}
		\eta'(E) = \int_\R \frac{\eta (E-x)}{((x-E)^2+\eta^2)^2}\rho_T(x)\dif x \left(\int_\R \frac{(x-E)^2}{((x-E)^2+\eta^2)^2}\rho_T(x)\dif x   \right)^{-1}.
		\label{eq:eta'}
	\end{equation}
	
	Armed with these preliminaries, we will obtain Lemma \ref{lem:ray}~(i) by contradiction, so assume that for some $t\in [0,T)$ the ray property is violated. Then there exist two points $z_{1,t}, z_{2,t}$ with $\Im z_{j,t}< N^{100}$, $j=1,2$, on the boundary of $\Omega_t\cap\HH$ such that the vertical ray which enters $\Omega_t$ through one of this points leaves it through the other one. Denote $z_{j,T}:=\mathfrak{F}_{T,t}z_{j,t}$ and $E_j:=\Re z_{j,T}$, $j=1,2$. Without loss of generality assume that $E_1<E_2$. Then we have
	\begin{equation}
		\Re[\mathfrak{F}_{t,T}z_{1,T}]=\Re[\mathfrak{F}_{t,T}z_{2,T}].
		\label{eq:Re=Re}
	\end{equation}
	Since $z_{j,t}\in\partial\Omega_t$, $\Im z_{j,t}<N^{100}$ and $\Im z_{j,T}<\Im z_{j,t}$ by Lemma \ref{lem:flow_properties}, it holds that $\Im z_{j,T}=\eta(E_j)$, where $\eta(E)$ is defined in \eqref{eq:eta(E)}. Combining \eqref{eq:Re=Re} with \eqref{eq:z_t} we see that
	\begin{equation}
		\ee^{(T-t)/2}E_1+2\Re \langle M_T(z_{1,T})\rangle\sinh \frac{T-t}{2}=\ee^{(T-t)/2}E_2+2\Re \langle M_T(z_{2,T})\rangle\sinh \frac{T-t}{2}.
		\label{eq:deriv_aux}
	\end{equation}
	This is equivalent to
	\begin{equation}
		\frac{\Re \langle M_T(z_{1,T})\rangle - \Re \langle M_T(z_{2,T})\rangle}{E_1-E_2} = -\frac{1}{(1-\ee^{-(T-t)})}.
		\label{eq:deriv}
	\end{equation}
	While the rhs.~of \eqref{eq:deriv} is strictly smaller than $-1$, the lhs.~equals 
	\begin{equation}
		(E_2-E_1)^{-1}\int_{E_1}^{E_2}\partial_E \Re \langle M(E+\ii\eta(E))\rangle\dif E.
	\end{equation}
	Further, denoting $z:=E+\ii\eta(E)$ we have
	\begin{equation*}
		\begin{split}
			&\partial_E \Re \langle M(E+\ii\eta(E))\rangle = \partial_E \Re \langle M(z)\rangle + \partial_\eta \Re \langle M(z)\rangle \eta'(E) \\
			&\quad= \partial_E \Re \langle M(z)\rangle + 2\int_\mathbb{R} \frac{(E-x)\eta}{((x-E)^2+\eta^2)^2}\rho_T(x)\dif x\eta'(E)\ge \partial_E \Re \langle M(z)\rangle.
		\end{split}
	\end{equation*}
	In the last inequality we used \eqref{eq:eta'} to show that the second term is positive. Since
	\begin{equation*}
		\partial_E \Re\langle M(z)\rangle = \Re \frac{\langle M^2\rangle}{1-\langle M^2\rangle} = -1 +  \frac{1-\Re\langle M^2\rangle}{\vert 1-\langle M^2\rangle\vert^2}\ge -1,
	\end{equation*}
	the lhs.~of \eqref{eq:deriv} is lower bounded by  $-1$. Hence, we arrive at a contraction and hence Lemma \ref{lem:ray}~(i) holds. \\[2mm]
	\underline{Part (ii):} Now we prove the second part of Lemma \ref{lem:ray} concerning the bulk-restricted domains $\Omega_{\kappa,t}$. We aim to prove that there exists $t_*\in [0,T)$ with $T-t_*\sim 1$ such that $\Omega_{\kappa,t}$ has the ray property for all $t\in [t_*,T]$. By construction \eqref{eq:specdom_bulk} $\Omega_{\kappa,T}$ satisfies the ray property. As in the argument above assume that $\Omega_{\kappa,t}$ does not satisfy the ray property for some $t\in [0,T)$. However, unlike in the previous part of the proof we do not argue by contradiction, but rather prove that $T-t\gtrsim 1$.

	Similarly to \eqref{eq:deriv_aux}, we find $z_{1,T}$, $z_{2,T}$ on the boundary of $\Omega_{\kappa,T}\cap\HH$ such that \eqref{eq:deriv} holds, where we denoted $E_j:=\Re z_{j,T}$, $j\in[2]$. Moreover, by choosing the time $t$, for which the ray property of $\Omega_{\kappa,t}$ is violated, sufficiently close to $T$,
	one can find such $z_{1,T}$, $z_{2,T}$ 
	meeting the following additional condition: Either $E_1, E_2\in [b_r, (b_r+a_{r+1})/2]$ or $E_1, E_2\in [(b_r+a_{r+1})/2, a_{r+1}]$ for some $r\in [m-1]$, where we freely used the notations $a_r, b_r$ from Definition \ref{def:specdom}. Without loss of generality we may assume that the first of these two options holds and that $E_1<E_2$. Then \eqref{eq:deriv} reads 
	\begin{equation*}
		\frac{1}{E_2-E_1}\int_0^{E_2-E_1}\partial_x \Re\langle M_T(z_{1,T}+x+\ii x)\rangle\dif x = - \frac{1}{1-\ee^{T-t}}. 
	\end{equation*}
	Therefore, we have
	\begin{equation}
		\sup_{x\in [E_1,E_2]}\left\vert \frac{\langle M_T^2\rangle}{1-\langle M_T^2\rangle}\right\vert \gtrsim (T-t)^{-1},
		\label{eq:T-t_bound}
	\end{equation}
	where $M_T$ is evaluated at $z_{1,T}+x+\ii x$. We view \eqref{eq:T-t_bound} as a lower bound on $T-t$ and are hence left to show that the lhs. of \eqref{eq:T-t_bound} has an upper bound of order one. 
	
	For the numerator it holds that $\vert\langle M_T^2\rangle\vert\le 1$, while Proposition \ref{prop:stab} applied to the denominator gives that
	\begin{equation}
		\vert 1-\langle M_T^2(z)\rangle\vert\gtrsim \Im z + \rho_T^2(z), \quad z=z_{1,T}+x+\ii x,\, x\in [0,E_2-E_1].
		\label{eq:denom_bound}
	\end{equation}
	Recall that $b_r\in\mathbf{B}_\kappa$, i.e. $\rho_T(b_r)\ge \kappa$. Then there exists $c_0>0$ which depends only on $\kappa$ and $L$ such that for any $y\in [0,c_0]$ we have $\rho_T(b_r+y+\ii y)\ge\kappa/2$. This is a simple consequence of the differential inequality
	\begin{equation*}
		\partial_y\rho_T(b_r+y+\ii y)\lesssim \frac{1}{\vert 1-\langle M^2_T(b_r+y+\ii y)\rangle\vert}\lesssim \frac{1}{\rho_T^2(b_r+y+\ii y)},
	\end{equation*}
	where in the last step we again used Proposition \ref{prop:stab}. Take $x\in[0,E_2-E_1]$ and choose $y:=E_1-b_r+x$, which guarantees that $z:=b_r+y+\ii y=z_{1,T}+x+\ii x$. If $y\in [0,c_0]$, then $\rho_T(z)\ge\kappa/2$ and \eqref{eq:denom_bound} shows that $\vert 1-\langle M_T^2(z)\rangle\vert\gtrsim 1$. In the case $y>c_0$ we have $\Im z\gtrsim 1$ and derive the same conclusion $\vert 1-\langle M_T^2(z)\rangle\vert\gtrsim 1$ from \eqref{eq:denom_bound}. 
	
	This finishes the proof of Lemma \ref{lem:ray}.
\end{proof}

\end{document}